\documentclass[11pt,letterpaper]{amsart}
\usepackage[dvipsnames,usenames]{color}
\usepackage{amsmath}
\usepackage{amsfonts}
\usepackage{amssymb}
\usepackage{mathtools}
\usepackage{mathrsfs}
\usepackage{amsthm}
\usepackage{hyperref}
\usepackage{url}
\usepackage{comment}

\def \k {K\"ahler }
\def \ke {K\"ahler-Einstein }
\newcommand{\sm}{\setminus}

\newcommand{\ol}{\overline}
\newcommand{\Ol}{\overline}
\newcommand{\Om}{\Omega}
\newcommand{\om}{\omega}

\newcommand{\e}{\epsilon}
\newcommand{\p}{\partial}
\newcommand{\mc}{\mathcal}
\newcommand{\w}{\widetilde}
\newcommand{\al}{\alpha}
\newcommand{\be}{\beta}
\newcommand{\la}{\lambda}
\newcommand{\cal}{\mathcal}
\newtheorem{theorem}{Theorem}[section]
\newtheorem*{theorem*}{Theorem}
\newtheorem{corollary}[theorem]{Corollary}
\newtheorem{lemma}[theorem]{Lemma}
\newtheorem{proposition}[theorem]{Proposition}
\newtheorem{example}[theorem]{Example}
\newtheorem{remark}[theorem]{Remark}
\newtheorem{definition}[theorem]{Definition}
\newtheorem*{definition*}{Definition}

\numberwithin{equation}{section}

\def\e{\epsilon}
\def\wt{\widetilde}

\begin{document}

\title[]{A new Poincar\'e type rigidity phenomenon with applications}
%\author{ Ming Xiao
%\footnote{Supported in part by NSF grant DMS-1800549 and DMS-2045104}}

%\address{%
% Department of Mathematics}
%
%\email{luking@zjhu.edu.cn}

%\thanks{This work was}

%----------Author 2
%
%\author{ }

%\subjclass[2000]{Primary 47B38; 47G10; Secondary 32A36}
%\keywords{Toeplitz operator; Hankel operator; Fock space; essential
%norm}

\begin{abstract}
We discover a new Poincar\'e type phenomenon by establishing an optimal 
rigidity theorem for local CR mappings between circle bundles that are defined in a canonical way over (possibly reducible) bounded symmetric domains.
We prove such a local CR map, if nonconstant, must extend to a rational biholomorphism between the corresponding disk bundles.
The result includes as a special case the classical Poincar\'e--Tanaka--Alexander theorem. 
Among other applications, we show, for two irreducible bounded symmetric domains with rank at least two, a local CR diffeomorphism between (open connected pieces of) their anti-canonical circle bundles extends to a norm-preserving holomorphic isomorphism between  their anti-canonical bundles. The statement fails in the rank one case.
As another application, we construct, for any $n \geq 2,$ a countably infinite family of compact locally homogeneous strongly pseudoconvex CR hypersurfaces (in complex manifolds) of real dimension $2n+1$ with transverse symmetry such that they are all obstruction flat and Bergman logarithmically flat. Moreover, their local CR structures are mutually inequivalent. Such a family cannot exist in  dimension three by known results: A Bergman logarithmically flat  CR hypersurface must be spherical, and so is a compact obstruction flat CR hypersurface with transverse symmetry. 
\end{abstract}


%\date{\today}

\author{Ming Xiao}
\address{Department of Mathematics, University of California at San Diego, La Jolla, CA 92093, USA}
\email{{m3xiao@ucsd.edu}}

%\medskip

%\noindent
%{\bf Key Words}: Circle and disk bundles; bounded symmetric domains; CR mappings;  holomorphic mappings; rational biholomorphisms; line bundle isomorphisms.

\thanks{The author was supported in part by the NSF grants DMS-1800549 and DMS-2045104.}

\maketitle

\section{Introduction}

%\subsection{Background, main theorem and some applications}
Write $\mathbb{B}^m=\{z \in \mathbb{C}^m: \|z \|<1 \}$ for the $m-$dimensional complex unit ball, and $\partial \mathbb{B}^m$ for the unit sphere in $\mathbb{C}^m$. Poincar\'e \cite{Po} proved the following theorem for $n=1$, and Tanaka \cite{Ta}  extended it to higher dimensions (the theorem was later re-discovered  by Alexander \cite{Al}).

\smallskip

{\bf Theorem 1 } {\it Let $U$ be a connected open subset of $\mathbb{C}^{n+1}, n \geq 1,$ with $U \cap \partial \mathbb{B}^{n+1} \neq \emptyset.$ Let $F$ be a nonconstant holomorphic map from $U$ to $\mathbb{C}^{n+1}.$ If $F$ maps $U \cap \partial \mathbb{B}^{n+1}$ to $\partial \mathbb{B}^{n+1},$ then $F$ must extend to an automorphism of $\mathbb{B}^{n+1}$(in particular, $F$ is linear fractional).}
%and $F$ is isometric with respect to the standard hyperbolic metric of $\mathbb{B}^{n+1}$.}

\smallskip

This work is among the most celebrated classical theorems of several complex variables, and can be regarded as the CR geometric analogue of the classical Liouville's rigidity theorem on conformal maps between Euclidean spaces.  It is the starting point from which many far-reaching research directions stem. Since the work of Poincar\'e--Tanaka--Alexander, many deep results have been established in the investigation of extendability,  algebraicity, rigidity, and many other related problems of local holomorphic and CR maps. %between CR hypersurfaces. The investigation on these problems is also ultimately related to the study of proper holomorphic maps between complex domains as well.
 
It seems impossible to give a comprehensive survey on all these research topics here. %We just refer the interested readers to the papers   [We], [BR], [Be], [BER], [Fr1-2], [Hu1-2], [HM],  [Za], [KZ1-2], [M], [BX], [LM] and many references therein. The list here is by no mean to be complete. 
In this paper, we concentrate on one of the research directions: the study of {\it the Poincar\'e type phenomenon}. In CR geometry, the latter term refers to the occurrence when a local holomorphic or CR map between the boundaries (or some strata of the boundaries) of two domains extends to a global holomorphic map between the domains, and often the resulted global map possesses some rigidity (e.g., biholomorphism, isometry) properties. Since the work of Poincar\'e, Tanaka and Alexander, numerous Poincar\'e type phenomenons have been discovered, especially for the complex balls and bounded symmetric domains (the source and target domains can even be different dimensional).  See many deep results along this line in \cite{W2,TK1, TK2, Fr1, Fr2, Hu2, HJ2, KZ1, KZ2, MN, BH, BEH, MZ} and many references therein. For example,
Tumanov--Khenkin \cite{TK1, TK2} and Mok--Ng \cite{MN} generalized Theorem 1 to any irreducible bounded symmetric domain $\Om$ of higher rank, with $\partial \mathbb{B}^{n+1}$ replaced by the Shilov boundary and the smooth boundary of $\Om$, respectively.

In this article, we discover a new Poincar\'e type rigidity phenomenon which extends Theorem 1 via a different viewpoint from the above literature. To explain our ideas, we first observe $\mathbb{B}^{n+1}$ can be regarded as a disk bundle over the lower dimensional ball $\mathbb{B}^{n}.$ More precisely, let $L=\mathbb{B}^{n} \times \mathbb{C}$ be the trivial line bundle over $\mathbb{B}^{n}$ (recall a holomorphic line bundle over a contractible Stein manifold is always trivial). We equip $L$ with the Hermitian metric $h(z, \overline{z})=(1-\|z\|^2)^{-1}$ where $z \in \mathbb{B}^{n}.$ Then $(L, h)$ is a negative line bundle and the negative of its Chern class $-c_1(L, h)$ induces the standard hyperbolic metric on $\mathbb{B}^{n}$. Moreover, in this way, $\mathbb{B}^{n+1}$ and (an open dense subset of) $\partial \mathbb{B}^{n+1}$ become the associated disk bundle $D(L,h)$ and circle bundle $C(L,h),$  respectively. Theorem 1 can now be formulated as follows: {\it A nonconstant local holomorphic map sending an open piece of $C(L,h)$ to $C(L,h)$ must extend to an automorphism of the disk bundle $D(L,h)$.} We will extend this version of Theorem 1 to an optimal general result for circle and disk bundles over bounded symmetric domains (to be defined in Definition \ref{D2}). As will be shown later, these circle bundles  possess important and distinguished CR geometric features, and it is of particular interest to develop CR function theory on them.
%carry out further study on their CR structures. 
%They are important models for non-spherical  obstruction flat and Bergman logarithmically flat strongly pseudoconvex hypersurfaces (to be explained in). 

To formulate our results, we pause to recall the notion of bounded symmetric domains. A Hermitian symmetric space of noncompact type can be realized as a convex and circular bounded domain containing $0$ in some complex Euclidean space, via the Harish-Chandra realization, equipped with the (possibly normalized) Bergman metric. In this paper, the term bounded symmetric domains will always refer to such realizations of Hermitian symmetric spaces of noncompact type.  A bounded symmetric domain is called irreducible if it cannot be written as the product of two of such domains. %Irreducible bounded symmetric domains can be classified into Cartan's four types of classical domains and two exceptional domains. 
The unit ball and the polydisc are the most basic models for irreducible and reducible bounded symmetric domains, respectively. The rank of a bounded symmetric domain $\Om$  can be defined as the dimension of the maximal polydisc that can be totally geodesically embedded into $\Om.$ Recall $\Om$ has rank one if and only if $\Om$ is biholomorphic to the unit ball. We also remark that $\Om$ has smooth boundary if and only if $\Om$ has rank one.

We next recall the notion of generic norms. Let $\Om$ be  an irreducible bounded symmetric domain and denote by $K(z, \overline{z})$ its Bergman kernel. Write $V_{\Om}$ for the volume of $\Om$ in the Euclidean measure, and $\gamma$ for the genus of $\Om$ (which is a positive integer associated to $\Om$, cf. \cite{Lo}). Then %there is a (unique) Hermitian polynomial $N(z, \overline{z})$ satisfying 
$(V_{\Om} K(z, \Ol{z}))^{-\frac{1}{\gamma}}$ gives a Hermitian polynomial, denoted by $N(z, \overline{z})$. The polynomial $N$ is called the generic norm of $\Om$ (see \cite{M1}, \cite{Lo} for more details). 
%It turns out that the polynomial $N(z, w) \in \mathbb{C}[z, w]$ is irreducible. 
It  satisfies that $0 < N(z, \overline{z}) \leq 1$ in $\Om$, and $N(z, \overline{z})=1$ if and only if $z=0$. Moreover, $N(z, \overline{z})=0$ on the boundary $\partial\Om$. Writing $N(z, w) \in \mathbb{C}[z, w]$ for the complexification of $N(z, \overline{z}),$ we have $N(\cdot, 0)=N(0, \cdot) \equiv 1.$ %Writing $K \subseteq \mathrm{Aut}_0(\Om)$ for the isotropy subgroup of $\Om$ at $0 \in \Om,$  then $N$ is invariant under the action of $K: N(z, \overline{z})=N(\phi(z), \overline{\phi(z)})$ for every $\phi \in K.$     
For example, the unit ball $\mathbb{B}^{m}$ has genus $m+1.$ By inspecting its Bergman kernel $K_{\mathbb{B}^{m}}=\frac{m!}{\pi^m} \frac{1}{(1-|z|^2)^{m+1}}$, we easily see the generic norm of $\mathbb{B}^{m}$ is $N(z, \overline{z})=1-\|z\|^2$ with $z \in \mathbb{C}^{m}$.
To give more examples,  we write ${M} (p,q;\mathbb{C})$ for the set of $p \times q$ complex
matrices. Recall the type I classical domain $D^I_{p,q}$ (which is a special type of bounded symmetric domain) is a domain in ${M} (p,q;\mathbb{C})$ consisting of complex matrices $Z$ such that $I_p-Z \overline{Z}^t$ is positive definite. The generic norm of $D^I_{p,q}$
equals the determinant of $I_p-Z \overline{Z}^t$. See \cite{Lo} for more details and formulas for generic norms in general cases. 

While the usual notion of generic norms seems to be only defined for irreducible cases, we extend it to the general case of possibly reducible bounded symmetric domains. Denote by $\mathbb{N}$ the set of positive integers $\{1, 2, 3, \cdots\}.$

\begin{definition}\label{D1}
Let  $\Omega=\Omega_1 \times \cdots \times \Omega_l$  be a bounded symmetric domain in $\mathbb{C}^n,$ where each $\Omega_i, 1 \leq i \leq l,$ is an irreducible bounded symmetric domain in $\mathbb{C}^{n_i}_{z_i}$ with $n=\sum_{i=1}^l n_i.$ Denote by $N_i(z_i, \Ol{z_i}), 1 \leq i \leq l,$ the generic norm of $\Omega_i.$ Write $z=(z_1, \cdots, z_l)$ for the coordinates in $\mathbb{C}^n.$  A polynomial $\rho(z, \Ol{z})$ is called a generalized generic norm of $\Omega$ if $\rho$ can be written as follows:
$$\rho(z, \Ol{z})=\prod_{i=1}^l \big(N_i(z_i, \Ol{z_i})\big)^{k_i}~\text{with all}~k_i \in \mathbb{N}.$$
If in addition there is some $1 \leq i_0 \leq l$ such that $k_{i_0}=1,$ then we further say $\rho$ is a generic norm of $\Omega.$
\end{definition}

%A classical result of Poincare says that Alexdandar.

According to the above definition, the generic norm, as well as the generalized generic norm, of a given bounded symmetric domain $\Omega$ is not unique in general. However, when $\Omega$ is irreducible, its generic norm is unique and coincides with the usual notion.
Given a bounded symmetric domain $\Omega$ and a generalized  generic norm $\rho$ of $\Omega,$ let $L=\Omega \times \mathbb{C}$ be the
trivial line bundle over $\Omega$.  Equip $L$ with the Hermitian metric $h=\rho^{-1}.$ Then $(L, h)$ is a negative line bundle and $-c_1(L, h)$ induces a complete K\"ahler metric on $\Omega$ given by
\begin{equation}\label{eqnmeo}
\omega_{\Omega}=\sqrt{-1} \partial \overline{\partial} \log h=- \sqrt{-1} \partial \overline{\partial} \log \rho.
\end{equation}
Note when $\Omega$ is irreducible, $\omega_{\Omega}$ is proportional to the Bergman metric and thus K\"ahler-Einstein; when $\Omega$ is reducible, $\omega_{\Omega}$ is a product of complete \ke metrics. We also make the following definition.

\begin{definition}\label{D2}
Let $\Omega$ be a (possibly reducible) bounded symmetric domain in $\mathbb{C}^n_z$ and $\rho$ a generalized  generic norm of $\Omega.$
Denote by $L$ the trivial line bundle $\Omega \times \mathbb{C}$ over $\Omega$, equipped with the Hermitian metric $h=\rho^{-1}$. 
The disk and circle bundle of $(L, h)$, as defined in (\ref{eqndlh}) and (\ref{eqnclh}), respectively, will be called the disk and circle  bundle over $\Om$ associated to $\rho,$ respectively. They will be denoted by $D(\Om, \rho)$ and $\Sigma(\Om, \rho)$ throughout the paper.  %The  circle  bundles of $(L, h)$, as defined in , will be called the circle bundle over $\Om$ associated to $\rho,$ and will be denoted by $C(\Om, \rho)$ throughout the paper.
\begin{equation}\label{eqndlh}
D(\Om, \rho):=\{(z, \xi) \in \Omega \times \mathbb{C}: |\xi|^2 < \rho(z, \overline{z}) \} \subseteq \mathbb{C}^{n+1};
\end{equation}
\begin{equation}\label{eqnclh}
\Sigma(\Om, \rho):=\{(z, \xi) \in \Omega \times \mathbb{C}: |\xi|^2 = \rho(z, \overline{z}) \} \subseteq \mathbb{C}^{n+1}.
\end{equation}
We will call $r(z, \xi, \Ol{z}, \Ol{\xi}):=\rho(z, \overline{z})-|\xi|^2$ the generalized generic norm of $D(\Omega, \rho)$, and say it is the generic norm of $D(\Omega, \rho)$ if
$\rho$ is a generic norm of $\Om$.
\end{definition}

The disk bundle $D(\Om, \rho)$ is a bounded pseudoconvex domain in $\mathbb{C}^{n+1}$. It  is of particular geometric interest as a model for nonhomogeneous and noncompact complex manifolds which still possess a high degree of symmetry. It is sometimes referred to as the {\it Cartan--Hartogs domain}, especially when  $\Omega$ is irreducible. %The terminology Cartan-Hartogs domain was first introduced by Roos-Yin \cite{RY}. 
The interior geometry, e.g., the \k geometry, of Cartan--Hartogs domains has been extensively studied (cf. \cite{WYZR}, \cite{BT1}, \cite{MZe} and many references therein). Researchers also worked on the Bergman kernel and related questions of $D(\Om, \rho)$.
In particular, numerous authors contributed to finding explicit formulas for the Bergman kernels of Cartan--Hartogs domains and other closely related domains. See  Anh--Park \cite{AP} and many references therein.

On the other hand, much less attention has been paid to the CR geometry of the circle bundle $\Sigma(\Om, \rho)$, which will be a primary interest of this paper. %After introducing our main result(Theorem \ref{T1}), 
We first collect some simple facts about the CR structure of $\Sigma(\Om, \rho)$. It is clear that $\Sigma(\Om, \rho)$ is a real algebraic smooth CR hypersurface in $\mathbb{C}^{n+1}$. Since $(\Om, \omega_{\Omega})$ is a homogeneous \k manifold, it follows that $\Sigma(\Om, \rho)$ is a homogeneous CR hypersurface. Since $(L, h)$ is a negative Hermitian line bundle, by a well-known observation of Grauert, $\Sigma(\Om, \rho)$ is strongly pseudoconvex. 
Moreover, as $\omega_{\Omega}$ has constant scalar curvature, by the work of Webster \cite{W0} and Bryant \cite{Br} (cf. Proposition 1.12 in \cite{EXX2} and references therein), $\Sigma(\Om, \rho)$ is spherical if and only if $\Om$ is biholomorphic to the unit ball. Here recall a CR hypersurface is called {\em spherical} if it is locally CR diffeomorphic to an open piece of the sphere in a neighborhood of any point.

The CR hypersurface $\Sigma(\Om, \rho)$ has more critical CR geometric features.  After introducing our main theorem and its corollaries, we will recall the notions of obstruction flatness and Bergman logarithmic flatness for strongly pseudoconvex CR hypersurfaces. The importance of the two notions lies in their intimate connections with the boundary regularity of the Cheng-Yau potential of \ke metric, as well as the asymptotic behavior of the Bergman kernel, on a strongly pseudoconvex domain. We will prove that $\Sigma(\Om, \rho)$ is flat in both senses (cf. Proposition \ref{prpnflat}). The CR hypersurfaces $\Sigma(\Om, \rho)$ indeed provide the simplest possible models of obstruction flat and Bergman logarithmically flat CR hypersurfaces that are non-spherical. Therefore it is of particular interest to develop CR function theory on these hypersurfaces.

We stress that  $\Sigma(\Om, \rho)$ is only an open and dense proper subset of the topological boundary of $D(\Om, \rho)$ in $\mathbb{C}^{n+1}$. In general, $D(\Om, \rho)$ has non-smooth boundary points, and can also have weakly pseudoconvex boundary points in $\mathbb{C}^{n+1}$. On the other hand, although $\Sigma(\Om, \rho)$ is not even closed in $\mathbb{C}^{n+1}$, it has a compact realization as the boundary of some domain in a (non-Stein) complex manifold.  More precisely, let  $X=X_1 \times \cdots \times X_l$ be a compact quotient of the bounded symmetric domain $\Omega=\Om_1 \times \cdots \times \Om_l$ (where $\Om_j$ covers $X_j$, respectively). By taking tensor product of appropriate powers of the anti-canonical bundles over $X_j$'s (cf. the proof of Theorem \ref{corcis}), we can obtain a line bundle over $X$ whose circle bundle (which is clearly compact) locally CR equivalent to $\Sigma(\Om, \rho)$. To present our main theorem, we recall the notion of a norm-preserving isomorphism between Hermitian line bundles.

\begin{definition}
Let $(L_i, h_i)$ be a Hermitian line bundle over a complex manifold $M_i, 1 \leq i \leq 2.$ Write $\pi_i: L_i \rightarrow M_i$ for the natural fiber projection.  A map $G: L_1 \rightarrow L_2$ is called a norm-preserving holomorphic isomorphism between line bundles $L_1$ and $L_2$, if $G$ is a biholomorphism and satisfies the following conditions:

(1) $G$ preserves the norms in the sense that $|G(v)|_{h_2}=|v|_{h_1}$ for every $v \in L_1.$ Consequently,  the restriction of $G$ to the zero section $M_1 \cong M_1 \times \{0\}$, denoted by $g$, is a biholomorphism from $M_1$ to $M_2$.

(2) $G$ preserves the fiber structures. More precisely, $G(\pi_1^{-1}(z))=\pi_2^{-1}(g(z))$, and $G$ induces a linear isomorphism from $\pi_1^{-1}(z)$ to $\pi_2^{-1}(g(z)),$  for every $z \in M_1.$

\end{definition}

We now introduce our main theorem and then make a few remarks.

\begin{theorem}\label{T1}
Let $\Omega_1$ and $\Omega_2$ be two (possibly reducible) bounded symmetric domains in $\mathbb{C}^n$ with generic norms  $\rho_1$ and $\rho_2$, respectively. Denote by  $\om_{\Om_i}$ the \k metric $- \sqrt{-1} \partial \overline{\partial} \log \rho_i$ on $\Om_i$ for $1 \leq i \leq 2$. %Write $D_1=D(\Omega_1, \rho_1), D_2=D(\Omega_2, \rho_2)$ and $\Sigma_1=C(\Omega_1, \rho_1), \Sigma_2=C(\Omega_2, \rho_2)$, which are defined as in Definition \ref{D1}. 
%Let $\omega_{\Omega_1}, \omega_{\Omega_2}$ and $\omega_{D_1}, \omega_{D_2}$ be the naturally induced K\"ahler metrics on $\Omega_1, \Omega_2$ and $D_1, D_2$, respectively as above.	
Let $F$ be a nonconstant continuous CR map from an open connected piece of $\Sigma(\Omega_1, \rho_1) \subseteq \mathbb{C}^{n+1}$ to $\Sigma(\Omega_2, \rho_2) \subseteq \mathbb{C}^{n+1}$. Then the following statements hold.

(1) The map $F$ extends to a rational biholomorphism, still denoted by $F$, from $D(\Omega_1, \rho_1)$ to $D(\Omega_2, \rho_2)$. 
%Furthermore, $F$ is isometric in the sense that $F^*(\omega_{D_2})=\omega_{D_1}.$

(2) The domains $(\Omega_1, \omega_{\Omega_1})$ and $(\Omega_2, \omega_{\Omega_2})$ are holomorphically isometric. 
%In particular, $\Omega_1$ and $\Omega_2$ have the same rank.

(3) If in addition  $\mathrm{rank}(\Omega_1) \geq 2$ or $\mathrm{rank}(\Omega_2) \geq 2$,  %$\mathrm{rank}(\Omega_1) \geq 2$ or $\mathrm{rank}(\Omega_2) \geq 2,$ 
then $F$ extends to a norm-preserving rational holomorphic isomorphism between the line bundles  $L_1=\Omega_1 \times \mathbb{C}$  and $L_2=\Omega_2 \times \mathbb{C}$ with respect to the Hermitian metrics $h_1=\rho_1^{-1}$ and $h_2=\rho_2^{-1}.$ 
\end{theorem}

\begin{remark}\label{rmkform}
Indeed under the assumption of part (3), $F$ always takes the form: $F(z, \xi)=(\psi(z),\xi \phi(z))$, where $\psi$ is a biholomorphic map from $\Om_1$ to $\Om_2$, and $\phi$ is an everywhere nonzero holomorphic function on $\Om_1$ satisfying $\rho_2(\psi(z), \Ol{\psi(z)})=|\phi(z)|^2 \rho_1(z, \Ol{z}).$ In particular, the restriction of $F$ to $\Omega_1 \times \{0\},$ which is $\psi,$ induces a holomorphic isometry from $(\Omega_1, \omega_{\Omega_1})$ to $(\Omega_2, \omega_{\Omega_2})$. This can be seen from the proof of Theorem \ref{T1}.
\end{remark}

\begin{remark}\label{R1}
By a theorem of Pinchuk-Tsyganov \cite{PT}, a continuous CR map between two real analytic strongly pseudoconvex hypersurfaces $M$ and $M'$ in $\mathbb{C}^m, m \geq 2, $ extends holomorphically to a neighborhood of $M$. Therefore, by this result,  in Theorem \ref{T1} it is equivalent to assume that $F$ is a nonconstant holomorphic map from an open connected subset $U \subseteq \mathbb{C}^{n+1}$  to $\mathbb{C}^{n+1}$ satisfying $M:=U \cap \Sigma(\Omega_1, \rho_1) \neq \emptyset$ and  $F(M) \subseteq \Sigma(\Om_2, \rho_2).$ %Here $\Sigma_i=\Sigma(\Omega_i, \rho_i), 1 \leq i \leq 2. $ 
\end{remark}

\begin{remark}\label{R2}
Theorem \ref{T1} is optimal in the following sense. If either $\rho_1$ or $\rho_2$ is merely a generalized generic norm (instead of a generic norm. See Definition \ref{D1}), then the conclusion fails: the map $F$ can be irrational and has no holomorphic extension to $D(\Omega_1, \rho_1)$ (see Examples \ref{e22} and \ref{e23});  $F$ can also be rational while not biholomorphic on $D(\Omega_1, \rho_1)$ (see Example \ref{e21}).
\end{remark}

\begin{remark}
In the setting of Theorem \ref{T1}, one can apply the deep work of Webster \cite{W0} and Bryant \cite{Br} (cf. Proposition 1.12 in \cite{EXX2}) to see that $\mathrm{rank}(\Om_1)=1$ if and only $\mathrm{rank}(\Om_2)=1.$ We will, however, not use their results and give a more self-contained proof of Theorem \ref{T1}. %without using their results.
\end{remark}

%\begin{remark}\label{R3}
Note part (1) of Theorem \ref{T1} implies Theorem 1 as a special case of $\Omega_1=\Omega_2=\mathbb{B}^n.$ We stress that, in Theorem \ref{T1} we do not assume a priori that $\Omega_1=\Omega_2$, but instead derive as a conclusion in part (2) that they are biholomorphic. We also remark, unlike Theorem 1, one cannot expect the map $F$ in Theorem \ref{T1} to be linear fractional in general. %even if $\Omega_1=\Omega_2$ and $\rho_1=\rho_2$. 
See Examples \ref{e25} and \ref{e26}. 
%\end{remark}
Besides the unit ball $\mathbb{B}^n,$ the other basic model of bounded symmetric domains is the polydisc $\Delta^n=\{z \in \mathbb{C}^n: |z_i|< 1\}$. While Theorem \ref{T1} in the unit ball case gives precisely Theorem 1, it is interesting to check its application in the polydisc case. We record it below, which seems already a somewhat surprising statement.

\begin{corollary}\label{corpdk}
Let $n \geq 2$ and $(p_1, \cdots, p_n), (q_1, \cdots, q_n)$ be two $n$-tuples of positive integers. %where at least one $p_i$ and at least one $q_j$ equal to $1$. 
Write $z=(z_1,  \cdots, z_{n}) \in \mathbb{C}^{n}$ and set
$$\rho_1(z, \overline{z})=\prod_{i=1}^n (1-|z_i|^2)^{p_i}; \quad \rho_2(z, \overline{z})=\prod_{i=1}^n (1-|z_i|^2)^{q_i}.$$
Define  $\Sigma_i=\{(z, \xi) \in   \Delta^n \times \mathbb{C}: |\xi|^2=\rho_i(z, \overline{z}) \}$ for $1 \leq i \leq 2$. Assume at least one $p_i$ and at least one $q_i$ equal to $1$.  Let $F$ be a nonconstant continuous CR map from an open connected piece of $\Sigma_1$ to $\Sigma_2$.  Then after a permutation,  $(p_1, \cdots, p_n)$ and $(q_1, \cdots, q_n)$ are identical. Moreover, $F$ extends to a rational norm-preserving holomorphic isomorphism between the Hermitian line bundles $( \Delta^n \times \mathbb{C}, h_1)$ and $(\Delta^n \times \mathbb{C}, h_2)$, where $h_1=\rho_1^{-1}$ and $h_2=\rho_2^{-1}$.
%In particular, $F$ induces a biholomorphism from $D_1$ to $D_2$. %In addition, $F$ is isometric in the sense that $F^*(\omega_{D_2})=\omega_{D_1}.$
\end{corollary}

Like in Remark \ref{R2}, the assumption that at least one $p_i$ and at least one $q_i$ equal to $1$ cannot be dropped in Corollary \ref{corpdk}. See Examples \ref{e21} and \ref{e22}. We continue to discuss more applications of Theorem \ref{T1}.

\begin{corollary}\label{corglt}
Let $\Omega_1, \Omega_2$ be two (possibly reducible) bounded symmetric domains in $\mathbb{C}^n$ equipped with generic norms  $\rho_1$ and $\rho_2$, respectively. Assume $\mathrm{rank}(\Omega_1) \geq 2$ or $\mathrm{rank}(\Omega_2) \geq 2$.  Let $\lambda$ be a positive real number. Set  $M_i=\{(z, \xi) \in \mathbb{C} \times \Om_i: |\xi|^2=\rho_i^{\lambda}(z, \overline{z}) \}$ for $1 \leq i \leq 2$. Let $G$ be a nonconstant continuous CR map from an open connected piece $M$ of $M_1$ to $M_2$. Then $\Om_1$ and $\Om_2$ are biholomorphic, and  $G$ extends to a  norm-preserving holomorphic isomorphism between the Hermitian line bundles $(\Om_1 \times \mathbb{C}, h_1=\rho_1^{-\lambda})$ and $(\Om_2 \times \mathbb{C}, h_2=\rho_2^{-\lambda})$.
%where $h_1=\rho_1^{-\lambda}$ and $h_2=\rho_2^{-\lambda}$. %In particular, $F$ induces a biholomorphism from $D_1$ to $D_2$.
\end{corollary}

The latter conclusion in Corollary \ref{corglt} fails if $\Om_1$ and $\Om_2$ have rank one (see Example \ref{e23}). The following application asserts that different irreducible bounded symmetric domains can be distinguished by the local CR geometry of  their anti-canonical circle bundles.

\begin{corollary}\label{corac}
Let $\Omega_1, \Omega_2$ be irreducible bounded symmetric domains in $\mathbb{C}^n$ equipped with some complete K\"ahler-Einstein metrics. Let $(L_i, h_i)$ be the anti-canonical line bundles of $\Omega_i$ and let $M_i$ be the associated circle bundle, $1 \leq i \leq 2$. Then
 
(1) The domains $\Omega_1$ and $\Omega_2$ are biholomorphic if and only if there exists a CR diffeomorphism between some open pieces of   $M_1$ and $M_2$.

(2) Assume, in addition, that one of $\Omega_1$ and $\Omega_2$ has rank at least two.  Then any CR diffeomorphism $G$ between open connected pieces of  $M_1$ and $M_2$  extends a norm-preserving holomorphic isomorphism between $(L_1, h_1)$ and $(L_2, h_2).$ %In particular, $F$ maps $\Om_1 \times \{0\}$ biholomorphically to $\Om_2 \times \{0\}$.
\end{corollary}

Again the statement in part (2) of Corollary \ref{corac} fails when $\Om_1$ and $\Om_2$ have rank one. See Example \ref{e23} for a counterexample in the rank one case.
%The rigidity result of local CR maps between the boundaries usually has immediate applications in the study of proper maps. In our case 
Theorem \ref{T1} also immediately yields the following new result on proper holomorphic maps between Cartan--Hartogs domains. There has already been extensive study on the rigidity of proper holomorphic maps between different kinds of Cartan--Hartogs domains. See many results along this line in \cite{TW, BT2} and references therein. We emphasize that the approach to prove Theorem \ref{T1} is fundamentally different from the techniques in the study of proper maps:
A major difficulty in establishing Theorem \ref{T1} is to prove the single-valued extension of the local map $F$, while proper maps are already globally defined in the source domain.

%\begin{theorem}
%Let $\Omega_1, \Omega_2$ be two (possibly reducible) bounded symmetric domains in $\mathbb{C}^n$. Let $\rho_1$ and $\rho_2$ be generic norms of $\Omega_1$ and $\Omega_2,$ respectively. Let $D_1, D_2$, and $\omega_{\Omega_1}, \omega_{\Omega_2},$ and $\omega_{D_1}, \omega_{D_2}$ be as defined in Theorem \ref{T1}. Let $U$ be a connected open subset of $\mathbb{C}^{n+1}, n \geq 1,$ with $U \cap \partial D_1 \neq \emptyset.$ Let $F$ be a non-constant holomorphic map from $U$ to $\mathbb{C}^{n+1}.$ If $F$ maps $U \cap \partial D_1$ to $\partial D_2$, then the conclusions in (1)-(3) of Theorem \ref{T1} all hold.
%\end{theorem}

\begin{corollary}\label{corpr}
Let $\Omega_1, \Omega_2$ be two (possibly reducible) bounded symmetric domains in $\mathbb{C}^n$ equipped with generic norms  $\rho_1$ and $\rho_2$, respectively. Assume $\mathrm{rank}(\Omega_1) \geq 2$ or $\mathrm{rank}(\Omega_2) \geq 2$. Let $\lambda_1, \lambda_2$ be positive real numbers and
$$D_i=D(\Om_i, \rho_i^{\lambda_i}):=\{(z, \xi) \in \Omega_i \times \mathbb{C}: |\xi|^2 < \big(\rho_i(z, \overline{z})\big)^{\lambda_i} \}, \quad  1 \leq i \leq 2.$$
 Let $G$ be a holomorphic proper map from $D_1$ to $D_2$. Then $\Om_1$ and $\Om_2$ are biholomorphic, and $\lambda_2$ is an integer multiple of $\lambda_1,$ i.e., $\lambda_2=m\lambda_1$ for some $m \in \mathbb{N}.$ Moreover, $G$ can be written as $G(z, \xi)=(\psi(z),\xi^{m} \varphi(z))$, where $\psi$ is a biholomorphic map from $\Om_1$ to $\Om_2$, and $\varphi$ is a holomorphic function on $\Om_1$ satisfying $\rho_2(\psi(z), \Ol{\psi(z)})=|\varphi(z)|^{\frac{2}{\lambda_2}} \rho_1(z, \Ol{z}).$ 
 Consequently, $G$ is a biholomorphism from $D_1$ to $D_2$ if and only if $\lambda_1=\lambda_2$.
%, and the other conclusions in Theorem \ref{T1} hold as well.
\end{corollary}

%Like Theorem \ref{T1}, Corollary \ref{corpr} would also fail if $\rho_1$ and $\rho_2$ are generalized generic norms (cf. Example \ref{e21}). 
As mentioned earlier, Theorem \ref{T1} also has  applications in studying obstruction flatness and Bergman logarithmic flatness. 
For that, we first recall the two notions of flatness. Let $D$ be a smoothly bounded strongly pseudoconvex domain in $\mathbb{C}^m.$ 
The existence of a complete \ke metric on $D$ is governed by the following Dirichlet problem for Fefferman's complex Monge-Amp\`ere equation \cite{Fe2}:
\begin{equation}\label{Dirichlet problem}
	\begin{dcases}
		J(u):=(-1)^m \det \begin{pmatrix}
			u & u_{\ol{z_k}}\\
			u_{z_j} & u_{z_j\ol{z_k}} \\
		\end{pmatrix}=1 & \mbox{in } D\\
		u=0 & \mbox{on } \partial D
	\end{dcases}
\end{equation}
with $u>0$ in $D$. If $u$ is a solution of (\ref{Dirichlet problem}), then $-\log u$ is the \k potential of a complete \ke metric on $D$ with negative Ricci curvature. After the work of Fefferman \cite{Fe2},
%Let $u \in C^{\infty}(D)$ be the Cheng-Yau solution of $D$. That is, $u$ is the solution to the Fefferman's complex Monge-Amp\`ere equation $J(u)=1$ with the boundary condition $u=0$ on $\partial D$ (and thus $-\log u$ is the \k potential of a complete \ke metric on $\Omega$). 
Cheng and Yau \cite{CY} proved the existence and uniqueness of the solution to (\ref{Dirichlet problem}), nowadays called the Cheng--Yau solution on $D$. It follows from \cite{LM} that, in general, the Cheng--Yau solution $u$ can only be expected to possess a finite degree of boundary smoothness: $u\in C^{m+2-\varepsilon}(\Ol{D})$ for any $\varepsilon>0$. In the special case when $u \in C^{\infty}(\Ol{D}),$ we say $D$ has obstruction flat boundary. By the work of Graham \cite{Gr}, the obstruction flatness purely depends on the local CR geometry of the boundary $\partial D.$ Moreover, the obstruction flatness can be defined for any germ of strongly pseudoconvex CR hypersurface $M$ in terms of  %in terms of the regularity of the (formal) solution to 
a local Cauchy problem of $J(u)=1$ with certain Cauchy data along $M$ (cf. \cite{Gr} and \cite{EXX2} for more details on this topic). For convenience, we give an equivalent definition here.
We say a germ of strongly pseudoconvex CR hypersurface $M$ in $\mathbb{C}^m$ is obstruction flat if, for some (and therefore every, by \cite{LM} and \cite{Gr}) smoothly bounded strongly pseudoconvex domain $D$ with $M \subseteq \partial D$, the Cheng--Yau solution $u$ of $D$ extends $C^{\infty}$--smoothly across $M$.

Let $G=\{z \in \mathbb{C}^m: r(z, \Ol{z})>0\}$ be a smoothly bounded strongly pseudoconvex domain with a smooth defining function $r$. Fefferman \cite{Fe1} showed that the Bergman kernel $K_G$ of $G$ obeys that the following asymptotic expansion on $G$:
\begin{equation}\label{eqnbef}
K_G=\frac{\phi}{r^{m+1}}+\psi \log r.
\end{equation}
Here $\phi, \psi \in C^{\infty}(\Ol{G})$. %and $\phi$ is everywhere nonzero on $\partial G.$
We say $G$ has Bergman logarithmically flat boundary if $\psi$ vanishes to the infinite order along  $\partial G$.
By the localization property of the Bergman kernel (cf. \cite{Fe1, BS, En, HL}), Bergman logarithmic flatness only depends on the local CR geometry of the boundary, and we can define Bergman logarithmic flatness for any germ of strongly pseudoconvex CR hypersurface $M$ in $\mathbb{C}^m$.  We say $M$ is Bergman logarithmically flat if, for some (and therefore every, by the localization property of the Bergman kernel) smoothly bounded strongly pseudoconvex domain $G$ with $M \subseteq \partial G$, the coefficient function $\psi$ in the Fefferman expansion (\ref{eqnbef}) of $K_G$ vanishes to the infinite order along $M$. 

A CR hypersurface $\Sigma$ (in a complex manifold) is called obstruction flat (respectively, Bergman logarithmically flat) if it is so as a germ of CR hypersurface at every $p \in \Sigma$.  Both notions of flatness depend purely on the local CR structure, and are thus invariant under CR diffeomorphisms.
In particular, a spherical CR hypersurface must be flat in both senses. Note also that Bergman logarithmic flatness is stronger than obstruction flatness for $3$-dimensional hypersurfaces (cf. \cite{Gr}), while their relation in higher dimensions still seems unclear. The following proposition (to be proved in $\S$2.2) shows the circle bundle defined in Definition \ref{D2} is flat in both senses.

\begin{proposition}\label{prpnflat}
Let $\Omega$ be a (possibly reducible) bounded symmetric domain in $\mathbb{C}^n$ with a generalized generic norm $\rho$. Then
$\Sigma(\Om, \rho)$ is obstruction flat and Bergman logarithmically flat.
\end{proposition} 

An important question in CR geometry  asks to understand obstruction flat, as well as, Bergman logarithmically flat  CR hypersurfaces, especially the compact ones. In the real $3$-dimensional case, by the work of Graham, Burns and Boutet de Monvel (see references in \cite{EXX1}), a Bergman logarithmically flat CR hypersurface must be spherical. On the other hand, Graham \cite{Gr} showed that there are also plenty of (local) non-spherical CR hypersurfaces (of any dimension) that are obstruction flat. Unlike the local case, Ebenfelt \cite{E} proved a compact obstruction flat  $3$-dimensional CR hypersurface with transverse symmetry must be spherical (see \cite{CE} for more general results). Here recall that a CR hypersurface $M$ has transverse symmetry if there is a 1-parameter family of CR diffeomorphisms of $M$ such that its infinitesimal generator is transverse to the CR tangent space of $M$. When the dimension is at least $5$, there are examples of compact CR hypersurfaces with transverse symmetry that are Bergman logarithmically flat (respectively, obstruction flat) while not spherical. See \cite{EZ, ALZ} and \cite{Tk, EXX2} for such examples of the two notions of flatness, respectively. Nevertheless, for each notion of flatness, it remained unclear whether it was possible 
to construct two compact hypersurfaces that are transversally symmetric, flat, and non-spherical while possessing 
distinct local CR structures.

%it was not even clear whether one can construct two compact, transversally symmetric, flat and non-spherical hypersurfaces such that they have distinct local CR structures. (Recall both notions of flatness only depends on the local CR geometry of the hypersurfaces.)

Therefore the following questions remain open in $5$ and higher dimensions for each notion of flatness: How many locally CR inequivalent compact  flat  hypersurfaces with transverse symmetry are there? Furthermore, is it possible to classify them  in terms of the local CR structures?
%How large is the moduli space of compact Bergman logarithmically flat (respectively, obstruction flat)  CR hypersurfaces with transverse symmetry in terms of their local CR structures, and is it possible to classify such compact hypersurfaces via their local CR structures? 
Theorem \ref{corcis} demonstrates the existence of infinitely many such compact hypersurfaces with distinct local CR structures, 
suggesting the formidable challenge of achieving a comprehensive classification for these hypersurfaces.

%shows there are at least countably infinitely many such compact hypersurfaces with mutually distinct local CR structures, and thus it appears challenging to achieve a comprehensive classification of these hypersurfaces. 

In the statement of Theorem \ref{corcis}, we say a CR hypersurface $M$ is locally homogeneous, if
for any two points $p, q \in M$, there is an open piece of $M$ containing $p$ CR diffeomorphic to some open piece of $M$ containing $q$.

\begin{theorem}\label{corcis}
	Let $n \geq 2.$ There exists a countably infinite family $\mathcal{F}$ of compact real analytic CR hypersurfaces (realized as circle bundles over complex manifolds) of real dimension $2n+1$ such that the following hold.
	
	(1) Every CR hypersurface $M \in \mathcal{F}$ is obstruction flat and Bergman logarithmically flat. Moreover, every  $M \in \mathcal{F}$ is locally homogeneous, and has transverse symmetry.
	
	(2) For every pair of (distinct) CR hypersurfaces $M_1, M_2 \in \mathcal{F}$,  any open pieces $U \subseteq M_1$ and $V \subseteq M_2$ are not CR diffeomorphic. 
	
	%(3). Any open piece of each CR hypersurface $M \in \mathcal{F}$ cannot be embedded into the unit sphere of any dimension via a smooth CR map.

\end{theorem}

%At the end of this subsection \S 1.1, 
Finally we highlight some main difficulties and key ideas in the proof of Theorem \ref{T1}.
First for a bounded symmetric domain $\Om$ with generic norm $\rho$, in general the boundary of  $D(\Omega, \rho)$ is non-smooth, and may have weakly pseudoconvex points as well. As a result, typical analytic continuation theorems for local CR maps in CR geometry (cf. \cite{HJ1} and many references therein) cannot be applied.
Although by the work of Webster \cite{W1}, the local map $F$ in Theorem \ref{T1} is algebraic, a priori (the holomorphic extension of) $F$ is not necessarily single-valued.  The first key step in our proof is to  invoke ideas and techniques from CR geometry %cf. \cite{HZ, HLX, HX}) 
to prove $F$ is rational via a monodromy argument.
%As we will see,  the main difficulty in this approach is to establish certain ``nondegeneracy" condition of the map $F$ along the Segre varieties associated to the source hypersurface $C(\Om_1, \rho_1)$. To accomplish the latter, 
In this part,  the geometry of bounded symmetric domains and the algebraic property of the generic norms are fundamentally used. Secondly, using the rationality, we study the geometric property of $F$ in terms of
%We introduce a canonical isometric embedding of bounded symmetric domains $\Om$ into the so-called {\it indefinite hyperbolic spaces} (to be explained in $\S$3). The embedding is unique up to a special class of linear transformations. 
%Moreover, under the embedding, any automorphism of the source domain extends to an automorphism of the indefinite hyperbolic space. This aforementioned embedding can be regarded as the noncompact dual of the well-known
%Nakagawa-Takagi embedding theorem, which assets every irreducible Hermitian symmetric space of compact type can be isometrically embedded into projective spaces in a canonical way. This newly introduced embedding will play an essential role in our proof, and we expect it to be useful in the future study of mapping problems related to bounded symmetric domains.
the \k geometry of the disk bundles $D_i:=D(\Om_i, \rho_i), 1 \leq i \leq 2$. For that, letting $\rho$ and $r,$ as well as $D(\Om, \rho)$ be as in Definition \ref{D2}, we define a K\"ahler metric $\omega_D$ on $D:=D(\Om, \rho)$ that is naturally induced from the generic norms: %More precisely, the metric is given by the following, where $r$ and $\rho$ are as in Definition \ref{D2}:
\begin{equation}\label{eqnmetric}
	\omega_D=-\sqrt{-1} \partial \overline{\partial} \log r(z, \xi, \Ol{z}, \Ol{\xi})=  -\sqrt{-1} \partial \overline{\partial} \log \big(\rho(z, \overline{z})-|\xi|^2 \big).
\end{equation}
%Here $r$ and $\rho$ are as in Definition \ref{D2}.
%This also equals to $\pi^{*}(\omega_{\Omega}) -\sqrt{-1} \partial \overline{\partial} \log (1-\|w\|_h^2),$ with $\pi$ the canonical projection from $L$ to $\Omega$ (Is this formula needed?). 
This metric turns out to be complete on the disk bundle $D$ (see Proposition \ref{prpnmetric}), and it is canonical in the sense that it is invariant under automorphisms of $D$. Note when $\Omega$ is the unit ball $\mathbb{B}^n$ and $\rho$ is the generic norm, $\omega_D$ becomes the standard hyperbolic metric on $D=\mathbb{B}^{n+1}$ and is also proportional to the Bergman metric of $\mathbb{B}^{n+1}$.   We emphasize that, however, in general $\omega_D$ is neither a multiple of the Bergman metric nor the complete K\"ahler-Einstein metric on $D$. When $\Om$ is irreducible, the \k geometric properties of $\om_D$ are extensively studied by many researchers, cf. \cite{BT1, Ze}, etc. It seems, however, we first use it to study the rigidity of holomorphic  mappings. A key geometric feature we will show about $F$ is that it preserves the canonical metric $\om_D$ we just defined: $F^*(\om_{D_2})=\om_{D_1}$. This will be important for the study of the extension and the structure of $F$. Note from the conclusion of Theorem \ref{T1}, $F$ extends to a biholomorphic map from $D_1$ to $D_2$, and thus also preserves other canonical metrics, such as the complete \ke metric and the Bergman metric. However, if we use any of the latter two metrics instead of $\om_D$,
the proof does not seem to work.

The paper is organized as follows. To prepare for the proof of the main theorem, we discuss in $\S$2  some interior and boundary geometry of the disk bundles introduced in Definition \ref{D2}, as well as the geometry of Segre varieties of the circle bundles. %results in Section 3 will play an important role in the proof of Theorem \ref{T1}. 
Then in $\S$3 and $\S$4, we prove Theorem \ref{T1} and the corollaries, respectively.

\medskip

{\bf Acknowledgment.} The author thanks Peter Ebenfelt, Xiaojun Huang and Hang Xu for helpful comments.

\section{Some examples and preparations}
%{Preparations, part I: \k and boundary geometry of $D(\Om, \rho)$, and its automorphism group}
In $\S$2.1, we first give some examples to justify the remarks we have made in $\S$1. Next as preparations for the proof of Theorem \ref{T1},  we study the boundary and \k geometry of the disk bundle $D(\Om, \rho)$ in $\S$2.2 and $\S$2.3,  respectively. Then we compute the automorphism group of $D(\Om, \rho)$ in $\S$2.4. We finally discuss in $\S$2.5 some algebraic properties of the generic norms and geometric properties of Segre varieties of the circle bundle $\Sigma(\Om, \rho)$.

\subsection{Some examples}

In this subsection, we include some examples to support the remarks in $\S$1. 
\begin{example}\label{e21}
Let $\Omega_1=\Omega_2=\Omega$ be a bounded symmetric domain and $\rho$ be a generic norm of $\Omega.$ Let $\rho_1=\rho$ and $\rho_2=\rho^2$ (and thus $\rho_2$ is not a generic norm of $\Omega$).  As in Theorem \ref{T1}, we write $D_i:=D(\Om_i, \rho_i)=\{(z, \xi) \in \Omega \times \mathbb{C}: |\xi|^2 < \rho_i(z,\overline{z}) \}$ and $\Sigma_i:=\Sigma(\Om_i, \rho_i)=\{(z, \xi) \in \Omega \times \mathbb{C}: |\xi|^2 = \rho_i(z,\overline{z})  \}$ for $1 \leq i \leq 2$.  Define a holomorphic polynomial map $F$ from $\Omega \times \mathbb{C}$ to $\Omega \times \mathbb{C}$ by $F(z, \xi)=(z, \xi^2).$ It is clear that $F$ maps $\Sigma_1$ to $\Sigma_2$ and indeed $F$ gives a proper holomorphic map from $D_1$ to $D_2$. However, $F$ is not biholomorphic from $D_1$ to $D_2$.
Therefore the conclusions in part (1) and (3) of Theorem \ref{T1} fail in this case.  In addition, with $\om_{\Om_i}=- \sqrt{-1} \partial \overline{\partial} \log \rho_i, 1 \leq i \leq 2,$ we have $\omega_{\Omega_2}=2\omega_{\Omega_1}.$ As a result,
	%$\omega_{\Omega_1}=-\sqrt{-1} \partial \overline{\partial} \log \rho_1$ and $\omega_{\Omega_2}=-\sqrt{-1} \partial \overline{\partial} \log \rho_2$, 
the conclusions in part (2)  of Theorem \ref{T1} fails as well (because $\omega_{\Omega_2}$ and $\omega_{\Omega_1}$ have different scalar curvatures).
\end{example}

\begin{example}\label{e22}
Let $\Omega_1=\Omega_2=\Omega$ be a bounded symmetric domain and $\rho$ be a generic domain of $\Omega.$ Let $\rho_1=\rho^2$ and $\rho_2=\rho$. As in Theorem \ref{T1}, we write $D_i:=D(\Om_i, \rho_i)=\{(z, \xi) \in \Omega \times \mathbb{C}: |\xi|^2 < \rho_i(z,\overline{z}) \}$ and $\Sigma_i:=\Sigma(\Om_i, \rho_i)=\{(z, \xi) \in \Omega \times \mathbb{C}: |\xi|^2 = \rho_i(z,\overline{z})  \}$ for $1 \leq i \leq 2$.  Define a local map $F$ by  $F(z, \xi)=(z, \sqrt{\xi}).$ Here we choose a local branch of $\sqrt{\xi}$ so that $F$ is a well-defined holomorphic map on an open subset $U \subseteq \Omega \times \mathbb{C}$ such that $U \cap \Sigma_1 \neq \emptyset.$ It is clear that $F$ maps $U \cap \Sigma_1$ to $\Sigma_2.$ However, $F$ does not extend holomorphically to any neighborhood of $(0,0)$, and in particular, does not extend holomorphically to $D_1$.
	
	%It is clear that $F$ maps $\Sigma_1$ to $\Sigma_2$ and maps $D_1$ to $D_2$. However, $F$ is not biholomorphic from $D_1$ to $D_2$.
	%Therefore part (1) of Theorem \ref{T1} fails.  In addition, since  $\omega_{\Omega_1}=-\sqrt{-1} \partial \overline{\partial} \log \rho_1$ and $\omega_{\Omega_2}=-\sqrt{-1} \partial \overline{\partial} \log \rho_2$, one can see parts (2) and (3) of Theorem \ref{T1} fail as well.
\end{example}

We also give some more concrete examples.

\begin{example}\label{e23}
Let $\Omega_1=\Omega_2=\mathbb{B}^n$ be the unit ball in $\mathbb{C}^n$ and $\rho_1=\rho_2=(1-\|z\|^2)^{\lambda}$ with $z \in \mathbb{C}^n$ for some positive real number $\lambda$. 
%Let $D_1, D_2$ and $\Sigma_1, \Sigma_2$ be as defined in Theorem \ref{T1}.
As in Theorem \ref{T1}, we write $D_i:=D(\Om_i, \rho_i)=\{(z, \xi) \in \mathbb{B}^n \times \mathbb{C}: |\xi|^2 < (1-\|z\|^2)^{\lambda} \}$ and $\Sigma_i:=\Sigma(\Om_i, \rho_i)=\{(z, \xi) \in \mathbb{B}^n \times \mathbb{C}: |\xi|^2 = (1-\|z\|^2)^{\lambda}  \}$ for $1 \leq i \leq 2$.  
Note when $\lambda =n+1,$ (up to a constant scaling of the $\xi-$variable) $\Sigma_1$ and $\Sigma_2$ are the anti-canonical circle bundle of $\mathbb{B}^n$ , and $D_1$ and $D_2$ are the anti-canonical disc bundle. With $z=(z_1, \cdots, z_n),$ we define a local map $F$ by  $F(z, \xi)=(z_1, \cdots, z_{n-1}, \xi^{\frac{1}{\lambda}}, z^{\lambda}).$ Here we choose some local branches of $\xi^{\frac{1}{\lambda}}$ and $z^{\lambda}$ so that $F$ is a well-defined holomorphic map on some small open subset $U \subseteq \Om_1 \times \mathbb{C}$ with $U \cap \Sigma_1 \neq \emptyset.$ It is clear that $F$ maps $U \cap \Sigma_1$ to $\Sigma_2,$ while $F$ does not extend holomorphically to $D_1$ for $\lambda \neq 1$ (in particular for $\lambda =n+1$).
\end{example}

\begin{comment}
\begin{example}\label{e24}
	Let $\Omega=\Delta^2 \subseteq \mathbb{C}^2_{z}= \mathbb{C}_{z_1} \times \mathbb{C}_{z_2}$ and $\rho=(1-|z_1|^2)(1-|z_2|^2)^2$. Then $\rho$ is a generic norm of $\Delta^2.$ Let $D=D(\Omega, \rho)$ and $\Sigma=C(\Omega, \rho)$ be the disk and circle bundles as defined (\ref{eqndlh}) and (\ref{eqnclh}), respectively. Fix $a \in \Delta.$ Set 
	$$F=(\frac{a-z_1}{1-\overline{a}z_1}, z_2, \frac{(1-|a|^2)\xi}{(1-\overline{a}z_1)^2}).$$
	One can verify that $F$ is an automorphism of $D$, and maps $\Sigma$ to itself. However, $F$ is not linear fractional if $a \neq 0$.
\end{example}
\end{comment}

\begin{example}\label{e25}
	Let $\Omega=\Delta^2 \subseteq \mathbb{C}^2_{z}= \mathbb{C}_{z_1} \times \mathbb{C}_{z_2}$ and $\rho=(1-|z_1|^2)(1-|z_2|^2)$. Then $\rho$ is a generic norm of $\Delta^2.$ Let $D=D(\Omega, \rho)$ and $\Sigma=\Sigma(\Omega, \rho)$ be the disk and circle bundles as defined in (\ref{eqndlh}) and (\ref{eqnclh}), respectively. Fix $a, b \in \Delta.$ Set 
	$$F=(\frac{a-z_1}{1-\overline{a}z_1}, \frac{b-z_2}{1-\overline{b}z_2}, \frac{(1-|a|^2)^{\frac{1}{2}}(1-|b|^2)^{\frac{1}{2}}\xi}{(1-\overline{a}z_1)(1-\overline{b}z_2)}).$$
	One can verify that $F$ is an automorphism of $D$ and maps $\Sigma$ to itself. Note $F$ is not linear fractional if $ab \neq 0$.
	
\end{example}

\begin{example}\label{e26}
	Let $\Omega$ be an irreducible bounded symmetric domain of rank at least two, with generic norm $\rho$. Let $D=D(\Omega, \rho)$ be the associated disk bundle as in (\ref{eqndlh}).  In general, the automorphism of $\Omega$ is rational, but not linear fractional. Fix such an automorphism $\phi$. We will see in $\S$2.4 that $\phi$ induces an automorphism of $D$ that is not linear fractional. %(see Proposition \ref{prpnspauto}).
\end{example}

\subsection{Boundary structures of $D(\Om, \rho)$}

Let $\Om=\Om_1 \times \cdots \times \Om_l \subseteq \mathbb{C}^n_{z}$ be a bounded symmetric domain, where each $\Om_i, 1 \leq i \leq l,$ is an irreducible bounded symmetric domain in $\mathbb{C}^{n_i}_{z_i}$. Let $\rho$ be 
a generalized generic norm of $\Om$.  Write $D$ and $\Sigma$ for $D(\Om, \rho)$ and $\Sigma(\Om, \rho)$, respectively, where the latters are defined as  in Definition \ref{D2}.
It is well-known that, by Grauert's observation, $\Sigma$ is strongly pseudoconvex. We next briefly discuss the boundary geometry of $D$ at points in $\partial D \setminus \Sigma=\partial \Om \times \{0\}$.  We note if $\mathrm{rank}(\Om) \geq 2,$ then for a generic point $\wt{z} \in \partial \Om,$  there is a nontrivial complex variety contained in $\partial \Om$ that passes through $\wt{z}$. Consequently, a point in $\partial D \setminus \Sigma$ is either a nonsmooth boundary point, or a weakly pseudoconvex boundary point of $D$.

We next assume $\rho$ is a generic norm of $\Om$. Write
$\rho(z, \ol{z})=\prod_{i=1}^l (N_i(z_i, \ol{z_i}))^{k_i}$, where each $N_i$ is the generic norm of $\Om_i$ and each  $k_i \in \mathbb{N}.$ By Definition \ref{D2}, there exists some $1 \leq i^* \leq l$ such that $k_{i^*}=1.$  Write, for each $i$, $\mathrm{Reg}(\p \Om_{i})$ for the set of smooth boundary point of $\Om_i$, which is an open dense subset of $\partial \Om_i$(cf. \cite{MN}). Note the gradient $d N_i$ is everywhere nonzero on $\mathrm{Reg}(\p \Om_{i}),$ and $\Om_i$ is given by  $\{N_i(z_i, \ol{z_i})>0\}$ near every point in $\mathrm{Reg}(\p \Om_{i})$ (see Lemma 2 in Mok  \cite{M3}). Define the following subsets of $\partial \Om$ and $\partial D,$ respectively:
$$\mathcal{T}_{i^*}:=\{(z_1, \cdots, z_n) \in \partial \Om: z_{i^*} \in  \mathrm{Reg}(\p \Om_{i^*})~\text{and}~z_i \in \Om_i~\text{for}~i \neq i^*\};$$
$$\mathcal{S}_{i^*}:=\{(z_1, \cdots, z_n, 0) \in \partial D: (z_1, \cdots, z_n) \in \mathcal{T}_{i^*} \} \subseteq \partial D \setminus \Sigma.$$

%z_{i^*} \in  \mathrm{Reg}(\p \Om_{i^*})~\text{and}~z_i \in \Om_i~\text{for}~i \neq i^*\}.$$

We note  every point in $\mathcal{T}_{i^*}$ is a smooth boundary point of $\Om.$ Moreover, recall $\rho=\prod_{i=1}^l N_i^{k_i}$ with $k_{i^*}=1;$ $N_{i^*}=0, d N_{i^*} \neq 0$ on $\mathrm{Reg}(\p \Om_{i^*})$; and $N_i>0$ in $\Om_i$ for all $1 \leq i \leq l$. It follows that the gradient $d \rho$ is everywhere nonzero on $\mathcal{T}_{i^*},$ and $\Om$ is given by  $\{\rho(z, \ol{z})>0\}$ near every $z^* \in \mathcal{T}_{i^*}$.     We next observe:

\begin{lemma}\label{lmstom}
Let $\mathrm{rank}(\Om) \geq 2$ (and $\rho$ be a generic norm of $\Om$ as above). Then every point $p$ in $\mc{S}_{i^*}$ is a smooth weakly pseudoconvex boundary point of $D,$ and $\partial D$ is minimal at $p$.
\end{lemma}

\begin{proof}
Fix $p=(z^*, 0) \in \mathcal{S}_{i^*} \subseteq \partial D.$ As $\rho(z, \ol{z}) \leq 0$ for $z \not \in \Om$ near $z^*$, we have $D$ is locally defined by $r(z, \xi, \ol{z}, \ol{\xi}):=|\xi|^2-\rho(z, \Ol{z})<0$ near $p$. Furthermore, $\frac{\partial r}{\partial z_i}(p)=-\frac{\partial \rho}{\partial z_i}(z^*).$ Since the gradient of $\rho$ is nonzero at $z^*$, so is the gradient of $r$ at $p$. Hence $r$ is a smooth defining of $D$ at $p$, and  $p$ is  a smooth boundary point of $D.$ The weakly pseudoconvexity of $\partial D$ at $p$ follows from the discussion at the beginning of $\S$2.2. Suppose $\partial D$ is not minimal at $p$. Then there is a germ of complex hypersurface $X$ of $\mathbb{C}^{n+1}$ contained in $\partial D$ and passes through $p$. But $\partial D \setminus \Sigma$ is of real codimension $3$ in $\mathbb{C}^{n+1}$. Hence $X$ must intersect $\Sigma.$ This is impossible as $\Sigma$ is strongly pseudoconvex. Therefore 
$\partial D$ must be minimal at $p$.
\end{proof}

To conclude this subsection, we give a proof of Proposition \ref{prpnflat}.

\smallskip

{\bf Proof of Proposition \ref{prpnflat}.} 
First note the metric $\om_{\Om}$ is a product of \ke metrics, and thus has constant Ricci eigenvalues. Here the latter refers to the eigenvalues of the Ricci tensor with respect to the metric (see \cite{EXX2} for more details). It then follows from Theorem 1.1 in \cite{EXX2} that $\Sigma:=\Sigma(\Om, \rho)$ is obstruction flat.
To prove the Bergman logarithmic flatness, we first note by Theorem 2.5 in Ahn-Park \cite{AP}, the Bergman kernel $K$ of $D:=D(\Om, \rho)$ is rational. Now fix any point $p \in \Sigma$ (which is a strongly pseudoconvex point).
We can find  a smoothly bounded strongly pseudoconvex domain $D' \subseteq D$ such that $D' \cap O=D \cap O$ and $M_0:=\partial D' \cap O=\partial D \cap O  \subseteq \Sigma$ for some small ball $O$ in $\mathbb{C}^{n+1}$ centered at $p$.

Write $K'$ for the Bergman kernel of $D'$, and recall $D$ is pseudoconvex.  Then by the localization of the Bergman kernel on pseudoconvex domains at a strongly pseudoconvex boundary point (cf. Theorem 4.2 in \cite{En}), there is a smooth function $\Phi$ in a neighborhood of $D' \cup M_0$ such that 
%\begin{equation}\label{eqnphi}
$K=K'+\Phi~\text{on}~D'.$
%\end{equation}
On the other hand, for any defining function $\rho$ of $D'$ with $D'=\{\rho>0\}$, the Bergman kernel $K'$ of $D'$ has the  Fefferman expansion (see \cite{Fe1}) for some
functions $\phi, \psi \in C^{\infty}(\ol{D'}):$ 
\begin{equation}
K'=\frac{\phi}{\rho^{n+1}}+ \psi \log \rho \quad \text{on}~  D'.
\end{equation}
Consequently, the Bergman kernel $K$ of $D$ has the following expansion:
\begin{equation}\label{eqnkz}
	K=\frac{\hat{\phi}}{\rho^{n+1}}+ \psi \log \rho \quad \text{on}~  D'.
\end{equation}
Here $\hat{\phi}$ is some smooth function in a neighborhood of $D' \cup M_0$. Then using the rationality of $K$ and applying the argument in \cite{EXX1} (see Step 5 in $\S 5$ there), we see $\psi$ vanishes to the infinite order along $M_0$. As $p \in \Sigma$ is arbitrary,  $\Sigma$ is thus Bergman logarithmically flat. \qed

\subsection{A canonical \k metric on $D$}

Let $\Om=\Om_1 \times \cdots \times \Om_l$ be a (possibly reducible) bounded symmetric domain in $\mathbb{C}^{n}$, where each $\Om_i, 1 \leq i \leq l,$ is an irreducible bounded symmetric domain in $\mathbb{C}_{z_i}^{n_i}$. Write $z=(z_1, \cdots, z_l)$ for the coordinates in $\mathbb{C}^n.$ Let 
$\rho(z, \Ol{z})=\prod_{i=1}^l \big(N_i(z_i, \Ol{z_i})\big)^{\la_i}$, where each $N_i$ is the generic norm of $\Om_i$ and each  $\la_i$ is a positive real number (not necessarily integers). Let $D=D(\Om, \rho) \subseteq \mathbb{C}_{(z, \xi)}^{n+1}, \Sigma=\Sigma(\Om, \rho) \subseteq \mathbb{C}_{(z, \xi)}^{n+1}$ be as defined in (\ref{eqndlh}) and (\ref{eqnclh}), respectively.
In this subsection, we discuss a metric defined on $D$ given by $\om_D:=-\sqrt{-1} \partial \overline{\partial} \log \big(\rho(z, \overline{z})- |\xi|^2 \big)$. %We will show $\om_D$ is a canonical complete \k metric on $D$.

%The Hermitian form $-\sqrt{-1} \partial \overline{\partial} \log \big(p(z, \overline{z})- |w|^2 q(z, \overline{z}) \big)$ is positive definite on $\Omega$ if and only 
%$-\sqrt{-1} \partial \overline{\partial} \log p(z, \overline{z})$ is positive definite on $D$. The Hermitian form $-\sqrt{-1} \partial \overline{\partial} \log \big(p(z, \overline{z})- |w|^2 q(z, \overline{z}) \big)$ induces a complete metric on $\Omega$ if and only 
%$-\sqrt{-1} \partial \overline{\partial} \log p(z, \overline{z})$ induces a complete metric on $D$. See Proposition \ref{prpnmetric}.

%%%%%%%%%%%%%%%%%%%%%%%%%%%

\begin{proposition}\label{prpnmetric}
It holds that $\om_D$ is a complete \k metric on $D$. Moreover, it is a canonical metric of $D$ in the sense that $\Phi^*(\om_D)=\om_D$ for every automorphism $\Phi$ of $D$.
%The following two statements hold: \\
	
%(a). The \k form $\hat{w}:=-\sqrt{-1} \partial \overline{\partial} \log \big(p(z, \overline{z})- |w|^2 q(z, \overline{z}) \big)$ is positive definite (respectively, semi-positive definite) on $\Omega$ if and only if
%$-\sqrt{-1} \partial \overline{\partial} \log p(z, \overline{z})$ is positive definite (respectively, semi-positive definite) on $D$. \\
	
%(b). The form $\hat{\omega}$ gives a complete metric on $\Omega$ if and only if $-\sqrt{-1} \partial \overline{\partial} \log p(z, \overline{z})$ gives a complete metric on $D$ .
\end{proposition}

\begin{proof}
%We first establish the following lemma.
%\begin{lemma}
%Write $\psi=\frac{1}{\phi}$ and $H:=|w|^2 \psi-1$. Then $-\log(-H)-H$ is plurisubharmonic on $\Omega.$ Moreover, $H$ is strictly plurisubharmonic on $(\mathbb{C} \setminus \{w=0\}) \times D,$ and thus plurisubharmonic on $\mathbb{C} \times D$ (both $-\sqrt{-1} \partial \overline{\partial} \log (H)$ and $-\sqrt{-1} \partial \overline{\partial} H$ are degenerate when $w=0.$).
%\end{lemma}

We first verify $\om_D$ is a (positive definite) \k form on $D$. Write $\psi:=|\xi|^2 \rho^{-1}-1$. Since $(L:=\Om \times \mathbb{C}, \rho^{-1})$ is a negative line bundle, a well-known fact (usually referred as Grauert's observation) asserts that  $\psi$ is plurisubharmonic on $L$. Moreover, $\psi$ is strictly plurisubharmonic on $\{(z, \xi) \in \Om \times \mathbb{C}: \xi \neq 0\}.$ Note $f(x)=-\log (-x)$ is an increasing convex function on $(-1, 0).$ Consequently, $-\log (-\psi)$ is strictly plurisubharmonic on $D^*=\{(z, \xi) \in D: \xi \neq 0\},$ and plurisubharmonic on $D$. Writing $r=\rho(z, \overline{z})-|\xi|^2$, we have

\begin{equation}\label{eqnhom}
\om_D=-\sqrt{-1} \partial \overline{\partial} \log r =-\sqrt{-1} \partial \overline{\partial} \log \rho -\sqrt{-1} \partial \overline{\partial} \log  (-\psi).
\end{equation}
Since $-\sqrt{-1} \partial \overline{\partial} \log \rho$ and  $-\sqrt{-1} \partial \overline{\partial} \log (-\psi)$  are both semi-positive definite as Hermitian forms in $D$, we have:
\begin{equation}\label{eqnine}
\om_D \geq -\sqrt{-1} \partial \overline{\partial} \log  (-\psi); \quad \om_D \geq -\sqrt{-1} \partial \overline{\partial} \log  \rho.
\end{equation}
It then follows immediately that $\om_D$ is a positive definite form on $D^*$. To show $\om_D$ is positive definite on $\Om \times \{0\}$,
we identify $\xi$ with $z_{n+1}$ and write the coordinates for $\Omega$ as $Z=(z, \xi)=(z, z_{n+1})=(z_1, \cdots, z_n, z_{n+1}),$ 
%Set $\Phi(x)=-\log (-x)$ for $x <0.$ Then $-\log(-H)=\Phi(H)$. Note for $0 \leq i, j \leq n,$ we have
%$$(-\log(-H))_{i  \overline{j}}=\frac{1}{-H} H_{i \overline{j}} + \frac{1}{H^2} H_i H_{\overline{j}}.$$
%Note on $\Omega,$  $-H$ takes values in $(0, 1].$ Consequently, we have the matrix $((-\log(-H))_{i  \overline{j}}-H_{i \overline{j}})_{0 \leq i, j \leq n}$ is semi-positive definite at every point in $\Omega.$ Consequently, $-\log(-H)-H$ is plurisubharmonic on $\Omega.$
%For the proof of the later part of the lemma, see the Hang's note "Bergman Kernels on Hermitian Symmetric domains".
%By Lemma 1.2, $-\sqrt{-1} \partial \overline{\partial} \log  (-H)$ is  positive definite on $(\mathbb{C} \setminus \{w=0\}) \times D$. Therefore $\hat{\omega}$ is positive definite on $(\mathbb{C} \setminus \{w=0\}) \times D$. To verify $\hat{\omega}$ is positive definite at points on $\{w=0\} \times D.$ We compute, writing $\Phi=p(z, \overline{z})- |w|^2 q(z, \overline{z}),$
and fix a point $Z^*=(z^*, 0) \in \Om \times \{0\}.$
A direct computation yields
\begin{equation}\label{eqnphiij}
\left( (-\log r)_{i \overline{j}}\right)_{1 \leq i, j \leq n+1}|_{Z^*}=  \left( (-\log \rho)_{i \overline{j}} \right)_{1 \leq i, j \leq n}|_{z^*} \oplus \big(\frac{1}{\rho(z^*, \Ol{z^*})}\big).
\end{equation}
Then the positive definiteness of $\om_D$ at $Z^*$ follows easily from the above equation.
	
%The converse statement is also easy to see from the above equation.
%We now prove part (b). From the assumption and part (a), we see $\hat{\omega}$ is a metric on $\Omega$. 
We next prove the completeness of $\om_D$ on $D$. By the Hopf-Rinow theorem, it suffices to show  that $(\Omega, \om_D)$ is geodesically complete.
Let $\gamma: [0,a) \rightarrow \Omega$ be a  non-extendible geodesic in $D$ of unit speed with respect to $\omega_D$. We will need to show that $a=+\infty,$ equivalently that $\gamma$ has infinite length with respect to $\omega_D$. Write $\om_{\Om}$ for the (complete) metric induced by $-c_1(L, \rho^{-1})$, i.e., $\om_{\Om}=-\sqrt{-1} \partial \overline{\partial} \log  \rho$ on $\Om$; and write  $\pi$ for the natural projection from $L=\Om \times \mathbb{C}$ to $\Om$. We first observe, by the second inequality in (\ref{eqnine}), $\om_D \geq \pi^{*}(\om_{\Om})$. Next we write $\hat{\gamma}=\pi(\gamma)$ for the projection of $\gamma$ to $\Om$ and proceed in two cases:

%Seeking a contradiction, we assume $a < \infty.$  Then in the topology of $\mathbb{C}^{n+1}$ (where $L=\Om \times \mathbb{C} \subseteq \mathbb{C}^{n+1}$), for every convergent sequence
%$\gamma(t_i)$   with $t_i \rightarrow a$ we must have $\gamma(t_i) \rightarrow q \in \partial \Omega.$ Write $\mathcal{Q}$ for the collection of such limit points $q$. Then $\mathcal{Q} \subseteq \partial \Omega.$

{\bf Case I:} Assume $\hat{\gamma}$ is not contained in any compact subset of $\Om$. In this case, since $\om_{\Om}$ is complete on $\Om$, the length $L_{\om_{\Om}}(\hat{\gamma})$ of $\hat{\gamma}$ under the metric $\om_{\Om}$ equals $+\infty.$  Since $\om_D \geq \pi^{*}(\om_{\Om})$,  the length $L_{\om_D}(\gamma)$ of $\gamma$ under $\om_D$ satisfies $L_{\om_D}(\gamma) \geq  L_{\om_{\Om}}(\hat{\gamma}),$ and thus equals $+\infty.$ %Hence $a=+\infty.$
	
{\bf Case II:} Assume $\hat{\gamma}$ is contained in some compact subset $K$ of $M$. To seek a contradiction, we suppose $a<+\infty$.  Then for any sequence  $\{t_i\}_{i \geq 1}$ with $t_i \rightarrow a$, we must have $|\gamma(t_i)|_h \rightarrow 1.$ Otherwise, by passing to a subsequence, $\gamma(t_i)$ converges in $D(L),$ which contradicts the assumption that $\gamma$ is non-extendible. 
Then we can find some small $\epsilon >0$ and some $0 < t_0 <a $ such that $\gamma([t_0, a)) \subseteq W_{\epsilon}:=\{Z=(z, \xi) \in D: z \in \Om, |\xi| > \epsilon \}$ (where $|\cdot|$ is usual Euclidean norm). %Note $\phi=(-\psi) \mu$ for some positive smooth function $\mu$ in a neighborhood of $\Ol{W_{\epsilon}}.$ 
Recall by Cheng-Yau (see page 509 in \cite{CY}), $G:=-\log (-\psi)$ has bounded gradient with respect to $\omega_0=\sqrt{-1}\partial \Ol{\partial} G$. More precisely, $|\nabla G|_{\omega_0} \leq 1$ in $W_{\epsilon}$.
Here the gradient is taken with respect to $\omega_0.$ This in particular implies the length of $\gamma([t_0, a))$ with respect to $\omega_0$  equals $\infty$ (cf. the argument right before Remark 2.15 in \cite{EXX2}).  Note by the first equation in (\ref{eqnine}), we have  $\om_D \geq \om_0$. Consequently, the length of $\gamma$ with respect to $\omega_D$ also  equals $\infty$ and thus $a=+\infty,$ a plain contradiction. 

In any case, $a=+\infty$. This proves the completeness of $\om_D.$ Finally it will follow from part (3) of Proposition \ref{prpnauto} that $\om_D$ is a canonical metric on $D$. This finishes the proof. %of Proposition \ref{prpnmetric}.
%Next we choose a sequence $\{t_i\}_{i \geq 1}$ such that all $t_i \geq t_0, t_i \rightarrow a,$ and $\gamma(t_i) \in O_{i}:=\{z \in W_{\e}: \Psi(z)>2^i\}.$ Moreover, we can find some $0 < t_i^* <t_i$ such that  $\gamma((t_i^*, t_i]) \subset  O_{i-1}$ and in particular $\Psi (\gamma(t_i^*))=2^{i-1}.$ Then we have
%$$\int_{0}^{t_i} |\gamma'(t)|_{\omega} dt \geq \int_{t_i^*}^{t_i} |\gamma'(t)|_{\omega} dt \geq   \int_{t_i^*}^{t_i} |\gamma'(t)|_{\omega_0} dt \geq \int_{t_i^*}^{t_i} |\nabla \Psi|_{\omega_0} |\gamma'(t)|_{\omega_0} dt  $$ 
%$$\geq \int_{t_i^*}^{t_i} \langle \nabla \Psi, \gamma'(t) \rangle_{\omega_0} dt=  \int_{t_i^*}^{t_i}  \frac{d}{dt} \big(\Psi \circ \gamma(t)\big) dt= \Psi \circ \gamma(t_i)-\Psi \circ \gamma (t_i^*)> 2^{i-1}.$$
%This proves $\gamma$ has infinite length with respect to $\om_D$ and thus $a=+\infty,$ a plain contradiction.
\end{proof}

%%%%%%%%%%%%%%%%%%%%%%%%%%%%%%%%%%%%%%

\subsection{Automorphisms of the disk bundle $D(\Om, \rho)$}

In this section, we give a description of the automorphism of the disk bundle over bounded symmetric domains. We first include the following fact which is well-known to experts (cf. \cite{WYZR}), and we sketch a proof here for the convenience of the readers.

\begin{proposition}
Let $\Omega$ be an irreducible bounded symmetric domain of genus $\gamma.$ Write $N$ for the generic norm of $\Omega.$ Let $z^* \in \Omega$
and $\phi$ an automorphism of $\Om$ with $\phi(z^*)=0.$ Write $J\phi$ for the Jacobian matrix of $\phi,$ then $N(z, \Ol{z^*})$ is nowhere zero in $\Om$ and the following hold:
\begin{equation}\label{eqngnp}
\mathrm{det}(J \phi(z))=e^{i\theta}\frac{N^{\frac{\gamma}{2}}(z^*, \Ol{z^*})}{N^{\gamma}(z, \Ol{z^*})}~\text{for some}~\theta \in [0, 2\pi); \quad \frac{N(\phi(z),  \Ol{\phi(z)})}{N(z, \Ol{z})}=\frac{N(z^*, \Ol{z^*})}{|N(z, \Ol{z^*})|^2}.
\end{equation}

\end{proposition}

\begin{proof}
Note the Bergman kernel of $\Omega$ equals
$K_\Omega(z, \Ol{z})=\frac{A}{N^{\gamma}(z, \Ol{z})}.$
Here $A$ is some constant depending on $\Omega$. Let $\phi$ be an automorphism of $\Omega$.   By the transformation law of the Bergman kernel, we have $|\mathrm{det}(J\phi(z)|^2K_\Omega(\phi(z), \Ol{\phi(z)})=K_\Omega(z, \Ol{z})$ with $z \in \Omega.$ This yields
\begin{equation}\label{eqnjzo}
|\mathrm{det}(J\phi(z)|^2 N^{\gamma}(z, \Ol{z})=N^{\gamma}(\phi(z), \Ol{\phi(z)})~\text{with}~z \in \Omega.
\end{equation}
We complexify the above equation to get
\begin{equation}\label{eqnjzxi}
\mathrm{det}(J\phi(z)) \Ol{\mathrm{det}(J\phi(w))} N^{\gamma}(z, \Ol{w})= N^{\gamma}(\phi(z), \Ol{\phi(w)})~\text{with}~z, w \in \Omega.
\end{equation}
We let $z=z^*$ in (\ref{eqnjzo}) to obtain
$|\mathrm{det}(J\phi(z^*)|^2=N^{-\gamma}(z^*, \Ol{z^*}).$
We then let $w=z^*$ in (\ref{eqnjzxi}) to obtain
$\mathrm{det}(J\phi(z)) \Ol{\mathrm{det}(J\phi(z^*))} N^{\gamma}(z, \Ol{z^*})=1$ (recall $N(\cdot , 0) \equiv 1$).
We see  $N(z, \Ol{z^*})$ is nowhere zero in $\Om$, and we combine the latter two equations to get
\begin{equation}\label{eqnjph}
|\mathrm{det}(J\phi(z))|^2=|\mathrm{det}(J\phi(z^*))|^{-2}|N(z, \Ol{z^*})|^{-2\gamma}=\frac{N^{\gamma}(z^*, \Ol{z^*})}{|N(z, \Ol{z^*})|^{2\gamma}}.
\end{equation}
Since both $\mathrm{det}(J\phi(z))$ and $N(z, \Ol{z^*})$ are holomorphic in $\Om,$ we have the first equation in (\ref{eqngnp}) holds.
The second equation in (\ref{eqngnp}) follows from (\ref{eqnjzo}) and (\ref{eqnjph}).
\end{proof}

Motivated by the above, we set up the following notions: Let $\Om$ be an irreducible bounded symmetric domain with generic norm $N$ and $z^* \in \Om$. For $\phi \in \mathrm{Aut}(\Om)$ with  $\phi(z^*)=0$,  we set $R_{\phi}(z):=\frac{N^{\frac{1}{2}}(z^*, \Ol{z^*})}{N(z, \Ol{z^*})}.$ Then by (\ref{eqngnp}), $R_{\phi}(z)$ is holomorphic and nowhere zero in $\Om$. Moreover,
\begin{equation}\label{eqnrph}
|R_{\phi}(z)|^2=\frac{N(\phi(z), \Ol{\phi(z)})}{N(z, \Ol{z})}=|\mathrm{det}(J \phi(z))|^{\frac{2}{\gamma}}.
\end{equation}

\begin{remark}\label{rmk212}
With the above notations, writing $K$ for the Bergman kernel of $\Om$, $K (z, \Ol{z^*})$ extends holomorphically to a neighborhood of $\Ol{\Om}$ and thus $N(z, \Ol{z^*})$ is nowhere zero on $\Ol{\Om}.$  (cf. Lemma 1.1.1 in \cite{M2}). Hence there is some neighborhood $U$  of $\Ol{\Om}$ such that $R_{\phi}(z)$ extends holomorphically to $U$ and is everywhere nonzero there.
Let $\la$ be a positive real number. By making $U$ simply connected, we have a well-defined (holomorphic) branch of $(R_{\phi})^{\la}$  in $U$.
%(delete? This also follows from the first equation of Proposotion 2.1 and the fact that every automorphism of $D$ extends holomorphically to a neighborhood of $\Ol{\Om}$ (see Bell's result, see Lemma 2.2 Tu-Wang, Rigidity... Hua domains.) )
\end{remark}

\begin{proposition}\label{prpnspauto}
Let $\Om=\Om_1 \times \cdots \times \Om_l$ be a bounded symmetric domain in $\mathbb{C}^n$, where each $\Om_j, 1 \leq j \leq l,$ is an irreducible bounded symmetric domain in $\mathbb{C}_{z_j}^{n_j}$.
%Write $z=(z_1, \cdots, z_l)$ for the coordinates in $\mathbb{C}^n.$
Denote by $z=(z_1, \cdots, z_l) \in \Om_1 \times \cdots \times \Om_l$ the coordinates of $\Om \subseteq \mathbb{C}^n$.
Let $\rho(z, \Ol{z})=\prod_{j=1}^l \big(N_j(z_j, \Ol{z_j})\big)^{\la_j}$, where $N_j$ is the generic norm of $\Om_j$ and each  $\la_j$ is a positive real number (not necessarily an integers). Let $D:=D(\Om, \rho), \Sigma:=\Sigma(\Om, \rho)$ be as defined in (\ref{eqndlh}) and (\ref{eqnclh}), respectively. Let $\phi_j \in \mathrm{Aut}(\Om_j)$  for every $1 \leq j \leq l.$  Denote by $(z, \xi)$ the coordinates of $\Om \times \mathbb{C} \subseteq \mathbb{C}^n \times \mathbb{C}.$ Define $\Phi$ as
\begin{equation}\label{eqnpartauto}
\Phi(z, \xi):=\left(\phi_1(z_1), \cdots, \phi_l(z_l), e^{i\theta} \xi\prod_{j=1}^l (R_{\phi_j}(z_j))^{\la_j}  \right),~~\text{with}~\theta \in [0, 2\pi).
\end{equation}
Here as mentioned in Remark \ref{rmk212}, we can choose well-defined branches of $(R_{\phi_j}(z_j))^{\la_j}$ for all $1 \leq j \leq l$ in some neighborhood of $\Ol{\Om}.$
Then the following statements hold:

(1) $\Phi$ is a biholomorphism in a neighborhood of $\Ol{\Om} \times \mathbb{C}.$ 

(2) Recall the line bundle $L=\Om \times \mathbb{C}$ is equipped with the metric $|(z, \xi)|_h^2=|\xi|^2 \rho^{-1}(z, \Ol{z}).$  We have $\Phi$  is a norm-preserving isomorphism of $(L, h).$ 

(3) $\Phi: D \rightarrow D$ is isometric in the sense that $\Phi^*(\om_D)=\om_D,$ with $\om_D$ as given in $\S$2.3.
\end{proposition}

\begin{proof}
%Let $\Phi$ be as above. 
First as each $\phi_j$ extends holomorphically across $\partial \Om_j,$ $\Phi$ is holomorphic in a neighborhood of $\Ol{\Om} \times \mathbb{C}$. Secondly, set $\psi_j(z_j)=\phi_j^{-1}(z_j)$ for each $1 \leq  j \leq l$ and define
$$\Psi(z, \xi):=\left(\psi_1(z_1), \cdots, \psi_l(z_l), e^{i\alpha} \xi\prod_{j=1}^l (R_{\psi_j}(z_j))^{\la_j} \right),~~\text{with}~\alpha \in [0, 2\pi).$$
Again we can choose well-defined branches of $(R_{\psi_j}(z_j))^{\la_j}$ for all $1 \leq j \leq l$ in some neighborhood of $\Ol{\Om}.$
Since $\det J\psi_j(\phi_j(z_j))=(\det J \phi_j(z_j))^{-1},$ it holds that $|R_{\phi_j}(z_j)| \cdot |R_{\psi_j}(\phi_j(z_j))|=1.$ Consequently, we can choose an appropriate $\alpha$ such that $\Psi \circ \Phi$ equals to the identity map, and so is $\Phi \circ \Psi.$ This shows $\Phi$ is a biholomorphism in a neighborhood of $\Ol{\Om} \times \mathbb{C}.$
On the other hand, by (\ref{eqnrph}),  it holds that $\frac{N_j(\phi_j(z_j), \Ol{\phi_j(z_j)})}{N_j(z_j, \Ol{z_j})}=|R_{\phi_j}(z_j)|^2$. Consequently, writing $\Phi=(\w{\Phi}, \Phi_{n+1})=(\Phi_1, \cdots, \Phi_n, \Phi_{n+1}),$ the following holds in $\Om \times \mathbb{C},$ and in particular on $D:$
$$|\Phi_{n+1}|^2 \rho^{-1}(\w{\Phi}, \Ol{\w{\Phi}})=|\xi|^2 \rho^{-1}(z, \Ol{z}).$$
It follows that $\Phi$ is a norm-preserving isomorphism  of Hermitian line bundle $(L,h),$ and part (2) is proved. To prove part (3), we note, as a consequence of the equation above,
$$|\Phi_{n+1}|^{-2}(\rho(\w{\Phi}, \Ol{\w{\Phi}})-|\Phi_{n+1}|^{2})=|\xi|^{-2}(\rho(z, \Ol{z})- |\xi|^2),~\text{with}~(z, \xi) \in \Om \times \mathbb{C}~\text{with}~\xi \neq 0. $$
Consequently,
%\begin{equation}\label{eqnmetric2d9}
$$\rho(\w{\Phi}, \Ol{\w{\Phi}})-|\Phi_{n+1}|^{2}=(\rho(z, \Ol{z})- |\xi|^2)\prod_{j=1}^l |R_{\phi_j}(z_j)|^{2\lambda_j} ~\text{with}~(z, \xi) \in \Om \times \mathbb{C}.$$
%\end{equation}
We finally take the logarithm of the above equation on $D$, and apply $-\partial \Ol{\partial}$ to see $\Phi$ is isometric.
\end{proof}

\begin{remark}
Let $\mathcal{G}$ be the set of automorphisms of $D$ in the form as in (\ref{eqnpartauto}). One can verify that $\mathcal{G}$ is a subgroup of $\mathrm{Aut}(D),$ and that $\mathcal{G}$ acts transitively on $\Sigma$ and transitively on $\Om \times \{0\}.$
%and fiber-wise transitively on the line bundle $L=\Om \times \mathbb{C}$ 
However, $D$ is not a homogeneous domain in general. We also emphasize that in general $\mathcal{G}$ is a proper subset of $\mathrm{Aut}(D)$; see the following example.
\end{remark}

\begin{example}
Let  $\Omega=\Delta^2 \subseteq \mathbb{C}^2_{z}= \mathbb{C}_{z_1} \times \mathbb{C}_{z_2}$ and $\rho=(1-|z_1|^2)(1-|z_2|^2)$. Define $H$ to be $H(z_1, z_2, \xi)=(z_2, z_1, \xi).$ It is clear that $H \in \mathrm{Aut}(D)$ but $H \not \in \mathcal{G}.$
\end{example}

The next proposition describes the automorphism group $\mathrm{Aut}(D)$ of the disk bundle $D$. The proof of it borrows some ideas from
 Bi-Tu \cite{BT2}, Zhao-Wang-Hao \cite{ZWH},  etc.

\begin{proposition}\label{prpnauto}
Let $\Omega=\Omega_1 \times \cdots \times \Omega_l$ be a bounded symmetric domain as in Proposition \ref{prpnspauto}. Assume it has rank at least two. Let $\rho$, $D, \Sigma$ and $(z, \xi)$ be as in Proposition \ref{prpnspauto}. Let $F \in \mathrm{Aut}(D)$. Then

(1) $F$ has the following format:
\begin{equation}\label{eqnmapformat}
F(z, \xi)=(f(z), \xi h(z)).
\end{equation}
Here $f$ is an automorphism of $\Omega$, and $h$ is a nowhere zero holomorphic function on $\Om$. %which is in particular a rational map, and $h$ is also a rational function. 

(2)  $F$ extends to a norm-preserving holomorphic self-isomorphism of the line bundle $(L=\Om \times \mathbb{C}, \rho^{-1}).$

(3) $F: D \rightarrow D$ is an isometry of $D$ in the sense that $F^*(\om_D)=\om_D.$
\end{proposition}

\begin{proof}
Let $F$ be an automorphism of $D$. Write $F=(f, g)=(f_1, \cdots, f_n, g).$

\smallskip

% Recall that $D$ is bounded circular and the Bergman kernel of $D$ is rational. Then it follows from a result of Bell \cite{B} that $F$ is rational. Write $F=
%\frac{p}{q}$ where $p=(p_1, \cdots, p_{n+1})$ and $\{p_1, \cdots, p_{n+1},q\}$ has only trivial common factors. Since $F$ is well-defined on 
%$\Om \times \{\xi=0\} \subseteq D,$ we see $q \not \equiv 0$ on $\{\xi=0\}$. Then there exists some $(Z^*, 0) \in \partial \Om \times \{\xi=0\}$ such that $q \neq 0$  and $F$ is holomorphic in a neighborhood $U$ of $Z^*$. Write $M:= (\partial \Om \times \{\xi=0\}) \cap U.$

{\bf Claim:} We have $F(\Omega \times \{0\})=\Omega \times \{0\}$. Equivalently, $g \equiv 0$ on $\Omega \times \{0\}.$

{\bf Proof of Claim:} It suffices to prove for any sequence $\{p_i\}_{i \geq 1}$ in $\Omega \times \{0\}$ such that $F(p_i)$ converges, we have
the limit of $F(p_i)$ lies in $\partial D \setminus \Sigma=\partial \Omega \times \{0\}$. Suppose this is not true. Then there is a sequence $\{q_i\}_{i \geq 1} \subseteq \Omega \times \{0\}$ such that $\lim_{i \rightarrow \infty} F(q_i) \in \Sigma.$ Choose for each $i \geq 1,$ $\Phi_i \in \mathcal{G}$ such that $\Phi_i$ maps $0 \in D$ to $q_i.$ Then we obtain a sequence $\{F \circ \Phi_i\} \subseteq \mathrm{Aut}(D)$ such that $q:=\lim_{i \rightarrow \infty} F \circ \Phi_i(0)  \in \Sigma$ (Note $q$ is a strongly pseudoconvex boundary point of $D$ as $q \in \Sigma$). This implies $D$ has a strongly pseudoconvex boundary orbit accumulation point under the action of $\mathrm{Aut}(D).$ By a well-known theorem of Wong-Rosay, we see $D$ is biholomorphic to the unit ball. This is a contradiction, as $\Om$ has rank at least two. \qed

\smallskip

%To prove this, we write $V=\{(Z, \xi) \in \mathbb{C}^n \times \mathbb{C}: q(Z, \xi) \neq 0, \det JF(Z, \xi)=0 \}.$ Since $\Om \times \{0\} \not \subset V,$ we see $M \not \subset V$. Thus by perturbing $Z^*$ and shrinking $U$ if needed, we can assume $det JF \neq 0$ everywhere in $U$ and therefore $F$ is biholomorphic in $U$. Then $F$ cannot map any point on $M$ to $\Sigma,$ as the latter is strongly pseudoconvex while $M$ has  nontrivial complex subvarieties passing through a generic point. The first part of the claim then follows. To prove the latter assertion,  By the first part of the claim, we have $g=0$ on $M$. Consequently, $g(Z, 0)=0$ for all $Z \in \Omega$ and thus, $F$ must map $\Omega \times \{\xi=0\}$ to itself. The same argument and conclusion apply to $F^{-1}$ as well. This implies
%$f(Z, 0) \in \mathrm{Aut}(\Om).$ \qed

By the above claim, we can pick some $\Phi \in \mathcal{G}$ such that $\w{F}:=F \circ \Phi$ is an automorphism of $D$ mapping the origin $0 \in D$ to itself. Since $D$ is bounded and circular, by a classical theorem of Cartan,
%by Ishi-Kai \cite{IK} (see Proposition 2.1 there. See also Theorem 4 in \cite{KNY}), 
we have $\w{F}$ is a linear map:
$$\w{F}=(z, \xi) \begin{pmatrix} \mathcal{A} & \mathcal{B} \\ \mathcal{C} & \mathcal{D}  \end{pmatrix}. $$
Here $\mathcal{A}$, $\mathcal{B}$ and $\mathcal{C}$ are $n \times n$, $n \times 1$, $1 \times n$ matrices with complex entries, respectively, and $\mathcal{D}$ is a complex number. Write $\w{F}=(\w{f}, \w{g})=(\w{f}_1, \cdots, \w{f}_n, \w{g}).$ By the claim, $\w{g}(z, 0) \equiv 0.$ This implies $\mathcal{B}=0$. As $\w{F}$ is a nondegenerate linear map, we have $\det\mathcal{A} \neq 0$ and $\mathcal{D} \neq 0$. Note
$$\w{G}:=\w{F}^{-1}=(z, \xi) \begin{pmatrix} \mathcal{A} & \mathcal{B} \\ \mathcal{C} & \mathcal{D}  \end{pmatrix}^{-1}=(z, \xi) \begin{pmatrix} \mathcal{A}^{-1} & 0_{n \times 1} \\ -\mathcal{D}^{-1}\mathcal{C} \mathcal{A}^{-1} & \mathcal{D}^{-1}  \end{pmatrix}.$$
Here $0_{k_1 \times k_2}$ denotes the $k_1 \times k_2$ zero matrix.
Since $\rho(0,0)=1$, we have $(0, \xi) \in D$ whenever $|\xi|< 1.$ Consequently, $\w{F}(0, \xi)=(\xi \mathcal{C}, \xi \mathcal{D}) \in D$  when $|\xi|<1.$ That is, $\xi \mathcal{C} \in \Om$ (and thus $\mathcal{C} \in \Ol{\Om}$) and
\begin{equation}\label{eqnxirho}
|\xi \mathcal{D}|^2< \rho(\xi \mathcal{C}, \Ol{\xi \mathcal{C}}) \leq 1,~\forall~ |\xi|<1.
\end{equation}
The latter inequality above follows from the properties of generic norms. Note (\ref{eqnxirho}) implies $|\mathcal{D}| \leq 1.$ Applying the same argument to $\w{G},$ we get $|\mathcal{D}^{-1}| \leq 1.$ Therefore $|\mathcal{D}|=1.$
This together with  (\ref{eqnxirho}) implies $|\xi|^2< \rho(\xi \mathcal{C}, \Ol{\xi \mathcal{C}}) \leq 1$ for $|\xi|<1.$ Letting $\xi \rightarrow 1,$ we see $\rho(\mathcal{C}, \Ol{\mathcal{C}})=1$. By the properties of generic norms again, this can only occur when  $\mathcal{C}=0.$ Thus $\w{F}=(z\mathcal{A}, \xi \mathcal{D}),$ and $F=\w{F} \circ \Phi^{-1}$ is of the format (\ref{eqnmapformat}). As $F$ is an automorphism  of $D$, we must have $f$ is an automorphism of $\Omega$, and $h$ is nowhere zero on $\Om$.  This proves part (1).

To see $F$ is norm-preserving and isometric on $D$, we only need to show $\w{F}$ is norm-preserving and isometric on $D$ (as $\Phi^{-1}$ is known to have these properties by Proposition \ref{prpnspauto}). For that, we note $\w{F}$ must map $\Sigma$ to $\Sigma$. Therefore for every $z \in \Om,$ we have $|\xi|^2=|\xi \mathcal{D}|^2=\rho(z \mathcal{A}, \Ol{z\mathcal{A}})$ if $|\xi|^2=\rho(z, \Ol{z}).$
This yields $\rho(z \mathcal{A}, \Ol{z\mathcal{A}})=\rho(z, \Ol{z})$ for $z \in \Om.$
Consequently,
$$|\xi \mathcal{D}|^2 \rho^{-1}(z \mathcal{A}, \Ol{z\mathcal{A}})= |\xi|^2 \rho^{-1}(z, \Ol{z}); \quad \rho(z \mathcal{A}, \Ol{z\mathcal{A}})-|\xi \mathcal{D}|^2=\rho(z, \Ol{z})-|\xi|^2.$$
The first equation above shows $\w{F}$ is norm-preserving. To see $\w{F}$ is isometric on $D$, we take the logarithm of the second equation above, and then apply the  $-\partial \Ol{\partial}$ operator. This proves part (2) and (3) of the proposition.
\end{proof}

\subsection{Generic norms and Segre varieties}

Let $\Om$ be a (possibly reducible) bounded symmetric domain in $\mathbb{C}^n_z$ and $\rho$ a generalized  generic norm of $\Omega.$ Let $D=D(\Om, \rho) \subseteq \mathbb{C}^{n+1}_{(z, \xi)}$ and
$\Sigma=\Sigma(\Om, \rho) \subseteq \mathbb{C}^{n+1}_{(z, \xi)}$ be as defined in (\ref{eqndlh}) and (\ref{eqnclh}), respectively. Write $r(z, \xi, w, \eta):=\rho(z, w)-\xi \eta$ for the generalized generic norm of $D$, where $z, w \in \mathbb{C}^n$ and $\xi, \eta \in \mathbb{C}$. Fix $Z^* =(z^*, \xi^*) \in \mathbb{C}^{n}_{z} \times \mathbb{C}_{\xi}.$ Recall the Segre variety of $\Sigma$ associate to $Z^*$ is
\begin{equation}\label{eqnbq}
	Q_{Z^*}=\big\{(z, \xi) \in \mathbb{C}^{n+1}:  r(z, \xi, \ol{z^*}, \ol{\xi^*})=0 \big\}.
\end{equation}

It is clear that $Q_{Z^*}$ is a (smooth) complex submanifold of $\mathbb{C}^{n+1}$ if $\xi^* \neq 0.$ We have the following facts about the generic norm $r$ and Segre varieties, which will be useful  in the proof of Theorem \ref{T1}.

\begin{proposition}\label{prpnzsh}
We have the following statements hold:

(a) Fix any $(z^*, \xi^*) \in \mathbb{C}^{n+1}$ with $\xi^* \neq 0.$  We have $r(z, \xi, \ol{z^*}, \ol{\xi^*})$ is irreducible in $\mathbb{C}[z, \xi]$.

(b) Let $Z^*=(z^*, \xi^*), \hat{Z}=(\hat{z}, \hat{\xi}) \in \mathbb{C}^{n+1}.$ Assume there is a germ of $n-$dimensional complex manifold $X \subseteq \mathbb{C}^{n+1}$ such that $X \subseteq Q_{Z^*} \cap Q_{\hat{Z}}.$ If $\xi^* \neq 0$ or $\hat{\xi} \neq 0,$ then $Z^*=\hat{Z}.$

\end{proposition}

\begin{proof}
To establish part (a), writing $p(z, \xi)=r(z, \xi, \ol{z^*}, \ol{\xi^*}),$ we assume
\begin{equation}\label{eqnr12}
p(z, \xi)=p_1(z, \xi)p_2(z, \xi).
\end{equation}
for some polynomials $p_1, p_2 \in \mathbb{C}[z, \xi]$. Since $p$ is linear in $\xi,$ without loss of generality, we can assume
$p_1(z, \xi)=a(z)\xi+b(z)$, and $p_2$ does not dependent on $\xi: p_2(z, \xi)=p_2(z).$ Then by comparing the $\xi-$terms on both sides of (\ref{eqnr12}), we get $a(z)p_2(z)=-\ol{\xi^*}.$ Consequently, $p_2(z)$ and $a(z)$ are constant. Hence $p$ is irreducible.

To establish part (b), without loss of generality, we assume $\xi^* \neq 0.$ Then by part (a), $Q_{Z^*}$ is irreducible. By assumption, we have
$Q_{\hat{Z}}$ contains some open piece of $Q_{Z^*}.$ Consequently, by Hilbert's Nullstellensatz, $r(z, \xi, \ol{z^*}, \ol{\xi^*})$ divides $r(z, \xi, \ol{
\hat{z}}, \ol{\hat{\xi}})$. This yields $\hat{\xi} \neq 0$(for otherwise, $r(z, \xi, \ol{\hat{z}}, \ol{\hat{\xi}})$ is independent of $\xi$, while $r(z, \xi, \ol{z^*}, \ol{\xi^*})$ has $\xi-$term, a plain contradiction). By part (a) again, $Q_{\hat{Z}}$ is also irreducible. Since $Q_{Z^*}$ and $Q_{\hat{Z}}$ share an open subset by assumption, we have $Q_{Z^*}=Q_{\hat{Z}}.$ Recall $\rho(0, \ol{z})=1$ for all $z$. Therefore $\big(0, (\ol{\xi^*})^{-1}\big) \in \mathbb{C}^{n} \times \mathbb{C}$ is on $Q_{z^*}$, and thus on $Q_{\hat{Z}}$. This implies $\xi^*=\hat{\xi}.$ As $Q_{Z^*}=Q_{\hat{Z}},$ we must have $\rho(z, \ol{z^*})=\rho(z, \ol{\hat{z}})$. By the special form of the generic norm $\rho$ (cf. \cite{Lo}. Indeed, with $z=(z_1, \cdots, z_n)$ and $w=(w_1, \cdots, w_n),$ the value of $\frac{\partial \rho}{\partial z_j}(z, w), 1 \leq j \leq n,$ at $(0, w)$  equals $c_j w_j$ for some nonzero universal constant $c_j$ that is independent of $w$), we have $z^*=\hat{z}.$ This proves part (b).
\end{proof}

The following proposition will also be applied in the proof of the main theorem.

\begin{proposition}\label{prpnrir}
Let $p(z, w)$ be a nonzero polynomial in $\mathbb{C}[z, w].$ Then $q(z, w, \xi, \eta):= p(z, w)-\xi \eta$ is irreducible in $\mathbb{C}[z, w, \xi, \eta].$ In particular, $r(z, w, \xi, \eta):= \rho(z, w)-\xi \eta$ is irreducible in $\mathbb{C}[z, w, \xi, \eta],$ where $\rho$ is a generalized generic norm of some (possibly reducible) bounded symmetric domain.
\end{proposition}

\begin{proof}
Suppose $q$ is reducible. Then there exist $P, Q \in \mathbb{C}[z, w, \xi, \eta]$ with $\deg P \geq 1, \deg Q \geq 1,$  such that $q=PQ.$ Without loss of generality, we assume the degree of $P$ in $(\xi, \eta)$ is greater than or equal to that of $Q$. Note we can only have the following two cases: (a) The degree of $P$ in $(\xi, \eta)$ equals $2$ and that of $Q$ equals $0$. In this case we have 
$P=a_0\xi \eta +a_1 \xi + a_2 \eta +a_3 $ and $Q=b$ for some $a_0, a_1, a_2, a_3, b \in \mathbb{C}[z, w].$ Note we must 
have  $a_0b=-1$. This is impossible as $\deg Q \geq 1$. (b) The degree of both $P$ and $Q$ in $(\xi, \eta)$ equals $1$. Without loss of generality, assume $P=a_0\xi+a_1, Q=b_0\eta +b_1$ with $a_0, a_1, b_0, b_1 \in \mathbb{C}[z, w].$
Note $a_0, a_1, b_0, b_1$ are all nonzero polynomials as we must have $a_0b_0=-1$ and $a_1b_1=p$. Then $a_0b_1 \neq 0$ and $a_1b_0 \neq 0.$  This yields $PQ$ has linear terms in $(\xi, \eta).$ But $q$ has no such terms, a plain contradiction. Hence  $q$ is irreducible.
\end{proof}

\section{Proof of Theorem \ref{T1}}

%{\bf Proof of Theorem \ref{T01}:} Write $p(z, \ol{z})=|r_1(z)|^2 \hat{p}(z, \ol{z})$ and $q(z, \ol{z})=|r_1(z)|^2 \hat{q}(z, \ol{z}).$ 
%We replace $p, q$ by $\hat{p}, \hat{q}$ and still call them $p, q$. Then the assumptions (a), (b) still hold (note (c) has nothing to do with $p, q$, therefore still holds as well). Indeed,

%For (a), since $p>0$ on $D$, see $\hat{p}>0, r_1 \neq 0$ on $D_1$. Likewise, we have $\hat{q} >0, r_2 \neq 0$ on $\ol{D_1}.$  Since the zero locus $V=\{R_1=0\}$ is nowhere dense, equivalently, for every $z^* \in \partial D_1,$ there is a point in $\partial D_1 \setminus V$ arbitrarily close to $z^*.$

%For (b), since $r_1 \neq 0, r_2 \neq 0$ on  $D_1$, we have $\omega_D:=\sqrt{-1} \partial \overline{\partial} \log \frac{q}{p}=\sqrt{-1} \partial \overline{\partial} \log \frac{\hat{q}}{\hat{p}}$.

In this section, we give a proof of Theorem \ref{T1}. For convenience, we will write $\Sigma_i=\Sigma(\Omega_i, \rho_i)$ for $i=1,2$.
As in Remark \ref{R1}, we can indeed assume $F$ is a nonconstant holomorphic map from an open connected subset $U \subseteq \mathbb{C}^{n+1}$ (with $U \cap \Sigma_1 \neq \emptyset$) to $\mathbb{C}^{n+1}$, and satisfies $F(U \cap \Sigma_1) \subseteq \Sigma_2.$ We carry out the proof in two steps. In $\S$3.1, we use a monodromy argument to establish the rationality of the map $F$. %and we will borrow ideas from \cite{HZ} and \cite{HLX}. 
Then in $\S$3.2, we use the rationality of $F$ to study its geometric property in terms of the canonical metric introduced in $\S$2.3. %will play an important role in the proof.
%We will also apply some recently developed ideas and techniques in CR geometry \cite{HZ, HLX, HX, FHX}.

\subsection{Rationality of the maps $F$ and $F^{-1}$}

By Webster's Theorem \cite{W1} (see also Huang \cite{Hu1}, as well as Zaitsev \cite{Za}, where very general algebraicity results were established), $F$ extends to an algebraic holomorphic map. Then to show the rationality of $F$, it suffices to prove the single-valueness of the extension of $F$.  First Lemma 3.7 in Huang-Zaitsev \cite{HZ} will be an important starting point of the proof. We need a slightly more general formulation of it as follows.

\begin{proposition}\label{prpnhz37}
	Let $U$ be a connected open subset in $\mathbb{C}^m,$ and $\mathcal{V}$ a complex subvariety of $U$ of codimension one with $p \in  U \setminus \mathcal{V}$.
	Write $\Gamma$ for the subgroup  of  $\pi_1(U \setminus \mathcal{V}, p) $ that is generated by loops obtained by concatenating paths $\gamma_1, \gamma_2, \gamma_3,$ where $\gamma_1$ connects $p$ with a point arbitrarily close to a smooth point $q_0 \in \mathcal{V}, \gamma_2$
	is a loop around $\mathcal{V}$ near $q_0$ and $\gamma_3$ is $\gamma_1$ reserved.
	
	(1). Let $\gamma$ be a loop in $U \setminus \mathcal{V}$ with base point $p$. If $\gamma$ is null-homotopic  in $U,$ then $\gamma \in \Gamma.$
	
	(2). In particular, if  $U$ is simply connected, then $\Gamma=\pi_1(U \setminus \mathcal{V}, p).$ 
\end{proposition}

\begin{proof}
	To prove part (1),  we first note that we can replace $U$ by the complement of the singular set $\mathrm{Sing}(\mathcal{V})$ of $\mathcal{V}$ in $U$ without affecting the assumption that $\gamma$ is null-homotopic  in $U$. Indeed, since $\mathrm{Sing}(\mathcal{V})$ is of complex codimension at least two, there is an isomorphism from $\pi_1(U \setminus \mathrm{Sing}(\mathcal{V}), p)$ to $\pi_1(U, p)$ induced by the inclusion map from $U \setminus \mathrm{Sing}(\mathcal{V})$ to $U$ (by Thom's transversality theorem). Then the conclusion follows from the same proof as Lemma 3.7 in \cite{HZ}. 
	Part (2) is a consequence of part (1), and is indeed identical with the statement of Lemma 3.7 in \cite{HZ}. 
\end{proof}

%\begin{lemma} \label{lm2d1}
%Let $U \subset \mathbb{C}^{n}, n \geq 2,$ be a simply connected open subset and
%$\mathcal{S} \subset U$ be a closed complex analytic subset of
%codimension one. Then for $p \in U \setminus \mathcal{S},$ the
%fundamental group $\pi_{1}(U \setminus \mathcal{S},p)$ is generated
%by loops obtained by concatenating (Jordan) path
%$\gamma_{1},\gamma_{2},\gamma_{3},$ where $\gamma_{1}$ connects $p$
%with a point arbitrary close to a smooth point $q_{0} \in
%\mathcal{S},$ $\gamma_{2}$ is a loop around $\mathcal{S}$ near
%$q_{0}$ and $\gamma_{3}$ is $\gamma_{1}$ reversed.
%\end{lemma} 

Let $\Om$ be a (possibly reducible) bounded symmetric domain in $\mathbb{C}^n_z$ and $\rho$ a generalized  generic norm of $\Omega.$
Let $\Sigma:=\Sigma(\Om, \rho) \subseteq \mathbb{C}^{n+1}_{(z, \xi)}$ be the real hypersurface as defined in Definition \ref{D2}. Denote by $\mathcal{M}$ the complexification of $\Sigma$, which is defined by
%$$M=\partial' \Omega_1=  \big\{(w, z) \in \mathbb{C} \times D : |w|^2 = \rho(z, \overline{z}) \big\};$$
\begin{equation}\label{eqncpm}
\big\{(z, \xi, w, \eta) \in \mathbb{C}^{n+1}_{(z, \xi)} \times \mathbb{C}^{n+1}_{(w, \eta)}:  r(z, \xi, w, \eta)=0 \big\},~\text{where}~r(z, \xi, w, \eta):=\rho(z, w)-\xi \eta.
\end{equation}
Write $\mathrm{Reg}(\mathcal{M})$ for the set of smooth points of $\mathcal{M}$ and $\mathrm{Sing}(\mathcal{M}):=\mathcal{M} \setminus \mathrm{Reg}(\mathcal{M}).$ Since $\mathcal{M}$ is irreducible (by Proposition \ref{prpnrir}, $r$ is irreducible), $\mathrm{Reg}(\mathcal{M})$ is connected. Moreover, by a simple computation, we have $\mathrm{Sing}(\mathcal{M}) \subseteq \{\xi=\eta=0\}.$
The following proposition is inspired by \cite{HLX}.

\begin{proposition} \label{prpnptloop}
Let $\Sigma$ and $\mathcal{M}$ be defined as above and let $p_0=(z_{p_0}, {\xi}_{p_0}) \in \Sigma$ (and thus $\xi_{p_0} \neq 0$). Let $\mathcal{S}$ be a complex algebraic hypervariety (variety of codimension one) in $\mathbb{C}^{n+1}$  satisfying $\mathcal{S} \neq \{\xi = 0\}$ and $p_0 \not \in \mathcal{S}.$  Let $\gamma \in \pi_1(\mathbb{C}^{n+1} \setminus \mathcal{S}, p_0)$ be obtained by concatenation of $\gamma_1,\gamma_2,\gamma_3$ as described  in Proposition \ref{prpnhz37} (with $U=\mathbb{C}^n$ and $\mathcal{V}=\mathcal{S}$), where $\gamma_2$ is a small loop around $\mathcal{S}$ near a smooth point $q_0=(z_{q_0}, \xi_{q_0})\in \mathcal{S}$ with $\xi_{q_0}\not =0$.  
Then, with a small perturbation of $p_0$ and $q_0$ if necessary, $\gamma$ can be homotopically perturbed to a  loop $\tilde{\gamma} \in \pi_1(\mathbb{C}^{n+1} \setminus \mathcal{S}, p_0)$ such that there exists a null-homotopic loop $\lambda\in \pi_1(\mathbb{C}^{n+1}\setminus \mathcal{S}, p_0)$ satisfying $(\lambda, \overline{\tilde\gamma}) \subseteq \mathrm{Reg}(\mathcal{M})$ and
$(\tilde\gamma,\overline{{\lambda}}) \subseteq \mathrm{Reg}(\mathcal{M}).$
%Also, for such an element $\hat{\gamma} \in \pi_{1}({\mathbb C}^{n+1}\setminus \mathcal{S}, p_0)$,  after a small perturbation to $\hat{\gamma}$ if needed, we can find a null-homotopic loop in $\hat{\lambda} \in \pi_{1}({\mathbb C}^{n+1} \setminus \mathcal{S}, p_0)$ such that $(\hat{\gamma},\overline{\hat{\lambda}}) \subset \mathcal{M}.$
\end{proposition}

{\bf Proof of Proposition \ref{prpnptloop}.} %For $p=(z_p, \xi_p) \in \mathbb{C}^{n+1},$ recall the Segre variety of $\Sigma$ associated to $p$ is given by $Q_p=\{(z, \xi) \in \mathbb{C}^n \times {\mathbb C}: \xi \Ol{\xi_p}= \rho(z, \Ol{z_p}) \}.$ Moreover, $Q_{p}$ is  smooth if $\xi_{p} \neq 0$. 
Define a function $\phi$ on $\mathbb{C}_z^n \times \mathbb{C}_{(w, \eta)}^{n+1}:$ 
$\phi(z, w, \eta)=\frac{\rho(z, w)}{\eta}$ which is well-defined for $\eta \neq 0.$ For each $z \in \mathbb{C}^n,$
We further define a map ${\cal R}_{z}(w, \eta)=(z, \phi(z, w, \eta))$ from $\mathbb{C}_{(w, \eta)}^{n+1}\sm \{\eta = 0\}$ to $\mathbb{C}^{n+1}.$ Note ${\cal R}_{z}(w, \eta)$ is holomorphic in $(w, \eta)$ for $\eta \neq 0$. Furthermore, from the
definition, we see $({\cal R}_{z}(w, \eta), w, \eta) \in \mathcal{M}$ wherever it is defined.
%that ${\cal R}_{z}$ sends $(w, \eta)$ to $Q_{(\Ol{w}, \Ol{\eta})}.$ 
Moreover, when
$p=(z_{p}, \xi_{p}) \in \Sigma$ (which implies $\xi_{p} \neq 0$),   we have, writing $\Ol{p}:=(\Ol{z_p}, \Ol{\xi_p}),$ ${\cal R}_{z_p}(\Ol p)=(z_p, \phi(z_p, \Ol{z_p}, \Ol{\xi_p}))=p$. 

\smallskip

%??(Moreover, for any $p_{1} \neq p_{2}\in {\mathbb C}^{n+1}$ with $w_{p_{1}} \neq 0, w_{p_{2}} \neq 0$ and for any $U \subset \mathbb{C}^{n+1},$
%$Q_{q_{1}} \not\equiv Q_{q_{2}}$ in $U$ unless they both are empty subset. seems correct but needed?) After slightly perturbing   $p_0$ in $M$, if needed, we can assume without loss of generality  that $w_{p_0}\not = 0$.??

To establish Proposition \ref{prpnptloop}, a couple of lemmas are in order.  Denote by $\Ol{\mathcal{S}}$ the conjugate of $\mathcal{S}: \Ol{\mathcal{S}}=\{(w, \eta) \in \mathbb{C}^{n} \times \mathbb{C}: (\Ol{w}, \Ol{\eta}) \in \mathcal{S} \}.$ It is clear that $\Ol{\mathcal{S}}$ is also a complex algebraic hypervariety  in $\mathbb{C}^{n+1}_{(w, \eta)}$ and $\Ol{\mathcal{S}} \neq \{\eta = 0\}$. %We next prove the following lemma:

\smallskip

\begin{lemma}\label{lm3d3}
Let $M, \mathcal{S}$ be as in Proposition \ref{prpnptloop} and $U$ an open set in $\mathbb{C}_z^n$. Let $V$ be an open piece of the hypervariety $\Ol{\mathcal{S}} \subseteq \mathbb{C}^{n+1}_{(w, \eta)}$ such that for every $(w, \eta) \in V$, it holds that $\eta \neq 0.$ Set
%$W$ of $U \times V$ under the map ${\cal R}_{z}(w, \eta):$
$W:=\{{\cal R}_{z}(w, \eta): z \in U, ~(w, \eta) \in V \}.$
%It holds that 
We have $W$ contains an open subset of $\mathbb{C}^{n+1}.$

\end{lemma}

{\bf Proof of Lemma \ref{lm3d3}.}
By shrinking $V$ if necessary, we can assume $V$ is an open piece of the regular part of (some irreducible component of) $\Ol{\mathcal{S}}$. 
%We regard  $U$ and $\mathcal{S}$ (and thus $V$) to be in the space $\mathbb{C}_z^n$ and $\mathbb{C}_{(\eta, \xi)}^{n+1}$, respectively. 
We first note $\phi(z, w, \eta)=\frac{\rho(z, w)}{\eta}$ is not constant on  $U \times V.$  Indeed by the special form of the generic norm $\rho$ (see the latter part of the proof of Proposition \ref{prpnzsh}),  $\rho(z, w)$ is constant in $z$ if and only if $w=0$. Hence $\phi$ cannot be constant on $U \times V$ if $V \not \subseteq \{w=0\}.$ On the other hand if $V \subseteq \{w=0\}$ (this can only occur when $n=1$), then $\eta$ cannot be constant on $V$. This again yields $\phi$ is not constant on $U \times V$.
We next prove the following claim:

%Define a function $\psi$ from $U \times V$ by $\psi(z, \eta, \xi)=\frac{\rho(z, \xi)}{\eta};$ and a holomorphic map ${\cal T}_{z}(\eta, \xi)=(\psi(z, \eta, \xi), z)$ from $U \times \Ol{V}$ to $\mathbb{C}^{n+1}.$ Note ${\cal T}_{z}$ only differs from ${\cal R}_{z}$ by taking the conjugate of $\eta, \xi.$ Therefore, to establish the lemma, we only need to prove that the image of $U \times \Ol{V}$ under the map ${\cal T}_{z}(\eta, \xi)$ contains an open set of $\mathbb{C}^{n+1}.$ For that, we first prove the following claim.

{\bf Claim.} There is a holomorphic tangent vector field $L$ along $V$ such that $L \phi(z, w, \eta) \not \equiv 0$ on  $U \times V$.

{\bf Proof of Claim.} Seeking a contradiction, we suppose $L \phi(z, w, \eta) \equiv 0$ on $U \times V$  for every holomorphic tangent vector $L$  along $V$. Then for each $z \in U$, $\phi(z, w, \eta)$ is constant in $(w, \eta) \in V$. Fixing $(w^*, \eta^*) \in  V$, we have 
$\phi(z, w,  \eta)=\frac{\rho(z, w^*)}{\eta^*}=a(z)$ on $U \times  V$. Here $a$ is some nonconstant holomorphic polynomial (as $\phi$ is not constant on $U \times V$). Then the following holds for all $z$ in $U$ and thus for all $z \in \mathbb{C}^n.$
$$h(z, w, \eta):=\rho(z, w)-a(z) \eta \equiv 0~\text{for every}~ (w, \eta) \in  V.$$
Write $b(w, \eta)$ for the irreducible polynomial that defines the irreducible component of $\Ol{\mathcal{S}}$ containing $V$.
Then we have $b(w, \eta)$ divides $h(z, w, \eta)$, 
which is impossible. Indeed, assume $b(w, \eta) | h(z, w, \eta).$ Fix any $z_0 \in \mathbb{C}^n$ such that $a(z_0)=0$. As $\rho(\cdot, 0) \equiv 1$, we have $\rho(z_0, w) \not \equiv 0$. The assumption would then lead to $b(w, \eta) | h(z_0, w, \eta),$ i.e., $b(w, \eta) | \rho(z_0, w).$ This yields $b$ does not depend on $\eta,$ and we write $b(w, \eta)=b(w).$ But then in this case we can use the assumption
to see $b(w)$ divides $h(z, w, \eta_1)-h(z, w, \eta_2)$ for any two different complex numbers $\eta_1$ and $\eta_2.$ This would lead to $b(w)| a(z),$ a plain contradiction. This proves the claim. \qed

Writing $z=(z_1, \cdots, z_n)$,  we next compute the complex Jacobian $\mathcal{J}$ of ${\cal R}_{z}(w, \eta)$ as a holomorphic map from $U \times V$ to $\mathbb{C}^{n+1}$ along the complex directions $\frac{\partial}{\partial z_1}, \cdots, \frac{\partial}{\partial z_n}, L.$ Note $Lz=0$ and
$\frac{\partial z_j}{\partial z_i}=\delta_{ij}.$ Combining this fact with the above claim, we see $\mathcal{J}$ is nondegenerate. Hence the image $W$ of $U \times V$  must contain an open subset. This proves Lemma \ref{lm3d3}. \qed

%We next prove the following lemma.

\begin{lemma}\label{lm3d4}
Given any open piece $M$  of  $\Sigma$  and any open piece $V$ of the complex hypervariety $\Ol{\mathcal{S}}$, then there exist a point $p^*=(z^*, \xi^*) \in M$ and a point $(w^*, \eta^*) \in  V$,  such that ${\cal R}_{z^*}(w^*, \eta^*) \not \in \mathcal{S}$.
\end{lemma}
{\bf Proof of Lemma \ref{lm3d4}:} 
Note the projection of $M \subseteq \mathbb{C}_{(z, \xi)}^{n+1}$ to $\mathbb{C}^n_z$ contains an open subset of $\mathbb{C}^n_z$. Then the conclusion follows from Lemma \ref{lm3d3}. \qed
%Suppose not. Then ${\cal R}_{z}(w, \eta)  \in \mathcal{S}$ for every $z \in M$ and $(w, \eta)  \in V$.
%By the holomorphicity of ${\cal R}_{z}(w, \eta)$ in $z$, we see the same is true for $z \in \mathbb{C}^n$ and $(w, \eta)  \in V$.
%This, however, contradicts Lemma \ref{lm3d3}. \qed

\smallskip

We continue to prove Proposition \ref{prpnptloop}.  Recall $\Ol{q}$ denotes $(\Ol{z_{q}}, \Ol{\xi_{q}})$ for any $q= (z_q, \xi_q) \in \mathbb{C}^{n+1}.$ By Lemma \ref{lm3d4}, perturbing the points $p_0$ and $q_0$ in Proposition \ref{prpnptloop} if needed, we can assume ${\cal R}_{z_{p_0}}(\Ol{q_0}) \not \in \mathcal{S}.$ Consequently, there exists a small ball $O_{q_0}$ centered at $\Ol{q_0}$ such that ${\cal R}_{z_{p_0}}(O_{q_0}) \cap \mathcal{S} =\emptyset$ and that ${\cal R}_{z_{p_0}}(O_{q_0})$ is contained in some contractible open subset of $\mathbb{C}^{n+1}\setminus \mathcal{S}$. Here ${\cal R}_{z_{p_0}}(O_{q_0}):= \{{\cal R}_{z_{p_0}}(w, \eta): (w, \eta) \in  O_{q_0} \}.$ Back to the loop $\gamma,$ we now deform
$\gamma_1, \gamma_2,
\gamma_3$  to $\wt\gamma_1, \wt\gamma_2, \wt\gamma_3$ respectively.
Here $\wt \gamma_1$ connects $p_0$ with a point  $q^* \approx q_0$ in $O_{q_0}$ , $\wt\gamma_2$ is a loop based at $q^*$ around
$\mathcal{S}$ inside $O_{q_0}$ and sufficiently close to $q_0$, and
$\wt\gamma_3$ is $\wt\gamma_1$ reserved such that the loop
$\wt\gamma$ obtained by concatenation of $\wt\gamma_1,
\wt\gamma_2,\wt\gamma_3$ is the same as $\gamma$ as elements
in $\pi_1(\mathbb{C}^2\setminus \mathcal{S}, p_0)$. 
With a further perturbation if needed, we can assume that the $\xi$-coordinate of points on $\wt\gamma$
never vanishes. %For convenience, we still denote $\wt\gamma_i'$s and $\wt \gamma$ by $\gamma_i'$s and $\gamma$, respectively.

Define $\lambda_1:={\cal R}_{z_{p_0}}(\Ol{\wt\gamma_1})$ and $\lambda_{2}:={\cal R}_{z_{p_0}}(\Ol{\wt\gamma_2}) \subseteq {\cal R}_{z_{p_0}}(O_{q_0})$.
We note $\lambda_{2}$ is null-homotopic in $\mathbb{C}^{n+1}\setminus \mathcal{S}.$
Moreover, the closure of $\{(w, \eta) \in \mathbb{C}^{n+1}: \eta \neq 0, {\cal R}_{z_{p_0}}(w, \eta) \in \mathcal{S}\}$ defines a complex algebraic hypervariety $\mathcal{V}_0$ in $\mathbb{C}^{n+1}_{(w, \eta)}$ with $\ol{p_0} \not \in \mathcal{V}_0.$ Indeed, if we write $\mathcal{S}=\{(z, \xi) \in \mathbb{C}^{n+1}: A(z, \xi)=0\}$ for some nontrivial holomorphic polynomial $A$, then ${\cal R}_{z_{p_0}}(w, \eta) \in \mathcal{S}$ implies $B(w, \eta):=A( z_{p_0}, \phi(z_{p_0},w, \eta))=0$. One can readily see that the latter yields a nontrivial holomorphic polynomial equation in $(w, \eta)$, by annihilating the denominator, that defines $\mathcal{V}_0$. Since $p_0={\cal R}_{z_{p_0}}(\Ol{p_0}) \not \in \mathcal{S}$, we have $\ol{p_0} \not \in \mathcal{V}_0.$   Hence we can slightly perturb $\wt \gamma_1$ (without changing its initial point $p_0$) to make $\ol{\wt \gamma_1}$ avoids $\mathcal{V}_0$, and as a result,  $\lambda_1$ avoids $\mathcal{S}.$

%We  choose a suitable path $\{\xi(t):0 \leq t \leq 1\}$ in $\mathbb{C}$ with $\xi(0)=\xi(1)=\xi_{0}$ such that if  we
%define $\lambda_1={\cal R}_{\xi(t)}(\wt\gamma_1),$ then $\lambda_{1}$
%avoids $\mathcal{S}$(with possibly a slight perturbation of $\wt\gamma_{1}$ fixing endpoints).
With such choices of $\wt \gamma_1$ and $\lambda_1,$ define $\lambda_3$ to be
the reverse of $\lambda_1,$ and $\lambda$ to be the concatenation of
$\lambda_1, \lambda_2, \lambda_3.$ Then $\lambda$ is a
null-homotopic loop in $\pi_1(\mathbb{C}^{n+1}\setminus \mathcal{S},
p_0).$ Moreover, $(\lambda(t), \overline{\wt\gamma(t)}) \in  \mathcal{M}$ for all $t$ by the way it was constructed. As the $\xi$-coordinate of  $\wt\gamma(t)$
is nonzero, $(\lambda(t), \overline{\wt\gamma(t)}) \not \in \{\xi=\eta=0\}.$ Consequently, $(\lambda(t), \overline{\wt\gamma(t)}) \not \in \mathrm{Sing}(\mathcal{M}).$   The last statement
in the proposition follows from the symmetric property of Segre varieties. \qed

\smallskip

The following proposition will be important for  our later proof. To formulate it, we set up a few notations. Let $\Sigma_1$ and $\Sigma_2$ be as at the beginning of $\S$3.  Write $\mathcal{M}_1$ and $\mathcal{M}_2$ for the complexification of  $\Sigma_1$ and $\Sigma_2$, respectively. That is, for $1 \leq i \leq 2$,
\begin{equation}\label{eqncpmi}
\mathcal{M}_i=\big\{(z, \xi, w, \eta) \in \mathbb{C}^{n+1}_{(z, \xi)} \times \mathbb{C}^{n+1}_{(w, \eta)}:  r_i(z, \xi, w, \eta)=0 \big\},~\text{with}~r_i(z, \xi, w, \eta):=\rho_i(z, w)-\xi \eta.
\end{equation}

For $p, q \in \mathbb{C}^{n+1},$ write $Q_{1,p}$ and $Q_{2,q}$ for the Segre variety of $\Sigma_1$ and $\Sigma_2$ associate to $p$ and $q$, respectively.
More precisely,
\begin{equation}\label{eqnbq1}
	Q_{1,p}=\big\{(z, \xi) \in \mathbb{C}^{n+1}:  r_1(z, \xi, \ol{z_p}, \ol{\xi_p})=0 \big\},~\text{where}~p=(z_p, \xi_p) \in \mathbb{C}^{n} \times \mathbb{C}.
\end{equation}
Likewise, $Q_{2,q}$ is defined as above with $r_1$ and $p$ replaced by $r_2$ and $q$, respectively. It is clear that if $p \in \Sigma_1,$ then $p \in Q_{1, p};$ if $q \in \Sigma_2,$ then $q \in Q_{2,q}.$

%(i.e., $$ is defined as in (\ref{eqncpm}) with $\rho$ replaced by $\rho_i$ ).

\begin{proposition}\label{prpnpq4d5}
Let $F$ be a nonconstant holomorphic map from a neighborhood $U$ of some $p \in \Sigma_1$ to $\mathbb{C}^{n+1},$ and maps $U \cap \Sigma_1$ to $\Sigma_2$. If $q \in \mathbb{C}^{n+1}$ is such that $F(Q_{1,p} \cap U) \subseteq Q_{2,q}$, then $q=F(p).$
\end{proposition}

\begin{proof}
Since $F$ sends a piece of strongly pseudoconvex hypersurface to another, by a standard Hopf lemma type argument and shrinking $U$ if needed, we have  $F$ is a biholomorphic  map in $U$. %Writing $p=(z_p, \xi_p) \in \mathbb{C}^{n} \times \mathbb{C}$, since $p \in \Sigma_1,$ we have $\xi_p \neq 0.$ This yields $Q_{1,p}$ and thus 
We also note $Q_{1,p} \cap U$ is an $n-$dimensional complex manifold. Consequently, so is $F(Q_{1,p} \cap U)$. 
On the other hand, since $F$ maps $\Sigma_1 \cap U$ to $\Sigma_2$,  by a standard complexification argument and shrinking $U$ if needed,  we have $F(Q_{1, p} \cap U) \subseteq Q_{2, F(p)}.$ Write $F(p)=(z^*, \xi^*).$ Since $F(p) \in \Sigma_2,$ we have $\xi^* \neq 0.$ Since we also have $F(Q_{1,p} \cap U) \subseteq Q_{2,q}$,
by part (b) of Proposition \ref{prpnzsh}, we see $q=F(p).$
\end{proof}

We now come back to prove $F$ is rational, where $F$ is as given at the beginning of $\S$3. For convenience,
we set up a few notations. Recall we already know $F$ is algebraic. There is a complex algebraic variety $\mathcal{S}$ in $\mathbb{C}^{n+1}$ (the union of the branching locus and the set of points of indeterminacy of $F$) such that $(F, U)$ admits analytic continuation along every path in $\pi_1(\mathbb{C}^{n+1} \setminus \mathcal{S}, p_0)$ with $p_0 \in \Sigma_1 \cap U.$ Write $\mathcal{P}$ for the subgroup of $\pi_1(\mathbb{C}^{n+1} \setminus \mathcal{S}, p_0) $ that is generated by loops obtained by concatenating paths $\gamma_1, \gamma_2, \gamma_3,$ where $\gamma_1$ connects $p_0$ with a point arbitrarily close to a smooth point $q_0$ of $\mathcal{S}$ with $q_0 \not \in \{\xi=0\}$,  $\gamma_2$
is a loop around $\mathcal{S}$ near $q_0$ and $\gamma_3$ is $\gamma_1$ reserved.  %We will first prove that $F$, defined in a neighborhood of $p$, has no monodromy when extending along any loop in $\mathcal{P}$.  

Fix any $\gamma \in \mathcal{P}.$ 
We will show $F$ is {\it single-valued along $\gamma$}. This means, if we apply analytic continuation to $F$ along $\gamma,$ then we end up with the same germ of map as $F$ at $p_0$. To prove that, we write $F_1$ for the germ of holomorphic map at $p_0$ we get when applying analytic continuation to $F$ along $\gamma.$   By Proposition \ref{prpnptloop}, after slightly perturbing $\gamma$ if necessary, there exists a null-homotopic loop
$\lambda$ in $\pi_{1}(\mathbb{C}^{n+1} \setminus \mathcal{S},p_{0})$ with $(\lambda,\overline{\gamma})$ contained in the regular part
$\mathrm{Reg}(\mathcal{M}_1)$ of $\mathcal{M}_1.$ Since $F(U \cap \Sigma_1) \subseteq \Sigma_2,$  by complexification, we see $(F,\overline{F}):=(F(\cdot),\overline{F(\overline{\cdot})})$ sends a neighborhood
of $(p_{0},\overline{p_{0}})$ in $\mathcal{M}_1$ to $\mathcal{M}_2.$ Applying the analytic continuation along the loop
$(\lambda,\overline{\gamma})$ in $\mathcal{M}_1$ for $\rho(F,\overline{F}),$ one concludes by the uniqueness of analytic
functions
that $(F,\overline{F_1})$ also sends a neighborhood of $(p_{0},\overline{p_{0}})$ in $\mathcal{M}_1$ to $\mathcal{M}_2.$
Consequently, for every  $p \in \Sigma_1$ near $p_{0},$ we get $F(Q_{1, p} \cap U_p) \subseteq Q_{2, F_1(p)}$ for some neighborhood $U_p$ of $p$.   
%In particular,
%$F_{2}$ maps a neighborhood of $p_{0}$ in $M_{\varepsilon}$ into $M'.$
%we have the following:
%\begin{equation}\label{eqndiffb}
%F_{1}(p) \in Q'_{F(p)}, \forall p \in \Sigma_1,~ p\approx p_{0}.
%\end{equation}
%\begin{remark}
%Now applying the holomorphic continuation along the loop $(\lambda,\overline{\gamma})$ in $\mathcal{M}$ for $\rho(F_{2},F),$ we get by uniqueness of analytic functions that $(F_{2},\overline{F_{2}})$ sends a neighborhood of $(p_{0},\-{p_0})$ in $\mathcal{M}$ into $\mathcal{M}'.$ Hence, we also have
%\begin{equation}\label{eqdiffb1}
%F_{2}(p) \in Q'_{F_2(p)}, \forall p \in M,~ p\approx p_{0}.
%\end{equation}
%\end{remark}
By Proposition \ref{prpnpq4d5}, we see $F_1(p)=F(p)$ for $p\approx p_{0}$ on $\Sigma_1.$ This shows $F$ and $F_1$ define the same germ of map
at $p_0.$ Since $\gamma$ is arbitrary, we obtain the following
%$(F, p_0)$ and $(F_2, p_0)$ are identical. 

\begin{lemma}\label{lmnomop}
$F$ is single-valued along every loop in $\mathcal{P}$. 
\end{lemma}
 
We use the lemma to further  prove

\begin{proposition}\label{prpnmnpi}
$F$ is single-valued along every loop in $\pi_1(\mathbb{C}^{n+1} \setminus \mathcal{S}, p_0).$ Consequently, $F$ extends to a rational map.
\end{proposition}

{\bf Proof of Proposition \ref{prpnmnpi}.} We write $\mathcal{W}:=\{(z, \xi):\xi=0\} \subseteq \mathbb{C}^{n+1}$, and split the proof into two cases:

(a). Assume $\mathcal{W} \not \subseteq \mathcal{S}.$ By Proposition \ref{prpnhz37}, $\pi_1(\mathbb{C}^{n+1} \setminus \mathcal{S}, p_0)$ is generated by loops obtained by concatenating paths $\gamma_1, \gamma_2, \gamma_3,$ where $\gamma_1$ connects $p_0$ with a point arbitrarily close to a smooth point $q_0$ of $\mathcal{S}$, $\gamma_2$ is a loop around $\mathcal{S}$ near $q_0$ and $\gamma_3$ is $\gamma_1$ reserved. Since $\mathcal{W} \not \subseteq \mathcal{S},$ we can always perturb $\gamma_i'$s such that $q_0 \not \in \{\xi=0\}$. Consequently, $\pi_1(\mathbb{C}^{n+1} \setminus \mathcal{S}, p_0)=\mathcal{P}.$  Then by Lemma \ref{lmnomop}, $F$ is single-valued along every loop in $\pi_1(\mathbb{C}^{n+1} \setminus \mathcal{S}, p_0).$

(b). Next we assume $\mathcal{W} \subseteq \mathcal{S}$. Fix any loop $\alpha \in \pi_1(\mathbb{C}^{n+1} \setminus \mathcal{S}, p_0).$ %By Proposition \ref{prpnhz37}, we can assume $\alpha$ is obtained by concatenating paths $\alpha_1, \al_2, \al_3,$ where $\al_1$ connects $p$ with a point arbitrarily close to a smooth point $q_0$ of $\mathcal{S}$ with $q_0 \in \mathcal{W}$,  $\al_2$
%is a loop around $\mathcal{S}$ near $q_0$ and $\al_3$ is $\al_1$ reserved.
We will show  $F$ is single-valued along $\alpha.$    

Note the topological boundary of $D_1:=D_1(\Om_1, \rho_1)$ in $\mathbb{C}^{n+1}$ is given by $\partial D_1=\{(z, \xi) \in \Ol{\Omega_1} \times \mathbb{C}: |\xi|^2 = \rho_1(z, \overline{z}) \}$ and $\partial D_1 \cap \mathcal{W}=\{(z, 0): z \in \p \Om_1 \}.$ By Lemma \ref{lmstom}, there is an open subset $T$ of $\partial D_1 \cap \mathcal{W}$ such that every point in $T$ is a smooth and minimal boundary point of $D_1.$ Consequently, there exists some point $p^*=(z^*, 0) \in \mathcal{W}$ such that the following hold:

(I) $p^*$ is a smooth and minimal boundary point of $D_1;$ 

(II)  $p^*$ is a smooth point of $\mathcal{S}$;

(III) $\mathcal{W}$ intersects $\partial D_1$ transversally at $p^*$.

We can then find some small loop $\beta_0$ around $p^*$  in the real tangent plane $T_{p^*} (\partial D_1)$  such that $\beta_0$ stays away from $\mathcal{S}$ and generates the fundamental group $\pi_1(\mathbb{C}^{n+1} \setminus \mathcal{W}) \cong \mathbb{Z}.$
Furthermore, %we note any loop in $T^{(1,0)}_{p^*} \Sigma$ can be easily homotopically deformed into a loop on $\Sigma.$ In particular, 
$\beta_0$ can be deformed into a loop $\beta$ on $\partial D_1,$ via the natural local diffeomorphism between $T_{p^*} (\partial D_1)$ and $\partial D_1$ near $p^*$.  By choosing $\beta_0$ sufficiently close to $p^*,$ we can make
$\beta$ arbitrarily close to $p^*$ and stays away from $\mathcal{S}$. To make use of the loop $\beta$, we will need:

{\bf Claim.} Let $(F, U)$ be as at the beginning of $\S$3 and let $\mathcal{S}$ be as above. Then there exists a complex hypervariety $\mathcal{S}_1$ of $\mathbb{C}^{n+1}$ such that for any $q \in \Sigma_1 \setminus (\mathcal{S} \cup \mathcal{S}_1)$ , there is a path $\gamma$ in $\Sigma_1 \setminus (\mathcal{S} \cup \mathcal{S}_1)$ from $p$ to $q$ such that $F$ extends holomorphically along $\gamma$ and the analytic continuation of $F$ along $\gamma$ sends (open pieces of) $\Sigma_1$ to $\Sigma_2.$

{\bf Proof of Claim.} %(This is the last claim in the handwritten notebook.) 
We first note since $\Sigma_1$ is connected, so is $\Sigma_1 \setminus \mathcal{S}.$ Then for any $q \in  \Sigma_1 \setminus \mathcal{S},$ there is a path $\tau$ in $\Sigma_1 \setminus \mathcal{S}$ connecting $p$ and $q$. We can therefore continue  $(F, U)$ holomorphically along $\tau$ to get a holomorphic map at $q$. Write $\{F_{q, 1}, \cdots, F_{q, m}\}, m \geq 1,$ for all the germs of maps at $q$ that can be obtained from such holomorphic continuations of $(F,U)$ along paths in $\Sigma_1 \setminus \mathcal{S}$ (by the connectedness of $\Sigma_1 \setminus \mathcal{S}$, the value of $m$ is independent of the choice of $q$).  For each $j,$ write $F_{q, j}=(F_{q, j}^1, \cdots, F_{q, j}^{n}, F_{q, j}^{n+1})$ and $G_q=\prod_{j=1}^m F_{q, j}^{n+1}.$ Note $G_q$ is well-defined for every $q \in \Sigma_1 \setminus \mathcal{S},$ this induces a holomorphic map, denoted by $G$, in a neighborhood $O$ of $\Sigma_1 \setminus \mathcal{S}$  Note for each $j$, the last component $F_{q, j}^{n+1}$ of $F_{q, j}$ cannot be identically zero, for if it were identically zero, then the same would hold for $F$ (as $F_{q, j}$ is obtained from $F$ by analytic continuation), a plain contradiction. Consequently, $G$ is
not identically zero. Hence $\{G=0\}$ defines a proper complex subvariety (if not empty) of $O$, which will be denoted by $\mathcal{S}_1.$

Since $\Sigma_1 \setminus (\mathcal{S} \cup \mathcal{S}_1)$ is also connected, if $q \in \Sigma_1 \setminus (\mathcal{S} \cup \mathcal{S}_1),$ then there is a path $\gamma$ in $\Sigma_1 \setminus (\mathcal{S} \cup \mathcal{S}_1)$ connecting $p$ and $q$. 
%Finally to prove the claim, we only need to show  the image of the extension of $F$ along $\gamma$, still denoted by $F$, stays in $\Sigma_2.$ To see this, 
Next since $\gamma$ does not intersect $\mathcal{S}_1,$ the image of $F$ along $\gamma$ cannot intersect $\mathcal{W}.$ On the other hand, as  $F$ maps $U\cap \Sigma_1$ to $\Sigma_2,$ by analyticity,  the germ of map $(F_t, U_t)$ at every $\gamma(t)$ obtained by analytic continuation must satisfy $F_t(U_t \cap \Sigma_1) \subseteq X:=\{(z, \xi) \in \mathbb{C}^{n+1}: |\xi|^2= \rho_2(z, \ol{z}) \}.$ Consequently, $F_t(U_t \cap \Sigma_1)$ must be lie in the connected component of $X \setminus \mathcal{W}$ containing $F(U\cap \Sigma_1),$ which is precisely $\Sigma_2.$
%Now $F \circ \gamma$ is a curve in $\mathbb{C}^{n+1}$ with initial point in $\Sigma_2$ satisfies $r_2(F \circ \gamma, \Ol{F \circ \gamma}) \equiv 0$ and cannot intersect with the boundary of $\Sigma_2$ (which is $\partial D_2 \cap \mathcal{W}.$) Therefore $F \circ \gamma$ must stay in $\Sigma_2.$ 
This proves the claim. \qed

\smallskip

By the claim above,  we can choose a path $\tau$ connecting $p_0$ to a point $\hat{p}$ arbitrarily close to $p^* \in \mathcal{W}$ such that the germ of map $\hat{F}$ at $\hat{p}$ obtained by analytic continuation of $F$ along $\tau$ also sends (an open piece of) $\Sigma_1$ to $\Sigma_2.$ Since $F$ is nonconstant, so is $\hat{F}.$  Then replacing $(F, p_0)$ by $(\hat{F}, \hat{p})$ if needed,  we can make $p_0$ arbitrarily close to $p^* \in \mathcal{W}$. Also by perturbing $\beta$ and $p_0$, we can indeed assume $p_0 \in \beta.$ 
Note  $\alpha \subseteq A:=\mathbb{C}^{n+1} \setminus \mathcal{W}$,
%and they do not intersect  $\mathcal{S}$. 
and $\beta$ is a generator of $\pi_1(A, p_0) \cong \mathbb{Z}.$ Hence 
there exists $k \in \mathbb{Z},$ such that %$\al^{-1} \circ \be^k$ %(and 
$\be^{-k} \circ \al$
is null homotopic in $\pi_1(A, p_0)$.
Note also $\al, \be \in \pi_1(A \setminus \mathcal{S}, p_0)=\pi_1(\mathbb{C}^{n+1} \setminus \mathcal{S}, p_0)$ and let $\mathcal{P}$ be as defined right after Proposition \ref{prpnpq4d5}.
%and  (we choose $\al$ and $\be$ both to travel around $\mathcal{W}$ only once). 
Applying part (1) of Proposition \ref{prpnhz37} (with  $U=A,\mathcal{V}=A \cap \mathcal{S}$ and $p=p_0$ there) and a similar argument as in part (a), we have $\be^{-k} \circ \al \in \mathcal{P}$.
%as an element in $\pi_1(\mathbb{C}^{n+1} \setminus \mathcal{S}, p_0),$ belongs to $\mathcal{P}$.
%$\be^{-k} \circ \al \in \mathcal{P}.$ 
Hence there exists some $\mu \in \mathcal{P}$ such that $\al =\be^k \circ \mu.$ By Lemma \ref{lmnomop}, $F$ is single-valued along $\mu$. Consequently,  to prove $F$ is single-valued along $\alpha$, it suffices to show $F$ is single-valued along $\be$.
%we get a different branch when we apply holomorphic continuation of  $F$ along $\al$ if and only if we get a different branch when we apply holomorphic continuation of  $F$ along $\be.$ Assume at first $F$, which is a holomorphic map defined a neighborhood $V$ of $p$, sending $V \cap \Sigma_1$ to $\Sigma_2.$ 
For that, we denote by $F_1$ the germ of holomorphic map (at $p_0$) we get when applying analytic continuation to  $F$ along $\be$. We will show $F$ and $F_1$ define the same germ of maps at $p_0$. Shrinking $U$ if needed, we assume $F_1$ is also defined on $U$.

Recall $p^*$ is a smooth and minimal boundary point of $D_1$. We take a small neighborhood $\hat{O}$ of $p^*$ in $\mathbb{C}^{n+1}$ such that $\partial D_1 \cap \hat{O}$ is a smooth real analytic hypersurface. Note Proposition 3.10 in Huang-Zaitsev \cite{HZ} can be applied to (the complex submanifold  $\mathcal{F}$ of $(\mathbb{C}^{n+1} \setminus \mathcal{S}) \times \mathbb{C}^{n+1}$ corresponding to) the analytic extension of $F$ in $\mathbb{C}^{n+1} \setminus \mathcal{S}$. (For the readers' convenience, we remark that $p_0, \Om, V, M$ and $M'$ in Proposition 3.10 of \cite{HZ} can be taken to be $p^*,$ $\mathbb{C}^{n+1},\mathbb{C}^{n+1}, \partial D_1 \cap \hat{O}$ and $\{(z, \xi) \in \mathbb{C}^{n+1}: |\xi|^2= \rho_2(z, \ol{z}) \}$, respectively in our setting.) Therefore
there is a neighborhood $W$ of $p^*$ (depending only on $\p D_1 \cap \hat{O}$ and $p^*$) such that the conclusion of Proposition 3.10 in \cite{HZ} holds. By making $p_0$ sufficiently close to $p^*$ and further shrinking $\beta$ and $U$ if needed, we can assume $\beta$ and $U$ are in $W$. Then by the conclusion of Proposition 3.10 in \cite{HZ},   for every $q \in \Sigma_1$ near $p_0$, it holds that $F(Q_{1, q} \cap U_q) \subseteq Q_{2, F_1(q)}$ for some small neighborhood $U_q$ of $q$. This implies, by Proposition \ref{prpnpq4d5}, $F_1(q)=F(q)$ for $q \approx p_0$ on $\Sigma_1,$ and therefore $F$ and $F_1$ give the same germ of map at $p_0$. Hence $F$ is single-valued along $\beta$, and thus along $\al$. Since $\al$ is arbitrary, we see $F$ is single-valued along every loop in $\pi_1(\mathbb{C}^{n+1} \setminus \mathcal{S}, p_0).$

By parts (a) and (b), the analytic continuation of $F$ is single-valued in any case. Since $F$ is algebraic, we see $F$ must be rational. This finishes the proof of Proposition \ref{prpnmnpi}. \qed

\begin{remark}
We remark that in the course of proving the rationality of $F$, we only used the condition that $\rho_1$ is a generic norm of $\Om_1$ (in particular to make Lemma \ref{lmstom} hold), and the assumption on $\rho_2$ can be relaxed to be merely a generalized generic norm of $\Om_2$. But we need $\rho_2$ to be a generic norm to achieve the rationality of $F^{-1}$.
\end{remark}

As $F$ is biholomorphic near $p_0$,  $F$ has a local inverse $F^{-1}$ defined in some neighborhood $W$ of $F(p_0)$.  Since  $F^{-1}$ maps $W \cap \Sigma_2$ to $\Sigma_1$, by applying the same argument as above to $F^{-1}$, we get $F^{-1}$ extends to a rational map as well.

\subsection{Metric-preserving property of the map and finishing the proof}

In $\S3.1,$ we have proved that $F$ and $F^{-1}$  are rational. There exists an affine algebraic complex subvariety $\mathcal{V}_1$ (respectively, $\mathcal{V}_2$) of $\mathbb{C}_{(z, \xi)}^{n+1},$ which precisely consists of points $p \in \mathbb{C}_{(z, \xi)}^{n+1}$ at which $F$ (respectively, $F^{-1}$) fails to be a local holomorphic immersion. Since $F$ is biholomorphic on $U$, we have $U \cap \mathcal{V}_1 =\emptyset.$  It is clear that $F: \mathbb{C}^{n+1} \setminus \mathcal{V}_1 \rightarrow \mathbb{C}^{n+1} \setminus \mathcal{V}_2$ is a biholomorphic map, with $F^{-1} \circ F$ equals the identity map on  $\mathbb{C}^{n+1} \setminus \mathcal{V}_1$ and $F \circ F^{-1}$ equals the identity map on $\mathbb{C}^{n+1} \setminus \mathcal{V}_2.$ We pause to introduce the following lemma.

\begin{lemma}\label{lmhzx}
	Let $\mathcal{V}$ and $\hat{\mathcal{V}}$ be affine algebraic complex subvarieties of $\mathbb{C}^{n+1}.$ Let $h \in \mathbb{C}[Z, X]$, with $Z, X \in \mathbb{C}^{n+1},$ be such that its zero locus satisfies 
	$$\{(Z, X) \in \mathbb{C}^{n+1} \times \mathbb{C}^{n+1}: h(Z, X)=0\} \subseteq (\mathcal{V} \times \mathbb{C}^{n+1}) \cup (\mathbb{C}^{n+1} \times \hat{\mathcal{V}}).$$
	Then $h(Z, X)=h_1(Z)h_2(X)$ for some $h_1 \in \mathbb{C}[Z]$ and $h_2 \in \mathbb{C}[X].$ Moreover, $\{h_1(Z)=0\} \subseteq \mathcal{V}$ and $\{h_2(X)=0\} \subseteq \hat{\mathcal{V}}.$
\end{lemma}

\begin{proof}
	Assume $h$ is nonconstant.  Let $g$ be an irreducible factor of $h$. Then the irreducible subvariety  $\{g(Z, X)=0\}$ must be an irreducible component of either $\mathcal{V} \times \mathbb{C}^{n+1}$ or $\mathbb{C}^{n+1} \times \hat{\mathcal{V}}.$
	In the former case we have $g \in \mathbb{C}[Z]$; in the latter case $g \in  \mathbb{C}[X].$ This proves the lemma.
\end{proof}

We also pause to fix some notations. For $Z \in \mathbb{C}^m$,  denote by $\Ol{Z}$ the vector obtained from $Z$ by taking the conjugate of each component.  Again for $V \subseteq  \mathbb{C}^{m}$,  denote by $\Ol{V}$ the set $\{\Ol{Z} \in \mathbb{C}^{m}: Z \in V \};$ note if $V$ is an affine algebraic complex subvariety, then so is $\Ol{V}$. 
For a map $H: V \subseteq \mathbb{C}^m \rightarrow \mathbb{C}^k$, define a map $\Ol{H}$ from $\Ol{V}$ to $\mathbb{C}^k$ by $\Ol{H}(\xi):=\Ol{H(\Ol{\xi})}$.

We continue to discuss the map $F$. Write $W_1:= (\mathbb{C}_{(z, \xi)}^{n+1} \setminus \mathcal{V}_1) \times (\mathbb{C}_{(w, \eta)}^{n+1} \setminus \Ol{\mathcal{V}_1})$ and $W_2:= (\mathbb{C}_{(z, \xi)}^{n+1} \setminus \mathcal{V}_2) \times (\mathbb{C}_{(w, \eta)}^{n+1} \setminus \Ol{\mathcal{V}_2}).$ Write $Z=(z, \xi)$ and $X=(w, \eta).$
Define a rational map $\mathcal{F}$ by  $\mathcal{F}(Z, X):=(F(Z), \Ol{F}(X))$ with $Z \in \mathbb{C}^{n+1} \setminus \mathcal{V}_1$ and $X \in \mathbb{C}^{n+1} \setminus \Ol{\mathcal{V}_1}.$
Then $\mathcal{F}$ is a biholomorphism from $W_1$ to $W_2$. By Proposition \ref{prpnrir}, the polynomial $r_i$ as defined in (\ref{eqncpmi}) is irreducible in $\mathbb{C}[z, w, \xi, \eta]$ for $1 \leq i \leq 2.$ Consequently, the regular part $\mathrm{Reg}(\mathcal{M}_i)$ of $\mathcal{M}_i$ is connected (where $\mathcal{M}_i$ is as defined in (\ref{eqncpmi})). Moreover, $\mathrm{Reg}(\mathcal{M}_i) \cap W_1$ is nonempty (cf. Lemma \ref{lmhzx}), and connected. Since $F$ maps an open piece of $\Sigma_1$ to $\Sigma_2,$  by a standard complexification argument and the connectedness of $\mathrm{Reg}(\mathcal{M}_1) \cap W_1$, we have $\mathcal{F}=(F, \Ol{F})$ maps $\mathcal{M}_1 \cap W_1$ to $\mathcal{M}_2 \cap W_2, $ and maps $\mathrm{Reg}(\mathcal{M}_1) \cap W_1$ to $\mathrm{Reg}(\mathcal{M}_2) \cap W_2.$ One can readily see the latter two maps are both onto by considering the inverse map $(F^{-1}, \Ol{F^{-1}}).$   

We next prove the following lemma. Write $R(Z, X):= \frac{r_2(F(Z), \Ol{F}(X))}{r_1(Z, X)}.$ Since $F$ and $\Ol{F}$ are rational in $Z$ and $X$ respectively, it is clear that $R(Z, X)$ is rational in $Z$ and $X$.

\begin{lemma}\label{lmrzx}
%We have the following statements hold:
The function $R(Z, X)$ is holomorphic and everywhere nonzero  in $W_1$. Consequently, 
writing $R(Z, X)=\frac{P(Z, X)}{Q(Z, X)}$ with $P, Q \in \mathbb{C}[Z, X]$ coprime, we have $P\neq 0$ and $Q \neq 0$ everywhere in $W_1$.
Furthermore, $P(Z, X)=P_1(Z)P_2(X)$ and $Q(Z,X)=Q_1(Z)Q_2(X)$ for some polynomials $P_1, Q_1 \in \mathbb{C}[Z]$ and $P_2, Q_2 \in \mathbb{C}[X],$
with $\{P_1=0\}, \{Q_1=0\} \subseteq \mathcal{V}_1$ and $\{P_2=0\}, \{Q_2=0\} \subseteq \ol{\mathcal{V}_1}.$
\end{lemma}

\begin{proof}
Recall $\mathcal{F}=(F, \Ol{F})$ is a biholomorphism from $W_1$ to $W_2$; it also maps $\mathcal{M}_1 \cap W_1$ and $\mathrm{Reg}(\mathcal{M}_1) \cap W_1$ onto $\mathcal{M}_2 \cap W_2$ and $\mathrm{Reg}(\mathcal{M}_2) \cap W_2,$ respectively.
Fix any
$(Z_0, X_0) \in W_1 \setminus \mathcal{M}_1$, which then satisfies $r_1(Z_0, X_0) \neq 0$. We must have  $\mathcal{F}(Z_0, X_0) \not \in \mathcal{M}_2 \cap W_2,$ i.e., $r_2(F(Z_0), \Ol{F}(X_0)) \neq 0.$ This yields that $R(Z, X)$ is holomorphic and nonzero on $W_1 \setminus \mathcal{M}_1.$
Next as $r_i(Z, X)$ is a defining function of $\mathrm{Reg}(\mathcal{M}_i), 1 \leq i \leq 2$, and $r_2(F(Z), \Ol{F}(X))$ vanishes on $W_1 \cap \mathcal{M}_1$, it follows that
$R(Z, X)$ is holomorphic in a neighborhood of $\mathrm{Reg}(\mathcal{M}_1)\cap W_1$. As $\mathcal{F}$ is biholomorphic on $W_1$, we check the gradient of both sides of  $r_2(F(Z), \Ol{F}(X))=r_1(Z, X) R(Z, X)$ to see $R(Z, X) \neq 0$ on $\mathrm{Reg}(\mathcal{M}_1) \cap W_1$.
Hence $R(Z, X)$ is holomorphic and nonzero everywhere in $W_1 \setminus \mathrm{Sing}(\mathcal{M}_1)$. Consequently, if $R(Z, X)=\frac{P(Z, X)}{Q(Z, X)},$ where $P, Q \in \mathbb{C}[Z, X]$ are coprime, then  $P\neq 0$ and $Q \neq 0$ everywhere in $W_1 \setminus \mathrm{Sing}(\mathcal{M}_1)$. But $\mathrm{Sing}(\mathcal{M}_1)$ has complex codimension at least $2$ in $\mathbb{C}^{n+1} \times \mathbb{C}^{n+1}$. Therefore $P\neq 0$ and $Q \neq 0$ everywhere in $W_1$. That is, $\{P=0\}, \{Q=0\} \subseteq (\mathcal{V}_1 \times \mathbb{C}^{n+1}) \cup (\mathbb{C}^{n+1} \times \Ol{\mathcal{V}_1}).$ Then the last assertion of the lemma follows from Lemma \ref{lmhzx}.
\end{proof}

By Lemma \ref{lmrzx}, $R(Z, X)=R_1(Z)R_2(X)$ for some rational functions $R_1$ and $R_2.$ Here $R_1(Z)$ is holomorphic and nowhere zero in $\mathbb{C}^{n+1} \setminus \mathcal{V}_1$, and $R_2(X)$ is holomorphic and nowhere zero in $\mathbb{C}^{n+1} \setminus \Ol{\mathcal{V}_1}$. By the definition of $R$,  we have $r_2(F(Z), \Ol{F}(X))=r_1(Z, X)R_1(Z)R_2(X).$ Setting $X=\Ol{Z} $ in the equation, we get $r_2(F(Z), \Ol{F(Z)})=r_1(Z, \Ol{Z})R_1(Z)R_2(\Ol{Z})$ for $Z$ in $\mathbb{C}^{n+1} \setminus \mathcal{V}_1.$ Consequently, $R_1(Z)R_2(\Ol{Z})$ is real valued in $\mathbb{C}^{n+1} \setminus \mathcal{V}_1$. This yields, by elementary complex analysis, $R_2(\Ol{Z})=c \Ol{R_1(Z)}$ for some nonzero real number $c$. Furthermore, since $F$ maps $U \cap D_1$ (where $r_1(Z, \ol{Z})>0$) to $D_2$ (where $r_2(Z, \ol{Z})>0$), we must have $c >0.$ Writing $R_3(Z)=\sqrt{c} R_1(Z),$ we have $R_3(Z) \neq 0$ in $\mathbb{C}^{n+1} \setminus \mathcal{V}_1$, and
\begin{equation}\label{eqnr2f}
r_2(F(Z), \Ol{F(Z)})=r_1(Z, \Ol{Z}) |R_3(Z)|^2~\text{on}~\mathbb{C}^{n+1} \setminus \mathcal{V}_1.
\end{equation}

Let $\om_{D_i}=-\sqrt{-1} \partial \overline{\partial} \log r_i(Z, \Ol{Z}), 1 \leq i \leq 2,$ be the canonical metrics of $D_i$ as defined in $\S$2.3. On $U \cap D_1 \subseteq \mathbb{C}^{n+1} \setminus \mathcal{V}_1$, both sides of (\ref{eqnr2f}) are positive. We take the logarithm of (\ref{eqnr2f}), and then apply $-\partial \Ol{\partial}$ to see $F$ is isometric: $F^*(\omega_{D_2})=\omega_{D_1}$ in $U \cap D_1$. Note $\Om_i$ is simply connected, and so is $D_i$ for $1 \leq i \leq 2.$ Applying the classical extension result in Riemannian Geometry (cf. Proposition 11.3 and 11.4 \cite{He}), we conclude $F$ extends to a global holomorphic isometric map, still denoted by $F$,  from $(D_1, \omega_{D_1})$ to $(D_2, \omega_{D_2})$. We then apply the same argument in the above to $F^{-1}$ (which is initially defined on $F(U)$) to see $F^{-1}$ extends to a global holomorphic isometric map, denoted by $G$,  from  $(D_2, \omega_{D_2})$ to $(D_1, \omega_{D_1})$. It is clear $F \circ G$ equals the identity map in $D_2$ and $G \circ F$ equals the identity map in $D_1.$ Hence $F$ is a biholomorphic isometry from $(D_1, \omega_{D_1})$ to $(D_2, \omega_{D_2})$ (this in particular yields $D_1 \cap \mathcal{V}_1=\emptyset$). This proves part (1) of Theorem \ref{T1}.

To prove part (2) of Theorem \ref{T1}, we first consider the case $\mathrm{rank}(\Om_1)=\mathrm{rank}(\Om_2)=1$. In this case,  both $\Om_1$ and $\Om_2$ are the unit ball $\mathbb{B}^n.$ Since $\rho_1$ and $\rho_2$ both equal the generic norm of $\mathbb{B}^n,$  we have $(\Om_1, \omega_{\Om_1})$ and $(\Om_2, \omega_{\Om_2})$ are holomorphically isometric. In the remaining cases, it holds that $\mathrm{rank}(\Omega_1) \geq 2$ or $\mathrm{rank}(\Omega_2) \geq 2.$ The conclusion in part (2) is then a consequence of part (3) and Remark \ref{rmkform}, which we will prove in the following.
For that, we first establish the next lemma. As before, $Z=(z, \xi)=(z_1, \cdots, z_n, \xi)$ denotes the coordinates of $\mathbb{C}^{n+1}.$

\begin{lemma}\label{lmf0}
Assume $\mathrm{rank}(\Omega_1) \geq 2$ or $\mathrm{rank}(\Omega_2) \geq 2.$ It holds that $F$ maps $\Omega_1 \times \{0\}$  to $\Omega_2 \times \{0\}.$ Likewise, $F^{-1}$ maps $\Omega_2 \times \{0\}$  to $\Omega_1 \times \{0\}$ (in particular the restriction $F: \Omega_1 \times \{0\} \rightarrow \Omega_2 \times \{0\}$ is biholomorphic). Consequently, writing $F=(f, f_{n+1})=(f_1, \cdots, f_n, f_{n+1}),$ we have $f_{n+1}(z, 0) \equiv 0$ for all $z \in \Om_1,$ and $f(z, 0)$ is a biholomorphism from $\Om_1$ and $\Om_2$.
\end{lemma}

{\bf Proof of Lemma \ref{lmf0}.} First assume $\mathrm{rank}(\Omega_1) \geq 2$.  Write the rational map $F(Z, X)=\frac{P(Z, X)}{Q(Z, X)}$, where $Q$ and the components of $P$ are in $\mathbb{C}[Z, X]$, and $\frac{P}{Q}$ is reduced to the lowest term. Write 
$JF$ for the determinant of the complex Jacobian matrix of $F$ (which also rational). Then $Q \neq 0, JF \neq 0$ everywhere in $D_1$.  In particular, they are everywhere nonzero on $\Omega_1 \times \{0\}.$ Recall by Lemma \ref{lmstom}, there exists an open piece $A$ of the boundary of $\Om_1$ such that every point in $\hat{A}:=A \times \{0\}$ is a weakly pseudoconvex (smooth) boundary point of $D_1.$ By the uniqueness of holomorphic functions, shrinking $A$ if needed, we can assume $F$ extends holomorphically across $\hat{A}$, with $Q \neq 0$ and $JF \neq 0$ on $\hat{A}.$ Consequently, by further shrinking $A$, the map $F$ is biholomorphic in some neighborhood $O \subseteq \mathbb{C}^{n+1}$ of $\hat{A}$. Then $F$ cannot map $\hat{A}$ to the strongly pseudoconvex boundary $\Sigma_2,$ and we must have $F(\hat{A}) \subseteq \p \Omega_2 \times \{0\}.$ This implies $F$ maps $\Omega_1 \times \{0\}$  to $\Omega_2 \times \{0\}.$ Since $F$ is biholomorphic in $D_1$, it is clear that $F^{-1}$ must map $\Omega_2 \times \{0\}$  to $\Omega_1 \times \{0\}.$ This finishes the proof of this case.

In the case $\mathrm{rank}(\Omega_2) \geq 2,$ we apply the same argument as above to $F^{-1}$ to obtain the assertion.  This finishes the proof. \qed

\smallskip

%Next by part (1), we know (\ref{eqnr2f}) holds on $D_1$. 
We continue to prove Theorem \ref{T1}. Write $g(z)=f(z, 0)$ and $h(z)=R_3(z, 0)$. Then by the above,  $g$ is a biholomorphism from $\Om_1$ to $\Om_2,$ and $h$ is an everywhere nonzero holomorphic function on $\Om_1$.  Recall $r_i(Z, \Ol{Z})=\rho_i(z,\Ol{z})-|\xi|^2, 1 \leq i \leq 2.$  We let $\xi=0$ in (\ref{eqnr2f}) and use Lemma \ref{lmf0} to obtain
$$\rho_2(g(z), \Ol{g(z)})=\rho_1(z,\Ol{z}) |h(z)|^2, \quad z \in \Om_1.$$
Next set
\begin{equation}\label{eqngz}
G(Z)=(g(z), \xi h(z)).
\end{equation}
It is clear that $G$ is a norm-preserving holomorphic isomorphism between the line bundles  $L_1=\Omega_1 \times \mathbb{C}$  and $L_2=\Omega_2 \times \mathbb{C}$ (with respect to the Hermitian metrics $h_1=\rho_1^{-1}$ and $h_2=\rho_2^{-1}$). In particular, $G$ gives a biholomorphism from $D_1$ to $D_2$. Moreover, the inverse map of $G$ is $G^{-1}(z, \xi)=(g^{-1}(z), \frac{\xi}{h(g^{-1}(z))})$.
We finally consider the map $H=G^{-1} \circ F: D_1 \rightarrow D_1.$
By part (1) of Theorem \ref{T1} and what we just discussed, $H$ is an automorphism of $D_1$. Then by part (2) of Proposition \ref{prpnauto}, $H$ extends to a norm-preserving holomorphic self-isomorphism of $(L_1, h_1)$.
Consequently, $F=G\circ H$ is a norm-preserving holomorphic isomorphism between the line bundles $(L_1, h_1)$ and $(L_2,h_2)$. This establishes part (3) of Theorem \ref{T1}.
Furthermore, from (\ref{eqngz}) and the form of $H$ (cf. part (1) of Proposition \ref{prpnauto}), we see $F$ take the form as in Remark \ref{rmkform}: 
$F(z, \xi)=(\psi(z),\xi \phi(z))$, where $\psi$ and $\phi$ are as described in Remark \ref{rmkform} and satisfy the equation $\rho_2(\psi(z), \Ol{\psi(z)})=|\phi(z)|^2 \rho_1(z, \Ol{z}).$ To see $\psi$ is an isometry from $(\Omega_1, \omega_{\Omega_1})$ to $(\Omega_2, \omega_{\Omega_2})$, we take the logarithm of this equation, and then apply the $-\partial \Ol{\partial}$ operator.

%We then take the logarithm of the above equation, and then apply $\partial \Ol{\partial}$ with respect to the $z-$variables to see $g(z)$ is isometric: $g^*(\omega_{\Om_2})=\omega_{\Om_1}$ on $\Om_1$. 

\section{Proof of Corollaries}

It remains to prove the corollaries presented in $\S$1. We start with Corollary \ref{corpdk} and \ref{corglt}.

{\bf Proof of Corollary \ref{corpdk}.} First we note $\rho_1$ and $\rho_2$ are  generic norms of $\Om_1$ and $\Om_2$ respectively, as at least one $p_i$ and at least one $q_i$ equal $1$. Then applying Theorem \ref{T1},  we see $(\Delta^n, \om_1=-\sqrt{-1} \partial \Ol{\partial} \log \rho_1)$ and $(\Delta^n, \om_2=-\sqrt{-1} \partial \Ol{\partial} \log \rho_2)$ are holomorphically isometric. This in particular implies that $(\Delta^n, \om_1)$ and $(\Delta^n, \om_2)$ have the same set of Ricci eigenvalues. Here the latter refers to the eigenvalues of the Ricci tensor $\mathrm{Ric}_i=-\sqrt{-1} \partial \Ol{\partial}\log \om_i^n$ with respect to $\om_i$ (see \cite{EXX2} for more details). On the other hand, it is routine to verify that the Ricci eigenvalues of $(\Delta^n, \om_1)$ and $(\Delta^n, \om_2)$ are given by $(-\frac{2}{p_1}, \cdots, -\frac{2}{p_n})$ and $(-\frac{2}{q_1}, \cdots, -\frac{2}{q_n})$ respectively. Hence $(p_1, \cdots, p_n)$ and $(q_1, \cdots, q_n)$ are equal up to a permutation.
The last assertion follows from part (3) of Theorem \ref{T1}. \qed

\smallskip

%We next prove Corollary \ref{corglt} and \ref{corac}. 

{\bf Proof of Corollary \ref{corglt}.} Write $\Sigma_i=\Sigma(\Om_i, \rho_i)=\{(z, \xi) \in \Om_i \times \mathbb{C}: |\xi|^2=\rho_i(z, \overline{z}) \}$ and define $\Phi_i(z, \xi)=(z, \xi^{\frac{1}{\lambda}})$ for each $1 \leq i \leq 2$. Here by shrinking $M$, we can choose  branches of $\xi^{\frac{1}{\lambda}}$ so that $\Phi_1$ and $\Phi_2$ are biholomorphic maps in a small neighborhood $U_1$ of $M$ and a small neighborhood $U_2$ of $G(M)$. Note $\Phi_i$ maps $U_i \cap M_i$ to $\Sigma_i.$ Write $M'=\Phi_1(M),$ which is an open connected piece of $\Sigma_1,$ and define $F=\Phi_2 \circ G \circ \Phi_1^{-1}$ on $M'$. It is clear that $F$ is a nonconstant continuous CR map from $M'$ to $\Sigma_2.$ Then by part (2) of Theorem \ref{T1}, $\Om_1$ and $\Om_2$ are biholomorphic. Furthermore, by part (3) of Theorem \ref{T1} and Remark \ref{rmkform}, $F$ extends to a norm-preserving holomorphic isomorphism between the line bundles  $(\Omega_1 \times \mathbb{C}, \rho_1^{-1})$  and $(\Omega_2 \times \mathbb{C}, \rho_2^{-1})$. Moreover,  
$F(z, \xi)=(\psi(z),\xi \phi(z))$, where $\psi$ is a biholomorphic map from $\Om_1$ to $\Om_2$, and $\phi$ is an everywhere nonzero holomorphic function on $\Om_1$ satisfying $\rho_2(\psi(z), \Ol{\psi(z)})=|\phi(z)|^2 \rho_1(z, \Ol{z}).$ Consequently, $G$ extends to a biholomorphic map $H=\Phi_2^{-1} \circ F \circ \Phi_1$ in a small neighborhood of $M$. By the form of $\Phi_1, \Phi_2$ and $F$, we see
$H=(\psi(z),(\xi^{\frac{1}{\lambda}}\phi(z))^{\lambda})$ for some branches of $(\cdot)^{\frac{1}{\lambda}}$ and $(\cdot)^{\lambda}.$ Since $\phi$ is nowhere zero in the simply connected domain $\Om_1$, there is a well-defined branch of $\lambda-$power of $\phi$ in $\Om_1$, denoted by $(\phi(z))^{\lambda}.$ Note $(\xi^{\frac{1}{\lambda}}\phi(z))^{\lambda}=e^{i\theta}\xi(\phi(z))^{\lambda}$ for some real number $\theta.$
Consequently, $H=(\psi(z),e^{i\theta}\xi(\phi(z))^{\lambda})$ extends to a holomorphic map on $\Om_1 \times \mathbb{C}.$ Moreover, it is clear that $H$ is a  norm-preserving holomorphic isomorphism between the Hermitian line bundles $(\Om_1 \times \mathbb{C}, h_1)$ and $(\Om_2 \times \mathbb{C}, h_2)$. This finishes the proof. \qed

\smallskip

We next apply Corollary \ref{corglt} to prove Corollary \ref{corac}.

\smallskip

{\bf Proof of Corollary \ref{corac}.} To prove part (1), we note that $M_i, 1 \leq i \leq 2,$ is given by, up to a constant scaling of $\xi,$ $\{(z, \xi) \in \Om_i \times \mathbb{C}: |\xi|^2=\big( \rho_i(z, \overline{z}) \big)^{g_i} \}$, where $g_i$ is the genus of $\Om_i.$ Moreover, $M_i$ is locally CR diffeomorphic  to $\Sigma_i=\{(z, \xi) \in \Om_i \times \mathbb{C}: |\xi|^2=\rho_i(z, \overline{z}) \}.$ Consequently,  if there exists a CR diffeomorphism between some open pieces of $M_1$ and $M_2$, then there is also a CR diffeomorphism between some open pieces of $\Sigma_1$ and $\Sigma_2.$ By part (2) of Theorem \ref{T1}, $\Omega_1$ and $\Omega_2$ are biholomorphic (and thus $g_1=g_2$). This proves part (1) as the ``only if" part is trivial.
Part (2) follows from Corollary \ref{corglt} by setting $\lambda=g_1=g_2$. \qed %(by part (1), $g_1$ and $g_2$ must be equal). \qed

\smallskip

We then prove Corollary \ref{corpr}.

\smallskip

{\bf Proof of Corollary \ref{corpr}.} %Let $\Om$ be an irreducible bounded symmetric domain (which is convex and contains $0$) with generic norm $N(z, \ol{z}), z \in \Om.$
%Fixing $z_0 \in \Om.$ By Proposition 2.1 in \cite{TW}, the function $N(rz_0, r\ol{z_0})$ is decreasing in $r \in [0, 1].$ As a consequence of this fact, 
Fix $z^* \in \Om_1.$ By Proposition 2.1 in \cite{TW}, the function $\rho_1(r z^*, r \ol{z^*})$ is decreasing in $r\in [0, 1].$ As a result, if $(z^*, \xi^*) \in \Om_1 \times \mathbb{C}$ is in $D_1$, i.e., $|\xi^*|^2 < \big(\rho_1(z^*, \overline{z^*})\big)^{\lambda_1},$ then for $r\in [0, 1]$, we have
$r^2 |\xi^*|^2 \leq |\xi^*|^2 < \big(\rho_1(z^*, \overline{z^*})\big)^{\lambda_1} \leq  \big(\rho_1(rz^*, r\overline{z^*})\big)^{\lambda_1}.$
Hence $(rz^*, r\xi^*) \in D_1,$ and thus $rD_1:=\{(rz, r\xi): (z, \xi) \in D_1\} \subseteq D_1$ for all $r\in [0, 1].$ Then, as $D_1$ is complete circular, by Lemma 1.1.1 of \cite{M2}, for any compact $E \subseteq D_1,$ there is a neighborhood $V$ of $\ol{D_1}$ such that the Bergman kernel $K(z, \xi, \ol{w}, \ol{\eta})$ of $D_1$ extends to be holomorphic on $V$ as a function of $(z, \xi)$ for each $(w, \eta) \in E.$ Now as $D_1$ and $D_2$ are bounded and circular,  by Theorem 2 of \cite{B}, $G$ extends holomorphically to a neighborhood of $\ol{D_1}.$

Write $M_i=\{(z, \xi) \in \Om_i \times \mathbb{C}: |\xi|^2=\big( \rho_i(z, \overline{z}) \big)^{\lambda_i} \}$ for $1 \leq i \leq 2.$ By the properness of $G$ and the strong pseudoconvexity of $M_1$ and $M_2$,  the extension of $G$ gives a local CR diffeomorphism from $M_1$ and $M_2.$ Since $M_i$ is locally CR equivalent to $\Sigma_i=\{(z, \xi) \in \Om_i \times \mathbb{C}: |\xi|^2=\rho_i(z, \overline{z}) \}$, we have $\Sigma_1$ and $\Sigma_2$ are locally CR diffeomorphic. By Theorem \ref{T1}, $\Omega_1$ and $\Omega_2$ are biholomorphic.

Fix $p \in M_1$. Pick small simply connected neighborhoods $U_1$ and $U_2$ of $p$ and $F(p) \in M_2$, respectively, such that $U_1\cap M_1$ is connected, $G(U_1) \subseteq U_2$, and there are well-defined branches of $\xi^{\frac{1}{\lambda_i}}$ in $U_i$ for $1 \leq i \leq 2.$ Moreover,  $\Phi_i(z, \xi):=(z, \xi^{\frac{1}{\lambda_i}}), 1 \leq i \leq 2,$ is biholomorphic on $U_i$. Note $\Phi_i$ maps $U_i \cap M_i$ to $\Sigma_i.$ Write $M'=\Phi_1(U_1\cap M_1)$, which is an open connected piece of $\Sigma_1,$ and define $F=\Phi_2 \circ G \circ \Phi_1^{-1}$ on $M'$. It is clear that $F$ is a nonconstant continuous CR map from $M'$ to $\Sigma_2.$ By part (3) of Theorem \ref{T1} and Remark \ref{rmkform}, $F$ extends to a norm-preserving holomorphic isomorphism between the line bundles  $(\Omega_1 \times \mathbb{C}, \rho_1^{-1})$  and $(\Omega_2 \times \mathbb{C}, \rho_2^{-1})$. Moreover,  
$F(z, \xi)=(\psi(z),\xi \phi(z))$, where $\psi$ and $\phi$ are as described in Remark \ref{rmkform}.  Note $G=\Phi_2^{-1} \circ F \circ \Phi_1$ in $U_1$.
By the form of $\Phi_1, \Phi_2$ and $F$, we see
$G=(\psi(z),(\xi^{\frac{1}{\lambda_1}}\phi(z))^{\lambda_2})$ in $U_1$ for some branches of $(\cdot)^{\frac{1}{\lambda_1}}$ and $(\cdot)^{\lambda_2}.$ Choose branches of $\xi^{\frac{\lambda_2}{\lambda_1}}$ and $\big(\phi(z)\big)^{\lambda_2}$ in $U_1.$ Note
$(\xi^{\frac{1}{\lambda_1}}\phi(z))^{\lambda_2}=e^{i\theta} \xi^{\frac{\lambda_2}{\lambda_1}} \big(\phi(z)\big)^{\lambda_2}$ in $U_1$ for some real number $\theta.$ But by assumption, $G=(\psi(z), e^{i\theta} \xi^{\frac{\lambda_2}{\lambda_1}} \big(\phi(z)\big)^{\lambda_2})$ extends to a holomorphic map in $D_1.$
This can occur if and only if $\lambda_2$ is an integer multiple of $\lambda_1.$ Then the conclusion of Corollary \ref{corpr} follows easily. \qed

\smallskip

We finally prove Theorem \ref{corcis}. 

\smallskip

{\bf Proof of Theorem \ref{corcis}.} Let $X$  be a compact Riemann surface of genus at least two, so that $X$ is covered by the 
unit disk and is equipped with the hyperbolic metric $g$. Let $(L, h)$ be its anti-canonical line bundle. Note $L$ is just the holomorphic tangent vector bundle of $X$, and $h$ equals the hyperbolic metric $g$. %Since $(X, g)$ locally is holomorphically isometric the the unit disk equipped with a hyperbolic metric, 
At any point $p$ of $X$, we can choose local coordinates $z$ of $p$ such that $z(p)=0$ and $g$ is given, up to a constant, by $(1-|z|^2)^{-2} dz \otimes d \ol{z}$ near $p$.
Fix any integer $n \geq 2$. We will construct a countable family of compact real analytic CR hypersurfaces of dimension $2n+1$ as desired in Theorem \ref{corcis}.

Let $Y$ be the product of $n$ copies of $X: Y=X \times \cdots  \times X.$ Denote by $\pi_i, 1 \leq i \leq n, $ the natural projection from $Y$ to the $i$th copy of $X$. Let $S$ be the set of $n-$tuples $(k_1, \cdots, k_{n}) \in \mathbb{Z}^{n}$ satisfying $1 =k_1 \leq k_2 \leq \cdots \leq k_{n}.$ For any $k=(k_1, \cdots, k_{n}) \in S,$ we define a Hermitian line bundle $(L_k, h_k)$ over $Y$ to be 
$(L_k, h_k):=\pi_1^*(L^{k_1}, h^{k_1}) \otimes \cdots \otimes \pi_n^*(L^{k_{n}}, h^{k_{n}}).$

Consider the circle bundle $C(k):=C(L_k, h_k)$ of the line bundle of $(L_k, h_k),$ which clearly has transverse symmetry. Moreover,
one can readily see that, at every $q \in C(k),$  $C(k)$ is locally CR equivalent to some open piece of the CR hypersurface in $\mathbb{C}^{n+1}$ given by  $|\xi|^2 \prod_{i=1}^{n} \frac{1}{(1-|z_i|^2)^{2k_i}}=1$, with all $z_i'$s close to $0$. Then via  the map $(z_1, \cdots, z_{n}, \xi) \rightarrow ( z_1, \cdots, z_{n}, \sqrt{\xi})$ for some appropriate branch of $\sqrt{\xi},$ we have $C(k)$ is further locally CR equivalent to some open piece of the following CR hypersurface in $\mathbb{C}^{n+1}:$
$$|\xi|^2=\prod_{i=1}^{n} (1-|z_i|^2)^{k_i},  |z_i|< 1~\text{for}~1 \leq i \leq n.$$
That is, $C(k)$ is locally CR equivalent to $\Sigma(k):=\Sigma(\Delta^{n}, \rho)$, where $\rho$ is the generic norm $\prod_{i=1}^{n} (1-|z_i|^2)^{k_i}$ of $\Delta^{n}$. It then follows from Proposition \ref{prpnflat} that $C(k)$ is obstruction flat and Bergman logarithmically flat. Moreover, $C(k)$ is locally homogeneous, as $\Sigma(k)$ is so. Furthermore, %the following statement hold.
%(a). Given any unit sphere $\mathbb{S}^{2N+1}, N \geq 2,$ by part (2) of Theorem \ref{T115}, the only smooth CR map from an open piece of $C(k)$ to $\mathbb{S}^{2N+1}$ is the constant map.
by Corollary \ref{corpdk}, with two choices of $k: \hat{k}=(\hat{k}_1, \cdots, \hat{k}_n)$ and $\wt{k}=(\wt{k}_1, \cdots, \wt{k}_n)$ in $S,$ there exists a nonconstant continuous CR map from an open piece of $\Sigma(\hat{k})$ to $\Sigma(\wt{k})$ (equivalently, $C(\hat{k})$ to $C(\wt{k})$), if and only if $\hat{k}=\wt{k}.$
Hence, by letting $k$ vary in $S,$ we obtain the desired countable family of CR hypersurfaces $C(k)$ in Theorem \ref{corcis}. \qed

%---------------------------------------------------------------

%---------------------------------------------------------------


\begin{thebibliography}{00}
\bibitem [AP] {AP} H. Ahn and J. Park, {\em The explicit forms and zeros of the Bergman kernel function for Hartogs type domains}, J. Funct. Anal. {\bf 262} (2012), no. 8, 3518-3547.
	
\bibitem [Al] {Al} H. Alexander, {\em Holomorphic mappings from the ball and polydisc}, Math. Ann.
{\bf 209} (1974), 249-256.


\bibitem [ALZ] {ALZ} C. Arezzo, A. Loi and F. Zuddas, {\em Szeg\"o kernel, regular quantizations and spherical CR-structures}, Math. Z. {\bf 275} (2013), 1207-1216.



%\bibitem [BEH1] {BEH1} M. S. Baouendi, P. Ebenfelt, X. Huang, {\em Super-rigidity for CR embeddings of real hypersurfaces into hyperquadrics}, Adv. Math. { 219}, no.5, (2008):1427-1445.

\bibitem [BEH] {BEH} M. Baouendi, P. Ebenfelt and X. Huang, {\em Holomorphic mappings between hyperquadrics with small signature difference}, Amer. J. Math. {\bf 133} (2011), no. 6, 1633-1661.

%\bibitem [BER] {BER} M. S. Baouendi, P. Ebenfelt, L. P. Rothschild, {\em Transversality of holomorphic mappings between real hypersurfaces in different dimensions}, Comm. Anal. Geom. 15(3) (2007) 589-611.

\bibitem [BH] {BH}  M. Baouendi and X. Huang, {\em Super-rigidity for holomorphic mappings between
hyperqadrics with positive signature}, J. Differential Geom. {\bf 69} (2005), no. 2, 379-398.

\bibitem[B]{B} S. Bell, {\em Proper holomorphic mappings between circular domains}, Comment. Math. Helv.
{\bf 57} (1982), 532–538.

%\bibitem[B2]{B2} S. Bell, {\em Proper holomorphic mappings that must be rational}, Trans. Amer. Math. Soc. 284 (1) (1984) 425-429.




\bibitem [BT1] {BT1} E. Bi and Z. Tu, {\em Remarks on the canonical metrics on the Cartan--Hartogs domains}, Comptes Rendus Mathematique {\bf 355} (2017), no. 7, 760-768.

\bibitem [BT2] {BT2} E. Bi and Z. Tu, {\em Rigidity of proper holomorphic mappings between generalized Fock--Bargmann--Hartogs domains}, Pacific J. Math. {\bf 297} (2018), no. 2, 277-297.


\bibitem [BS] {BS} L. Boutet de Monvel and J. Sj\"ostrand, {\em Sur la singularité des noyaux de Bergman et de Szeg\"o},  In
Journ\'ees: \'Equations aux D\'eriv\'ees Partielles de Rennes (1975), Ast\'erisque, {\bf 34-35} (1976), 123-164.

\bibitem[Br] {Br} R. Bryant, Bochner-Kähler metrics, J. Amer. Math. Soc., {\bf 14} (2001),  623-715.

\bibitem [CY] {CY} S. Cheng and S. T. Yau, {\em On the existence of a complete K\"ahler metric on noncompact complex
manifolds and the regularity of Fefferman's equation}, Comm. Pure Appl. Math., {\bf 33} (1980), no. 4, 507-544.


%\bibitem [CE1]{CE1} S. Curry and P. Ebenfelt, {\em Bounded strictly pseudoconvex domains in $\mathbb{C}^2$ with obstruction at boundary II}, Adv. Math., {\bf 352} (2019), 611-631.

\bibitem [CE]{CE} S. Curry and P. Ebenfelt, {\em Bounded strictly pseudoconvex domains in $\mathbb{C}^2$ with obstruction at boundary},
Amer. J. Math., {\bf 143} (2021), no. 1, 265-306.





%\bibitem [CM]{CM} S. S. Chern, J. Moser, {\em Real hypersurfaces in complex manifolds}, Acta Math.  133 (1974) 219-271.

%\bibitem[D2]{D2} J. P. D'Angelo, {\em Several complex variables and the geometry of real hypersurfaces}, Studies in Advanced Mathematics. CRC Press, Boca Raton, FL, 1993. xiv+272 pp. ISBN: 0-8493-8272-6.

%\bibitem[DX]{DX} J. P. D'Angelo and M. Xiao, {\em Symmetries in CR complexity theory}, Adv. Math.  {\bf 313}  (2017), 590-627.

\bibitem[E]{E} P. Ebenfelt, {\em The log term in the Bergman and Szeg\"o kernels in strictly pseudoconvex domains in $\mathbb{C}^2$}, Doc.
Math., {\bf 23} (2018) 1659-1676.

\bibitem [En] {En} M. Engli\v{s}, {\em Pseudolocal estimates for $\Ol{\partial}$  on general pseudoconvex domains}, Indiana Univ. Math. J. {\bf 50} (2001), 1593-1607.

\bibitem[EXX1]{EXX1} P. Ebenfelt, H. Xu and M. Xiao, {\em Algebraicity of the Bergman kernel}, submitted, arXiv: 2007.00234.

\bibitem[EXX2]{EXX2} P. Ebenfelt, H. Xu and M. Xiao, {\em K\"ahler-Einstein metrics and obstruction flatness of circle bundles}, submitted, arXiv:2208.13367.


\bibitem [EZ]{EZ} M. Engli\v{s} and G. Zhang, {\em Ramadanov conjecture and line bundles over compact Hermitian symmetric
spaces}, Math. Z., {\bf 264} (2010), no. 4, 901-912.

%\bibitem[FHX] {FHX} H. Fang, X. Huang and M. Xiao, {\em Volume-preserving maps between Hermitian symmetric spaces of compact type}, {Adv. Math.} {\bf 360} (2020), 106885.

\bibitem[Fe1]{Fe1} C. Fefferman, {\em The Bergman kernel and biholomorphic mappings of pseudoconvex domains}, Invent. Math.,
{\bf 26} (1974), 1-65.

\bibitem[Fe2]{Fe2} C. Fefferman, {\em Monge-Amp\`ere equations, the Bergman kernel, and geometry of pseudoconvex domains}, Ann.
of Math.,
{\bf 103} (1976), 395-416.

%\bibitem[FT]{FT} Z. Feng and  Z. Tu, {\em On canonical metrics on Cartan-Hartogs domains}, Math. Z. {\bf 278} (2014), 301–320. 

\bibitem [Fr1] {Fr1} F. Forstneri\v{c}, {\em Proper holomorphic maps between balls}, Duke Math. J. {\bf 53} (1986), 427-440.

\bibitem [Fr2] {Fr2} F. Forstneri\v{c}, {\em Extending proper holomorphic mappings of positive codimension}, Invent. Math. {\bf 95} (1989), 31-62.


\bibitem [Gr] {Gr} C. Graham, Higher asymptotics of the complex Monge-Amp\`ere equation. Compositio Math., {\bf 64} (1987), no. 2, 133-
155.




%\bibitem [EHZ1] {EHZ1} P. Ebenfelt, X. Huang, D. Zaitsev, {\em Rigidity of CR-immersions into spheres}, Comm. Anal. Geom., 12, no.3, (2004):631-670.

%\bibitem [EHZ2] {EHZ2} P. Ebenfelt, X. Huang, D. Zaitsev,  {\em The equivalence problem and rigidity for hypersurfaces embedded into hyperquadrics}, Amer. J. Math., 127, no.1, (2005):169-191.

%\bibitem [HN] {HN} G. M. Henkin and R. Novikov, {\em Proper mappings of classical domains}, in Linear and Complex Analysis Problem Book, Lecture Notes in Math. Vol. 1043, Springer, Berlin  (1984) 625-627.


\bibitem [He]{He} S. Helgason, Differential Geometry and Symmetric Spaces. Pure 
and Applied Mathematics, vol. 12. Academic Press, New York 
(1962).

\bibitem [Hu1] {Hu1} X. Huang, {\em On the mapping problem for algebraic real hypersurfaces  in the complex spaces of different dimensions},
{Annales de l'institut Fourier}, {\bf 44} (1994), 433-463.


 \bibitem[Hu2]{Hu2} X. Huang, {\em On a linearity problem of proper
holomorphic maps between balls in complex spaces of different
dimensions}, {J. Differential Geom.}, {\bf 51} (1999), 13-33.

%\bibitem[Hu2]{Hu2} X. Huang, {\em On a semi-rigidity property for holomorphic maps}, Asian J. Math., 7 (4) (2003) 463--492. %(A special issue dedicated to Y. T. Siu on the ocassion of his 60th birthday.)

\bibitem [HJ1]{HJ1} X. Huang and S. Ji, {\em Global holomorphic extension of a local map and a Riemann mapping theorem
for algebraic domains}, Math. Res. Lett., {\bf 5} (1998), no. 2, 247-260.


\bibitem [HJ2]{HJ2} X. Huang and S. Ji, {\em Mapping $\mathbb{B}^n$ into $\mathbb{B}^{2n-1}$}, Invent. Math., {\bf 145}  (2001), 219-250.

\bibitem [HLX]{HLX} X. Huang, X. Li and M. Xiao, {\em Nonembeddability into a fixed sphere for a family of compact real algebraic hypersurfaces},
Int. Math. Res. Not., {\bf 16} (2015), 7382-7393.

\bibitem [HL]{HL} X. Huang and X. Li. {\em Bergman-Einstein metric on a Stein space with a strongly pseudoconvex
boundary}, arXiv:2008.03645, 2020.

%\bibitem [HX]{HX} X. Huang and M. Xiao, {\em Chern-Moser-Weyl tensor and embeddings into hyperquadrics},  Harmonic Analysis, partial differential equations, 79-95, Appl. Numer. Harmon. Anal., Birkh\"auser/Springer, Cham, 2017.


\bibitem[HZ]{HZ} X. Huang and D. Zaitsev, {\em Non-embeddable real algebraic hypersurface}, Math.Z., {\bf 275} (2013), 657-671.


%\bibitem [HLTX] {HLTX} X. Huang, J. Lu, X. Tang, M. Xiao, {\em Boundary characterization of holomorphic isometric embeddings between indefinite hyperbolic spaces}, Adv. Math.,  374 (2020) 107388.

%\bibitem [IK]{IK} H. Ishi and C. Kai, {\em The representative domain of a homogeneous bounded domain}, Kyushu J. Math. {\bf 64} (2009), no. 1, 35-47. 


%\bibitem [KNY] {KNY} H. Kim, V.T. Ninh and A. Yamamori, {\em The automorphism group of a certain unbounded nonhyperbolic domain}, J. Math. Anal. Appl. {\bf 409} (2014), no. 2, 637-642.

%\bibitem [K] {Ki1} S. Kim, {\em Proper holomorphic maps between bounded symmetric domains}, In complex analysis and geometry, (2015) 207-219.

%\bibitem [K2] {Ki2} S. Kim, {\em Holomorphic maps between closed $SU(l+1,m)$ orbits in Grassmannian}, preprint, 2019.


\bibitem[KZ1]{KZ1} S. Kim and D. Zaitsev,
{\em Rigidity of CR maps between Shilov boundaries of bounded symmetric
domains}, Invent. Math., {\bf 193} (2013), no. 2, 409-437.

\bibitem[KZ2]{KZ2} S. Kim and D. Zaitsev, {\em Rigidity of proper holomorphic maps between bounded
symmetric domains}, Math. Ann., {\bf 362} (2015), 639-677.

\bibitem [LM] {LM} J. Lee and R. Melrose, {\em Boundary behaviour of the complex Monge-Amp\`ere equation}. Acta Math.,
{\bf 148} (1982), 159-192.

\bibitem[Lo]{Lo} O. Loos, Bounded symmetric domains and Jordan pairs, Math. Lectures, University of California, Irvine, 1977.

\bibitem[MZ] {MZ} N. Mir and D. Zaitsev, {\em Unique jet determination and extension of germs of CR maps into spheres}, Trans. Amer. Math. Soc.,
{\bf 374} (2021), no. 3, 2149--2166.



%\bibitem[Mo1]{Mo1} N. Mok,
%{\em Extension of germs of holomorphic isometriesup to normalizing
%constants with respect to the Bergman metric}, {J. Eur. Math.
%Soc.} {\bf 14}, no. 5, 1617-1656 (2012).

\bibitem[M1]{M1}N. Mok, {Metric rigidity theorems on Hermitian locally symmetric manifolds}, Series in Pure Math. 6. World Scientific Publishing Co., Inc., Teaneck, NJ, 1989. xiv+278 pp.


\bibitem[M2]{M2} N. Mok, {\em Extension of germs of holomorphic isometries up to normalizing constants with respect to the Bergman metric}, J. Eur. Math. Soc. {\bf 14} (2012), no. 5, 1617-1656 .


\bibitem[M3]{M3} N. Mok, {Holomorphic isometries of the complex unit ball into irreducible bounded symmetric domains}, {Proc. Amer. Math. Soc.} {\bf 144} (2016), no. 10,  4515-4525.


\bibitem[MN] {MN} N. Mok and S. Ng, {\em Germs of measure-preserving holomorphic maps from bounded symmetric domains to their Cartesian products},
Crelle {\bf 669}  (2012), 47-73.


%\bibitem[M1]{M1} N. Mok, {\em Nonexistence of proper holomorphic maps between certain classical bounded symmetric domains}, Chinese Annals of Mathematics, Series B,  29 (2008) 135-146.

%\bibitem[MT]{MT} N. Mok, I-Hsun Tsai, {\em Rigidity of convex realizations of irreducible bounded symmetric domains of rank$\geq 2$}, J. Reine Angew. Math., 431 (1992) 91-122.


\bibitem [MZe]{MZe} R. Mossa and M. Zedda, {\em A Cartan--Hartogs version of the polydisk theorem}, Geometriae Dedicata {\bf 216} (2022), no. 51,
https://doi.org/10.1007/s10711-022-00709-3.

%\bibitem [Ng1]{Ng1} S. Ng, {\em Proper holomorphic mappings on flag domains of $SU (p, q)-$type on projective spaces}, Michigan Math. J., 62 (2013) 769-777.

%\bibitem [Ng2]{Ng2}  S. Ng, {\em Holomorphic Double Fibration and  the mapping problems of Classical Domains}, International Mathematics Research Notices,  2 (2015) 291-324.

\bibitem [PT]{PT} S. Pinchuk and S. Tsyganov, {\em The smoothness of CR-mappings between strictly pseudoconvex
hypersurfaces}, Mathematics of the USSR-Izvestiya, {\bf 35} (1990), no. 2, 457-467.


\bibitem[Po]{Po} H. Poincar\'e, {\em Les fonctions analytiques de deux variables et la repr\'esentation conforme}, Rend. Circ. Mat. Palermo  {\bf 23} (1907), no. 2, 185-220.


%\bibitem[Sh]{Sh} R. Shafikov, {\em Analytic continuation of holomorphic correspondences and equivalence of domains in $\mathbb{C}^n$}, Invent. Math. {\bf 152} (2003), 665-682.


\bibitem [Tk] {Tk} Y. Takeuchi, {\em Ambient constructions for Sasakian $\eta$-Einstein manifolds}, Adv. Math., {\bf 328} (2018), 82-111.


\bibitem[Ta]{Ta} N. Tanaka, {\em On the pseudo-conformal geometry of hypersurfaces of the space of $n$ complex
variables}, J. Math. Soc. Japan {\bf 14} (1962), 397-429.


\bibitem[TW]{TW} Z. Tu and L. Wang, {\em Rigidity of proper holomorphic mappings between equidimensional Hua domains}, Math. Ann. {\bf 363} (2015), 1-34.



\bibitem [TK1]{TK1} A. Tumanov and G. Khenkin, {\em Local characterization of analytic automorphisms of classical domains} (Russian), Dokl. Akad. Nauk SSSR {\bf 267} (1982), no. 4, 796-799.
%; English translation: Math. Notes 32 (1982), 849-852.

\bibitem [TK2]{TK2} A. Tumanov and G. Khenkin, {\em Local characterization of holomorphic automorphisms of Siegel domains}(Russian), Funktsional. Anal. i Prilozhen {\bf 17} (1983), 49-61.
%;  English translation: Functional Anal. Appl. 17 17 (1983), 285-294.

\bibitem [WYZR] {WYZR} A. Wang, W. Yin, L. Zhang and G. Roos, {\em The \ke metric for some Hartogs domains over
bounded symmetric domains}, Science in China Series A: Mathematics {\bf 49} (2006), no. 9, 1175-1210.

%\bibitem [YLR] {YLR} W. Yin, K. Lu, G. Roos, New classes of domains with explicit Bergman kernel, Science in China Ser. A Mathematics 47 (3)  (2004) 352-371.

\bibitem [W1] {W0} S. Webster, On the pseudo-conformal geometry of a Kähler manifold, Math. Z., {\bf 157} (1977), 265-270.

\bibitem [W2]{W1} S. Webster, {\em On the mapping problem for algebraic real hypersurfaces}, Invent.
math. {\bf 43} (1977), 53-68.
\bibitem [W3]{W2} S. Webster, {\em The rigidity of C-R hypersurfaces in a sphere}, Indiana Univ. Math. J. {\bf 28} (1979), no. 3, 405-416.

\bibitem[Za]{Za}  D. Zaitsev, {\em Algebraicity of local holomorphisms between real-algebraic submanifolds of complex spaces},
{Acta Math.}, {\bf 183} (1999), 273-305.

\bibitem [Ze] {Ze} M. Zedda, {\em Canonical metrics on Cartan--Hartogs domains}, International Journal of Geometric Methods in Modern Physics {\bf 9} (2012), no. 1, 1250011.


\bibitem[ZWH] {ZWH} J. Zhao, A. Wang and Y. Hao,  {\em On the holomorphic automorphism group of the Bergman--Hartogs domain}, Int. J. Math. {\bf 26}  (2015), no. 08, 1550056.



%\bibitem [Um]{Um} Umehara, M.: Diastasis and real analytic functions on complex manifolds, J. Math. Soc. Japan. 40 (1988),
%no. 3, 519-539.
%
%

%\bibitem [XY1]{XY1} M. Xiao and Y.  Yuan, Holomorphic maps from the complex unit ball to type IV classical do-
%mains, Journal de Math\'ematiques Pures et Appliqu\'ees, doi.org/10.1016/j.matpur.2019.05.009
%
%\bibitem [XY2]{XY2} M. Xiao and Y. Yuan, Complexity of holomorphic maps from the complex unit ball to classical
%domains, Asian J. Math., 22 (2018), 729-760.
%
%\bibitem [YZ]{YZ} Y. Yuan and Y. Zhang,  CR Submanifolds with vanishing second fundamental forms  Geometriae Dedicata 183 (2016), 169-180.

%\smallskip


%\noindent M. Xiao, Department of Mathematics, University of
%California, San Diego, 9500 Gilman Drive, La Jolla, CA 92093, USA.
%(m3xiao$@$ucsd.edu)


\end{thebibliography}
\end{document}